\documentclass[a4paper, reqno]{amsart}
\usepackage[margin=1.2in]{geometry}
\allowdisplaybreaks

\title[]{On Virasoro-type reductions and inverse Hamiltonian reductions for $\W$-algebras and $\mathcal{W}_\infty$-algebras}

\author[Justine Fasquel]{Justine Fasquel}
\address[J.F.]{School of Mathematics and Statistics, University of Melbourne, Parkville, Australia, 3010}
\email{justine.fasquel@unimelb.edu.au}

\author[Vladimir Kovalchuk]{Vladimir Kovalchuk}
\address[V.K.]{Department of Mathematics, University of Denver}
\email{vladimir.kovalchuk@du.edu}

\author[Shigenori Nakatsuka]{Shigenori Nakatsuka}
\address[S.N.]{Department Mathematik, FAU Erlangen–Nürnberg, Cauerstraße 11, 91058, Erlangen, Germany}
\email{shigenori.nakatsuka@fau.de}

\usepackage{array}
\usepackage{latexsym}
\usepackage{amsmath}
\usepackage{amsfonts}
\usepackage{amssymb,nccmath}
\usepackage{mathtools}
\usepackage{amsxtra}
\usepackage{mathrsfs}
\usepackage{eucal}
\usepackage[all]{xy}
\usepackage{color}
\usepackage{braket}
\usepackage{graphics, setspace}
\usepackage{etoolbox, textcomp}
\usepackage{comment}
\usepackage{changepage}
\usepackage{xcolor}
\definecolor{rouge}{rgb}{0.85,0.1,.4}
\definecolor{bleu}{rgb}{0.1,0.2,0.9}
\definecolor{violet}{rgb}{0.7,0,0.8}
\usepackage[colorlinks=true,linkcolor=bleu,urlcolor=violet,citecolor=rouge]{hyperref}
\usepackage{cleveref}
\usepackage{lscape}
\usepackage{caption}
\usepackage{subcaption}
\usepackage{enumitem}
\usepackage{graphicx}
\usepackage{arydshln}
\usepackage{tikz} \usetikzlibrary{calc} \tikzset{>=latex} \usetikzlibrary{backgrounds} \usetikzlibrary{shapes.geometric}
\usepackage{tikz-cd}
\usepackage[colorinlistoftodos]{todonotes}
\usepackage{bbm}
\usepackage{blindtext}

\usepackage[backend=biber,safeinputenc,sorting=nyt,maxbibnames=10]{biblatex}
\usepackage[aligntableaux=center]{ytableau}
\ytableausetup{mathmode,boxsize=0.4em}

\usepackage{tikz}			
\usetikzlibrary{matrix}
\usetikzlibrary{automata, arrows.meta, positioning,cd}
\usepackage{dynkin-diagrams}
\usetikzlibrary{cd}			

\newtheorem{definition}{Definition}[section]
\newtheorem{proposition}[definition]{Proposition}
\newtheorem{theorem}[definition]{Theorem}
\newtheorem{ConjLetter}{Conjecture}

\newtheorem{ThmLetter}{Theorem}
\newtheorem{corollary}[definition]{Corollary}

\newtheorem{conjecture}[definition]{Conjecture}
\theoremstyle{remark}
\newtheorem{remark}[definition]{Remark}

\newcommand{\mf}{\mathfrak}

\numberwithin{equation}{section}

\newcommand{\Z}{\mathbb{Z}}     
\newcommand{\Q}{\mathbb{Q}}     
\newcommand{\C}{\mathbb{C}}     
\newcommand{\R}{\mathbb{R}}      

\newcommand{\D}{\mathcal{D}}
\newcommand{\W}{\mathcal{W}}    

\newcommand{\V}{\mathcal{V}}    

\newcommand{\HH}{\mathrm{H}}    
\newcommand{\OO}{\mathbb{O}}    
\newcommand{\Nil}{\mathcal{N}}
\newcommand{\tW}{\texorpdfstring{$\W$}{W}}
\newcommand{\tA}{\texorpdfstring{$A$}{A}}
\newcommand{\Lie}{\mathrm{Lie}}
\newcommand{\ff}[1]{\beta\gamma^{\{#1\}}}
\newcommand{\bg}{\beta\gamma}
\newcommand{\affWak}[1]{\mathbb{W}^{\kk}_{{#1}}}
\newcommand{\fockIO}[1]{\mathrm{e}^{-\frac{1}{\kk+\mathrm{h}^\vee}\alpha_{#1}}}
\newcommand{\dz}{\mathrm{d}z}
\newcommand{\oo}[1]{\OO_{[#1]}}
\newcommand{\ad}{\mathrm{ad}}
\newcommand{\rs}[2]{{\scalebox{#1}{$#2$}}}

\newcommand{\chfermion}{\mathcal{F}_\mathrm{ch}}

\newcommand{\Aut}{\operatorname{Aut}}
\newcommand{\End}{\operatorname{End}}

\renewcommand{\ker}{\operatorname{Ker}}
\newcommand{\im}{\operatorname{Im}}

\renewcommand{\End}{\mathrm{End}}

\newcommand{\kk}{\mathsf{k}}   
\newcommand{\ee}{\mathsf{e}}  
\newcommand{\wun}{\mathbbm{1}}   
\newcommand{\dd}{\mathrm{d}}
\newcommand{\hv}{\mathsf{h}^\vee}   
\newcommand{\tr}{\mathrm{tr}}

\newcommand{\g}{\mathfrak{g}}   
\newcommand{\h}{\mathfrak{h}}   
\newcommand{\n}{\mathfrak{n}}

\newcommand{\sll}{\mathfrak{sl}}    
\newcommand{\spp}{\mathfrak{sp}}    
\newcommand{\so}{\mathfrak{so}}    
\newcommand{\Mat}{\mathrm{Mat}}
\newcommand{\E}{\mathrm{E}}
\newcommand{\e}[1]{\mathrm{e}^{#1}}

\newcommand{\I}{\mathrm{I}}

\newcommand{\Wsp}[0]{\mathcal{W}^{\mathfrak{sp}}_{\infty}}

\newcommand{\Lev}{L^{\mathrm{ev}}}
\newcommand{\gWev}[1]{W^{\mathrm{ev}}_{4,#1}}

\newcommand{\heis}{\pi_\mathfrak{h}^{\kk+\hv}}
\newcommand{\Fock}[1]{\pi_{\mathfrak{h},{#1}}^{\kk+\hv}}

\newcommand{\indic}[1]{\delta_{#1}}     
\renewcommand{\epsilon}{\varepsilon}
\newcommand{\longstick}[2]{\left. {#1} \right|_{#2}}

\DeclareMathOperator{\rk}{rk}

\newcommand{\intNOz}[1]{\int Y({#1},z) \ \dd z}  
\newcommand{\scr}[2]{\intNOz{{#2}\ee^{-\frac{1}{\kk+\hv}\alpha_{#1}}}}  

\newcommand{\KL}{\mathrm{KL}}
\newcommand{\weyl}{\mathbb{V}}

\renewcommand{\mod}{\operatorname{-Mod}}

\newcommand{\fd}[1]{d_{\,\ydiagram{#1}\,}}      

\newcommand{\fHdeg}[2]{\mathrm{H}^{#2}_{\,\ydiagram{#1}\,}}

\newcolumntype{M}[1]{>{\centering\arraybackslash}m{#1}}

\newcommand{\numtableaux}[1]{
\ytableausetup{boxsize=1.15em,aligntableaux=bottom}
\begin{ytableau}
    #1
\end{ytableau}
\ytableausetup{boxsize=0.4em}}
\newcommand\doi[2]{\href{http://dx.doi.org/#1}{#2}}

\usepackage[normalem]{ulem} 

\begin{document}
\begin{abstract}
In this article, the Virasoro-type reduction and the corresponding inverse reductions are established for $\W$-algebras associated with classical Lie type and nilpotent orbits of height two.
Moreover, these results are lifted to the universal objects by analyzing the Virasoro-type reduction of the vertex algebra $\Wsp$.
\end{abstract}

\maketitle

\section{Introduction}

Let $\g$ be a simple Lie algebra over the complex numbers $\C$. 
For $\kk \in \C$, one associates the affine vertex algebra $\V^\kk(\g)$, whose representation theory captures the smooth representations over the affine Lie algebra $\widehat{\g}$ at level $\kk$.
The $\W$-algebra $\W^\kk(\g,\OO)$ at level $\kk$ is defined \cite{FF90, KRW03} by applying to $\V^\kk(\g)$ a quantum Hamiltonian reduction corresponding to the nilpotent orbit $\OO\subset\g$.
Historically, $\W$-algebras are central players in conformal field theories in physics \cite{KW04}.
More recently, they have also appeared as important objects in several areas of mathematics, 
in particular in the study of the Whittaker models over $\widehat{\g}$ \cite{Raskin20}.

Although each $\W$-algebra $\W^\kk(\g,\OO)$ is associated with a nilpotent orbit $\OO$,
it has been noticed recently that they are indeed related in a way reflecting the closure relations of the nilpotent orbits with respect to the Zariski topology \cite{CFLN, FFFN, GJ25}. 
Precisely,
given two nilpotent orbits $\OO_1, \OO_2$ satisfying the closure relation ${\OO}_1\subset \overline{\OO}_2$, the existence of a (generalized) quantum Hamiltonian reduction $\HH_{\OO_2 \uparrow {\OO}_1}$ so that 
\begin{align}\label{statement of reduction by stages}
    \HH^n_{\OO_2 \uparrow {\OO}_1}\left(\W^\kk(\g,\OO_1) \right)\simeq \delta_{n,0} \W^\kk(\g,\OO_2)
\end{align}
holds is expected.
This can be regarded as a vertex-algebraic counterpart of similar known results in Whittaker models for the finite-dimensional Lie algebras (i.e. finite $\W$-algebras) \cite{GJ23, Morgan14} or for the reductive algebraic groups over local fields \cite{GGS17}. 

The finite $\W$-algebras $\mathscr{U}(\g,\OO)$ \cite{Kostant78, Pre02} are associative algebras obtained from the enveloping algebras $\mathscr{U}(\g)$ by the quantum Hamiltonian reductions associated with $\OO$. 
They are also obtained as Zhu's algebras of $\W^\kk(\g,\OO)$ \cite{Ara07, DK06}.  
In this finite setting, partial reductions of the form \eqref{statement of reduction by stages} are better understood.
Indeed, Losev's decomposition theorem \cite{Losev10} asserts that the $\hbar$-adic version of $\mathscr{U}(\g)$, say $\mathscr{U}_\hbar(\g)$, is isomorphic to the tensor product of $\mathscr{U}_\hbar(\g,\OO)$ and a Weyl algebra $\mathcal{A}_{{\OO}\, \hbar}$ after appropriate completions;
so it is natural to expect a similar statement for $\mathscr{U}_\hbar(\g,\OO_1)$ with the tensor product of $\mathscr{U}_\hbar(\g,\OO_2)$ and some Weyl algebra, namely,
\begin{align}\label{statement of finite reduction by stages}
    \mathscr{U}_\hbar(\g,\OO_1)^{\wedge} \simeq \mathscr{U}_\hbar(\g,\OO_2)^{\wedge}\  \widehat{\otimes}_{\C[\![\hbar]\!]} \ \mathcal{A}_{\OO_2 \downarrow {\OO}_1\, \hbar}{}^{\wedge}.
\end{align}
Taking the quantum Hamiltonian reduction, this leads to a finite analogue of \eqref{statement of reduction by stages}.

Conversely, $\W^\kk(\g,\OO_1)$ should be recovered from $\W^\kk(\g,\OO_2)$ by tensoring with a certain vertex-algebraic version of Weyl algebras, namely the algebras of chiral differential operators \cite{BD04, MSV99}. This is realized as inverse Hamiltonian reductions initiated in \cite{Ad19, Sem94}. 
They have recently started attracting intensive attention for the search for the braided tensor category of weight modules over $\W$-algebras e.g.\ \cite{Ad19, AKR23, ACK23, FRR, FR22} and for vertex algebras arising from divisors in Calabi--Yau threefolds \cite{Butson23}.
The inverse Hamiltonian reductions are embeddings of vertex algebras of the form
\begin{align}\label{statement of iHR}
    \W^\kk(\g,\OO_1) \hookrightarrow  \W^\kk(\g,\OO_2)\otimes \mathscr{D}^{\mathrm{ch}}_X
\end{align}
for some affine space $X$. Usually, the space $X$ takes the form $\mathbb{A}^n\times (\mathbb{A}^{\times})^m$ and, accordingly, $\mathscr{D}^{\mathrm{ch}}_X$ is a tensor product of free field algebras, e.g. $\beta\gamma$-systems $\beta\gamma$ and half-lattice vertex algebras $\Pi$ (obtained as a localization of $\beta\gamma$).

In the present paper, we consider cases where the reduction \eqref{statement of reduction by stages} and its inverse \eqref{statement of iHR} are similar respectively to the reduction of $\V^\kk(\sll_{2})$ resulting into the Virasoro vertex algebra $\W^\kk(\sll_{2},\OO_{[2]})$ and the historical example of inverse Hamiltonian reduction \cite{Sem94}.
The $\W$-algebras which admit such a reduction contain a generator $G^+$ which satisfies the operator product expansion (OPE) 
\begin{align*}
    G^+(z)G^+(w)\sim 0.
\end{align*}
This means that $G^+$ can be regarded as an analog of the upper nilpotent element in $\sll_2$, corresponding to the case $\V^\kk(\sll_2)$.
Then one can apply the Virasoro-type quantum Hamiltonian reduction $\fHdeg{2}{}$. 
For classical Lie types, such cases appear prominently for $\W$-algebras $\W^k(\g,\OO)$ with $\OO$ corresponding to partitions of ``height two". Since nilpotent orbits in classical Lie algebras are parametrized by certain partitions, ``height two" refers to the length of the partition.
In type $A$ we consider the following family of $\W$-algebras
\begin{align}\label{rectangular of height two type A}
    &\W^\kk(\sll_{N+M},\OO_{[N,M]}), \quad N\geq M,
\end{align}
and in types $BCD$
\begin{align}\label{rectangular of height two}
\begin{split}
    \W^\kk(\so_{2N+1},\OO_{[N^2,1]}),\quad \W^\kk(\spp_{2N},\OO_{[N^2]}),\quad \W^\kk(\so_{2N},\OO_{[N^2]}).
\end{split}
\end{align}
The resulting vertex algebra $\fHdeg{2}{0}\left( \W^k(\g,\OO) \right)$ turns out to be the $\W$-algebra associated with the smallest nilpotent orbit $\widehat{\OO}$ whose boundary contains $\OO$; see Table \ref{intro: Nilpotent orbits} below for the explicit form.
Our first result is the following.
\begin{ThmLetter} \label{intro: theorem A}
Let $\g$, $\OO$ and $\widehat{\OO}$ be as in Table \ref{intro: Nilpotent orbits} and $\kk$ a generic level.
\begin{enumerate}[wide, labelindent=0pt]
\item There is an isomorphism of vertex algebras
\begin{align*}
    \fHdeg{2}{0}(\W^\kk(\g,\OO)) \simeq \W^\kk(\g,\widehat{\OO}).
\end{align*}
\item There is an embedding of vertex algebras
\begin{align*}
    \W^\kk(\g,\OO) \hookrightarrow \Pi \otimes \W^\kk(\g,\widehat{\OO}).
\end{align*}
\end{enumerate}
\end{ThmLetter}

The proof uses the Wakimoto realization of $\W$-algebras \cite{Gen20} at generic levels, which characterizes them inside free field algebras. 
Such a realization is generalized to a certain class of modules over $\W$-algebras (Corollary \ref{Free field realization of W-modules}). 
More precisely, 
we consider the $\W^\kk(\g,\OO)$-modules obtained from $\V^\kk(\g)$-modules in the Kazhdan--Lusztig category $\KL^\kk(\g)$ by applying the quantum Hamiltonian reductions $\HH_\OO$. Our second result is the following.
\begin{ThmLetter}\label{intro: thm B}
For $\kk$ a generic level, there is an isomorphism of $\W^\kk(\g,\widehat{\OO})$-modules
\begin{align*}
    \fHdeg{2}{0}(\HH_\OO(M)) \simeq \HH_{\widehat{\OO}}(M).
\end{align*}
where $M$ is a $\V^\kk(\g)$-module in $\KL^\kk(\g)$. Thus, the functors
\begin{align*}
    \fHdeg{2}{0}\circ\HH_\OO, \HH_{\widehat{\OO}}\colon \KL^\kk(\g) \rightarrow \W^\kk(\g,\widehat{\OO})\mod
\end{align*}
are naturally isomorphic.
\end{ThmLetter}

\begin{table}[h!]
	\centering
	\renewcommand{\arraystretch}{1.3}
	\begin{tabular}{|c||c|c|c|c|}
		\hline
		$(\g,\OO)$ & $(\sll_{N+M},\oo{N,M})$ & $(\so_{2N+1},\oo{N^2,1})$ & $(\spp_{2N},\oo{N^2})$ & $(\so_{2N},\oo{N^2})$ \\
		\hline
             $\widehat{\OO}$ & $\oo{N+1,M-1}{}^*$ & $\begin{array}{c} \oo{N+1,N-1,1}{}^*\\ \oo{N+2,N-2,1} \end{array}$ & $\begin{array}{c} \oo{N+2,N-2}\\ \oo{N+1,N-1}{}^* \end{array}$ & $\begin{array}{c} \oo{N+1,N-1}{}^*\\ \oo{N+2,N-2} \end{array}$  \\ \hline
	\end{tabular}
 \captionsetup{font=small }
	\caption{Nilpotent orbits $\widehat{\OO}$. The upper and lower lows correspond to the cases $N$ being even and odd, respectively.}
    \label{intro: Nilpotent orbits}
\end{table}

The result on the inverse Hamiltonian reductions in Theorem \ref{intro: theorem A} holds for all levels based on a {continuity} argument deforming the levels (see, for example \cite{FFFN,FN, Fehily24}), which also implies the homomorphisms
\begin{align}\label{eq:reduciton}
    \fHdeg{2}{0}(\W^\kk(\g,\OO)) \rightarrow \W^\kk(\g,\widehat{\OO})
\end{align}
for all levels.
This argument can be completed to prove that \eqref{eq:reduciton} is an isomorphism for all levels and that the cohomology vanishing ($\fHdeg{2}{\neq0}(\W^\kk(\g,\OO)) =0$) holds in the cases denoted by ${}^*$ in Table \ref{intro: Nilpotent orbits}. 
It is natural to expect an isomorphism for {all levels} and that the cohomology vanishing holds in the remaining cases, too.
This problem seems to be reduced to a problem in the finite/classical setting, namely, on group actions on the Slodowy slices, which is of independent interest. See recent works \cite{GJ23, GJ25} for this approach to other classes of nilpotent orbits.

The generating types of $\W$-algebras appearing in Table \ref{intro: Nilpotent orbits} suggest that these algebras arise as 1-parameter quotients of universal 2-parameter $\W_\infty$-algebras, and that reductions (\ref{eq:reduciton}) lift to these 2-parameter $\W_\infty$-algebras (see generating types (\ref{eq:gen types}) and (\ref{eq:gen types reduced}) below).
One such universal family has already appeared in the literature as $\Wsp$ \cite{CKL24} defined over a commutative ring $R(\supset \C[c,\kk])$. It has generating type $\W(1^3,2,3^3,4,\dots)$ and gives rise to 
$$\W^\kk(\mathfrak{sp}_{4n+2},\oo{2n+1,2n+1}),\quad \W^\kk(\mathfrak{so}_{4n},\oo{2n,2n})$$
as 1-parameter quotients. This implies a universal version of the statements in Theorem \ref{intro: theorem A}. 
Indeed, this is realized by analyzing the Virasoro-type reduction for $\Wsp$ (which we denote by $\Wsp(c,\kk)$ to stress the $(c,\kk)$-dependence in this paper) itself as follows. 

\begin{ThmLetter} \label{intro: theorem C}
\phantom{x}
\begin{enumerate}[wide, labelindent=0pt]
    \item The vertex algebra $\fHdeg{2}{0}(\Wsp(c,\kk))$ is a simple freely generated vertex algebra, free over $R$, of generating type $$\W(2^3,3,4^3,5,6^3,\dots).$$ 
    \item The $\W$-algebras $\W^\kk(\g,\OO)$ for the pairs
    $$(\g,\OO)=(\spp_{4n+2},\oo{2n,2n+2}),\quad  (\so_{4n},\oo{2n-1, 2n+1})$$ arise as 1-parameter quotients of $\fHdeg{2}{0}(\Wsp(c,\kk))$ over $\C(\kk)$, which makes the left diagram below commute.
    \item There is an embedding of vertex algebras over $\C(\kk)$
\begin{align*}
    \Wsp(c,\kk)_{\C(\kk)} \hookrightarrow \fHdeg{2}{0}(\Wsp(c,\kk))_{\C(\kk)}\otimes \Pi,
\end{align*}
which makes the right diagram below commute.
\end{enumerate}
 \begin{center}
\begin{tikzcd}
 \Wsp(c,\kk)_{\C(\kk)} \arrow[d, twoheadrightarrow] \arrow[r,dashrightarrow, "\fHdeg{2}{0}"] & \fHdeg{2}{0}(\Wsp(c,\kk))_{\C(\kk)} \arrow[d, twoheadrightarrow] &  \Wsp(c,\kk)_{\C(\kk)} \arrow[d, twoheadrightarrow] \arrow[r,hookrightarrow,] & \fHdeg{2}{0}(\Wsp(c,\kk))_{\C(\kk)}\otimes \Pi \arrow[d, twoheadrightarrow] \\
\W^\kk(\g,\OO)_{\C(\kk)} \arrow[r, dashrightarrow,  "\fHdeg{2}{0}"] & \W^\kk(\g,\widehat{\OO})_{\C(\kk)}  &\W^\kk(\g,\OO)_{\C(\kk)} \arrow[r,hookrightarrow] & \W^\kk(\g,\widehat{\OO})_{\C(\kk)} \otimes \Pi.
\end{tikzcd}
\end{center}
\end{ThmLetter}

It is generally expected that the $\W$-algebras associated with classical Lie algebras are obtained by applying certain fundamental quantum Hamiltonian reductions \cite{CFLN,CKL24}. 
In particular, the $\W$-algebras appearing in Theorem \ref{intro: theorem C} are expected to be described as two successive reductions, for example,
\begin{align*}
    \W^\kk(\spp_{4n+2},\oo{2n,2n+2}) \simeq \HH_{\oo{2n}}^0\HH_{\oo{2n+2,1^{2n}}}^0 (\V^\kk(\spp_{4n+2})).
\end{align*}
This implies that the $\W$-algebras are obtained as conformal extensions of certain subalgebras of the form
\begin{align*}
    C^\kk_1(\g,\widehat{\OO})\otimes C^\kk_2(\g,\widehat{\OO})\hookrightarrow \W^\kk(\g,\widehat{\OO}),
\end{align*}
see \S \ref{sec: fundamental reduction} and \S \ref{sec: Wsp} for details. In our setting, these subalgebras are obtained as 1-parameter quotients of another universal object, called the even spin $\W_\infty$-algebra $\W^{\mathrm{ev}}_{\infty}(c,\lambda)$ \cite{CL21, KL19}. 
This suggests a corresponding statement for the universal objects summarized as the following commutative diagram of the form
\begin{center}
\begin{tikzcd}
\W^{\mathrm{ev}}_{\infty}(c_1,\lambda_1)\otimes \W^{\mathrm{ev}}_{\infty}(c_2,\lambda_2) \arrow[d, twoheadrightarrow] \arrow[r, hook]& \fHdeg{2}{0}(\Wsp(c,\kk)) \arrow[d, twoheadrightarrow] \\
C^\kk_1(\g,\widehat{\OO})\otimes C^\kk_2(\g,\widehat{\OO}) \arrow[r, hook]& \W^\kk(\g,\widehat{\OO}). 
\end{tikzcd}
\end{center}
This conjecture is addressed in \cite{CKL24}. 
Based on Theorem \ref{intro: theorem C} and explicit description of OPEs (presentation by generators and relations) in lower conformal weights, the desired commuting Virasoro and associated vectors of primary weight-four were obtained after appropriate base change $-\otimes_R \widehat{R}$. This provides a more precise statement of the conjecture as follows.
\begin{ConjLetter}[\cite{CKL24}]
The reduction $\fHdeg{2}{0}(\Wsp(c,\kk))$ is an extension of the tensor product of two even spin $\W_\infty$-algebras 
$$\W^{\mathrm{ev}}_{\infty}(c_1,\lambda_1)\underset{\widehat{R}}{\otimes}\W^{\mathrm{ev}}_{\infty}(c_2,\lambda_2) \hookrightarrow \fHdeg{2}{0}(\Wsp(c,\kk))_{\widehat{R}}$$
over $\widehat{R}$.
Moreover, the images of $\W^{\mathrm{ev}}_{\infty}(c_i,\lambda_i)$ $(i=1,2)$ are generated by the two commuting Virasoro vectors given by the formula \eqref{new commuting virasoro} with uniquely associated primary weight-4 vectors in $\fHdeg{2}{0}(\Wsp(c,\kk))_{\widehat{R}}$.
\end{ConjLetter}

Similarly, we expect that the $\W$-algebras $$\W^k(\mathfrak{sp}_{4n},\oo{2n,2n}),\quad 
\W^k(\mathfrak{so}_{4n+2},\oo{2n+1,2n+1})$$
arise as 1-parameter quotients of a 2-parameter algebra $\W(1,2^3,3,4^3,5,\dots)$, $\W^{\mathfrak{so}_2}_{\infty}$ say.
By analogy with Conjecture A, it is tempting to speculate that the reduction $\fHdeg{2}{0}(\W^{\mathfrak{so}_2}_{\infty})$ is an extension of two copies of two even-spin $\W_\infty$-algebras; furthermore, $\W^{\mathfrak{so}_2}_{\infty}$ itself is an extension of two even-spin $\W_\infty$-algebras, consistent with an outlook proposed in \cite{CKL24}.

Finally, let us make some comments on other possible outlooks.
First, distinguished $\W$-algebras are quite challenging to study with the tools developed so far. For instance, since they do not have Heisenberg fields, spectral flow twists cannot be used to classify the modules. 
Partial reductions in Theorem \ref{intro: theorem A} provide a new path to reach certain distinguished $\W$-algebras from non-distinguished ones. 
The reduction $\fHdeg{2}{}$ being relatively simple, it would be interesting to extend Theorem \ref{intro: thm B} to other categories of modules to better control the representation theory of distinguished $\W$-algebras.
On the other hand, as Theorem \ref{intro: theorem A} is suggested by Losev's decomposition theorem \cite{Losev10} and its relative version \eqref{statement of finite reduction by stages}, it is interesting to chiralize them.
We hope to come back to these points in our future works.

\subsection*{Organization of the paper} 
The paper is organized as follows. 
In \S\ref{sec: W-algebras}, we recall the set-up of the $\W$-algebras, and 
we show that $\W$-algebras are independent under the outer automorphisms coming from the Dynkin automorphisms in \S\ref{sec: Automorphisms of root systems}.
The reduction by stages for the $\W$-algebras of classical Lie types is stated in  \S\ref{sec: main results} where we also introduce the Virasoro-type reduction of $\W$-algebras and present the first half of the main results.
In \S\ref{sec: Wakimoto}, we recall the Wakimoto realization of the $\W$-algebras and derive the explicit forms for $\W$-algebras of our interest.
They are the main ingredient for the proof to the first half of the main results established in \S\ref{sec: proof of main results}.
In \S\ref{sec: iHR}, we show the second half of the main results on the inverse Hamiltonian reductions.
The generalization of partial and inverse reductions for the type $A$ $\W$-algebras of height two is achieved in a similar way as in \S\ref{sec: More on type A}.
We generalize the Virasoro-type reduction of $\W$-algebras for modules in \S\ref{sec: Virasoro reduction for modules} by establishing Wakimoto realizations of modules over $\W$-algebras.
In \S\ref{sec: Wsp}, we study the Virasoro-type reduction for the universal $\W_\infty$-algebra $\W^{\mathfrak{sp}}_{\infty}$ and establish Theorem \ref{intro: theorem C}.
\vspace{1em}

\paragraph{\textbf{Acknowledgements}} 
We thank Tomoyuki Arakawa and Andrew R. Linshaw for interesting communications and suggestions.
J.F. is supported by a University of Melbourne Establishment Grant. S.N. is supported by JSPS Kakenhi Grant number 21H04993.
This work was also supported by the Research Institute for Mathematical Sciences,
an International Joint Usage/Research Center located in Kyoto University.

\section{\tW-algebras} \label{sec: W-algebras}
We start reviewing and fixing basic notations of the simple Lie algebras, the nilpotent orbits, and the good gradings.
We recall the construction of the affine $\W$-algebras and establish isomorphisms between $\W$-algebras $\W^\kk(\g,\OO)$ when the nilpotent orbits $\OO$ are conjugated by certain automorphisms of $\g$. 

\subsection{Classical Lie algebras}
Let $\g$ be a finite-dimensional simple Lie algebra over the complex numbers $\C$, $\g=\n_+\oplus \h \oplus \n_-$ a triangular decomposition, and $\g=\h \oplus \oplus_{\alpha \in \Delta} \g_\alpha$ the root space decomposition with $\Delta$ the set of roots.
We decompose $\Delta=\Delta_{+}\sqcup \Delta_-$ into the sets of positive and negative roots, so that $\n_\pm=\oplus_{\alpha \in \Delta_\pm} \g_\alpha$, and denote by $\Pi=\{\alpha_1,\dots,\alpha_{\ell}\}\subset \Delta_+$ the set of simple roots.
Then, we fix a set of (non-zero) root vectors $e_\alpha$ ($\alpha \in \Delta$) and the coroots $\{h_1,\cdots, h_\ell\}\subset \h$ so that $C=(h_i(\alpha_j))_{i,j=1}^\ell$ is the Cartan matrix of $\g$.
As $\{e_\alpha ; \alpha \in \Delta\}$ and $\{h_i ; i=1,\dots,\ell\}$ form a basis of $\g$, we denote by $c_{\lambda,\mu}^\nu$ the structure constants;
for example $[h_i,e_\alpha]=c_{i,\alpha}^\alpha e_\alpha$ with $c_{i,\alpha}^\alpha=h_i(\alpha)$ and
$[e_\alpha,e_\beta]=\sum_{\gamma\in\Delta}c_{\alpha,\beta}^{\gamma}e_\gamma$.

Lie algebras of classical types have a standard realization as subalgebras of the space of $n\times n$-matrices $\Mat_n=\End{(\C^n)}$ endowed with the Lie bracket defined as the commutator of matrices, i.e. $[A,B]=AB-BA$, 
(see, e.g., \cite{Carter05}):
\begin{itemize}
    \item Type $A$: $\sll_N=\{X\in \Mat_N ; \tr(X)=0\}$,
    \item Type $B$: $\so_{2N+1}=\{X\in \Mat_{2N+1} ; {}^tXJ_B+J_B X=0\}$,
    \item Type $C$: $\spp_{2N}=\{X\in \Mat_{2N} ; {}^tXJ_C+J_C X=0\}$,
    \item Type $D$: $\so_{2N}=\{X\in \Mat_{2N} ; {}^tXJ_D+J_D X=0\}$,
\end{itemize}
where $\tr{(X)}$ is the trace of $X$, ${}^tX$ the matrix transpose of $X$, and 
\begin{align}
    J_B=\begin{pmatrix}1&&\\ &&I_N\\ &I_N&\end{pmatrix},\quad 
    J_C=\begin{pmatrix}&I_N\\-I_N&\end{pmatrix},\quad
    J_D=\begin{pmatrix}&I_N\\I_N&\end{pmatrix}. 
\end{align}
with $I_N$ the identity matrix in $\Mat_N$.

In type $A$, the root vectors are identified with the elementary matrices $\E_{i,j}$ $1\leq i,j\leq N,\, i\neq j$ in $\Mat_n$ and the coroots with the diagonal matrices $\E_{i,i}-\E_{i+1,i+1}$.
For types $BCD$, it is useful to take the index set of the basis of the natural representation to be 
\begin{itemize}
    \item Type $B$: $\{0,1,\dots,N,-1,\dots,-N\}$,
    \item Type $C$: $\{1,\dots,N,-1,\dots,-N\}$,
    \item Type $D$: $\{1,\dots,N,-1,\dots,-N\}$,
\end{itemize}
so that the basis of $\g$ is given by the matrices
\begin{align*}
    BD:\ \E^{o}_{i,j}=\E_{i,j}-\E_{-j,-i},\qquad
    C:\ \E^{s}_{i,j}=\begin{cases}
    \E_{i,j}-\E_{-j,-i} & (ij>0) \\
    \E_{i,j}+\E_{-j,-i} & (ij<0 \text{ and } |i|\neq |j|)\\
    \E_{i,j}  & (ij<0 \text{ and } |i|=|j|)
\end{cases}.
\end{align*}

\subsection{Nilpotent orbits and good gradings}\label{sec: Nilpotent orbits}

Let $\mathcal{N}\subset \g$ be the nilpotent cone, i.e. the set consisting of $\ad$-nilpotent elements in $\g$. It is closed under the adjoint action of the corresponding simple algebraic group $G$, which decomposes $\mathcal{N}$ into the conjugacy classes, called the nilpotent orbits.

For $\g=\sll_N$, the Jordan classification says that the nilpotent orbits are parametrized by the partitions $\lambda=[\lambda_1^{m_1},\dots,\lambda_s^{m_s}]$ (with $\lambda_1>\dots >\lambda_s>0$) of $N$. Setting $\OO_{\lambda}$ the nilpotent orbit corresponding to the partition $\lambda$,
we have 
$\mathcal{N}(\sll_N)=\bigsqcup_{\lambda\in \mathscr{P}_{N}} \OO_\lambda$ with $\mathscr{P}_{N}$ the set of partitions of $N$. 
For $\g$ of type $BCD$, the standard realization gives a similar description of the nilpotent orbits
\begin{align}
    \Nil(\so_{2n+1})=\bigsqcup_{\lambda\in \mathscr{P}^o_{2n+1}} \OO_\lambda,\quad 
    \Nil(\spp_{2n})=\bigsqcup_{\lambda\in \mathscr{P}^s_{2n}} \OO_\lambda,\quad 
    \Nil(\so_{2n})=\bigsqcup_{\lambda\in \mathscr{P}^o_{2n}} \OO_\lambda^\blacklozenge
\end{align}
    with 
\begin{align}
\mathscr{P}^o_{N}=\{\lambda\vdash N\mid \lambda_i\in 2\Z \Rightarrow m_i\in 2\Z\},\quad
\mathscr{P}^s_{N}=\{\lambda\vdash N\mid \lambda_i\notin 2\Z \Rightarrow m_i\in 2\Z\}.  
\end{align}
In type $D$, the symbol $\blacklozenge$ means that we have two disjoint orbits $\OO_\lambda^\mathrm{I}$ and $\OO_\lambda^\mathrm{II}$ in the case when $\lambda$ is very even, i.e., $\lambda_i\in 2\Z$ for all $i$.
See \cite{CM93} for more details.

Given a nilpotent element $f\in\g$, a $\frac{1}{2}\Z$-grading $\Gamma\colon  \g=\bigoplus_{j\in \frac{1}{2}\Z}\g_j$ is said \emph{good} for $f$ if it satisfies the following conditions 
\begin{align}
   f \in \g_{-1},\quad \ker \ad_f \cap \g_+=0,\quad \g_{-}\subset \im \ad_f
\end{align}
where we set $\g_{\pm}=\bigoplus_{j>0}\g_{\pm j}$.
The pair $(f,\Gamma)$ is called a \emph{good pair} \cite{KRW03}. 
Two good pairs $(f,\Gamma)$, $(f',\Gamma')$ are equivalent if they are conjugated under the adjoint $G$-action.
Given a (non-zero) nilpotent element, the good gradings up to equivalence are classified \cite{EK05} and expressed in terms of pyramids  \cite{BG05}. 

The pyramids are numbered Young diagrams with shiftings.
We explain concretely on an example. 
The Young tableau of the nilpotent orbit $\oo{3,2}$ in $\g=\sll_5$ is the diagram  
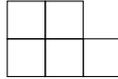
\begin{equation*}
        \begin{tikzpicture}[every node/.style={draw,regular polygon sides=4,minimum size=1cm,line width=0.04em},scale=0.5, transform shape]
        \node at (0,0)  {};
        \node at (1,0)  {};
        \node at (-1,0) {};
        \node at (0,1)  {};
        \node at (-1,1) {};
    \end{tikzpicture}
\end{equation*}
Here the boxes are of $1\times 1$ size, and we place the Young diagram in the $xy$-plane so that the center of the middle box of the bottom row sets at the origin. 
In type $A$, shiftings are allowed horizontally on upper rows with $\frac{1}{2}$ increments so that their boxes do not exceed the corners of the lower boxes. 
In the above case, all the shiftings are as follows:
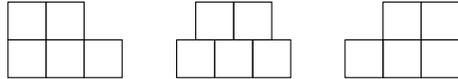
\begin{equation*}
    \begin{tikzpicture}[every node/.style={draw,regular polygon sides=4,minimum size=1cm,line width=0.04em},scale=0.5, transform shape]
        \node at (0,0)  {};
        \node at (1,0)  {};
        \node at (-1,0) {};
        \node at (0,1)  {};
        \node at (-1,1) {};
    \end{tikzpicture}\qquad 
        \begin{tikzpicture}[every node/.style={draw,regular polygon sides=4,minimum size=1cm,line width=0.04em},scale=0.5, transform shape]
        \node at (0,0)  {};
        \node at (1,0)  {};
        \node at (-1,0) {};
        \node at (-0.5,1)  {};
        \node at (0.5,1) {};
    \end{tikzpicture}\qquad
        \begin{tikzpicture}[every node/.style={draw,regular polygon sides=4,minimum size=1cm,line width=0.04em},scale=0.5, transform shape]
        \node at (0,0)  {};
        \node at (1,0)  {};
        \node at (-1,0) {};
        \node at (0,1)  {};
        \node at (1,1) {};
    \end{tikzpicture}.
\end{equation*}
Then boxes are labelled from the right to the left by the indexes $\{1,2,3,4,5\}$ of the basis of the natural representation $\C^5$ of $\sll_5$, we obtain for instance the pyramids: 
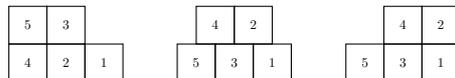
\begin{equation*}
    \begin{tikzpicture}[every node/.style={draw,regular polygon sides=4,minimum size=1cm,line width=0.04em},scale=0.5, transform shape]
        \node at (1,0)  {1};
        \node at (0,0)  {2};
        \node at (-1,0) {4};
        \node at (0,1)  {3};
        \node at (-1,1) {5};
    \end{tikzpicture}\qquad 
        \begin{tikzpicture}[every node/.style={draw,regular polygon sides=4,minimum size=1cm,line width=0.04em},scale=0.5, transform shape]
        \node at (1,0)  {1};
        \node at (0,0)  {3};
        \node at (-1,0) {5};
        \node at (0.5,1) {2};
        \node at (-0.5,1) {4};
    \end{tikzpicture}\qquad
        \begin{tikzpicture}[every node/.style={draw,regular polygon sides=4,minimum size=1cm,line width=0.04em},scale=0.5, transform shape]
        \node at (1,0)  {1};
        \node at (0,0)  {3};
        \node at (-1,0) {5};
        \node at (1,1) {2};
        \node at (0,1)  {4};
    \end{tikzpicture}
\end{equation*}

From a pyramid, we read the nilpotent element $f$ given by
\begin{align}\label{association of nilpotent orbit}
    f=\sum \delta_{i\rightarrow j}\E_{i,j}
\end{align}
where $\delta_{i\rightarrow j}$ is $1$ if the boxes {\tiny $\numtableaux{i}$} and {\tiny $\numtableaux{j}$} are adjacent with {\tiny $\numtableaux{i}$} on the left of {\tiny $\numtableaux{j}$} and $0$ otherwise; 
thus the Jordan class of $f$ agrees with the one obtained from the original Young diagram. 

The grading is obtained by defining a $\frac{1}{2}\Z$-grading on the natural representation by reading the $x$-coordinates for the basis.
The gradings for the previous labelled pyramids are as follows:
\begin{align}\label{examples of weighted Dynkin diagram}
\begin{tikzpicture}
{\dynkin[root
radius=.1cm,labels={\alpha_1,\alpha_2,\alpha_3, \alpha_4},labels*={1,0,1,0},edge length=0.8cm]A{oooo}};		
\end{tikzpicture}\qquad 
\begin{tikzpicture}
{\dynkin[root
radius=.1cm,labels={\alpha_1,\alpha_2,\alpha_3, \alpha_4},labels*={1/2,1/2,1/2,1/2},edge length=0.8cm]A{oooo}};		
\end{tikzpicture}\qquad 
\begin{tikzpicture}
{\dynkin[root
radius=.1cm,labels={\alpha_1,\alpha_2,\alpha_3, \alpha_4},labels*={0,1,0,1},edge length=0.8cm]A{oooo}};		
\end{tikzpicture}
\end{align}
This is equivalent to take the grading defined by the semisimple element
\begin{equation}
    x=\sum_{i}\left(x_i\E_{i,i}\right)-\frac{1}{N}\left(\sum_{i}x_i\right)I_N
\end{equation}
where $x_i$ is the $x$-coordinate of the center of the box {\tiny $\numtableaux{i}$}.
Indeed, the grading can be expressed by the weights $\Gamma(\alpha)=x(\alpha_i)$ of simple roots summarized in the weighted Dynkin diagrams as \eqref{examples of weighted Dynkin diagram}. 

In the above examples, the good gradings are compatible with the triangular decomposition in the sense that $\g_\pm \subset \n_\pm$ holds. This is the consequence 
of labelling the pyramids with the smaller numbers to the right. 
In this paper, all the good pairs will be taken in this way.
We will use the decomposition of the root system $\Delta$ with respect to the good grading $\Gamma$:
\begin{align}
    \Delta_j=\{\alpha \in \Delta; \Gamma(\alpha)=j\},\quad \Delta_{>0}=\{\alpha \in \Delta_+; \Gamma(\alpha)>0\}. 
\end{align}

For type $BCD$, some additional complications appear.
In particular, the Young diagrams are modified and placed in the $xy$ plane so the pyramids are invariant under half rotation and the labelling is skew-symmetric.
For instance, in type $C$, we have the following pyramids for the orbits $\oo{3^2}$, $\oo{4^2}$ and $\oo{6,2}$ of $\spp_6$ and $\spp_8$:
\begin{equation*}
\begin{tikzpicture}[every node/.style={draw,regular polygon sides=4,minimum size=1cm,line width=0.04em},scale=0.58, transform shape]
        \node at (1,1)  { 1};
        \node at (0,1)  { 3};
        \node at (-1,1)  {-2};
        \node at (1,0)  { 2};
        \node at (0,0)  { -3};
        \node at (-1,0)  { -1};
    \end{tikzpicture}\qquad 
\begin{tikzpicture}[every node/.style={draw,regular polygon sides=4,minimum size=1cm,line width=0.04em},scale=0.58, transform shape]
        \node at (2,1) { 1};
        \node at (1,1) { 2};
        \node at (0,1) { 4};
        \node at (-1,1) { -3};
        \node at (1,0) { 3};
        \node at (0,0) { -4};
        \node at (-1,0) { -2};
        \node at (-2,0) { -1};
    \end{tikzpicture}\qquad 
    \begin{tikzpicture}[every node/.style={draw,regular polygon sides=4,minimum size=1cm,line width=0.04em},scale=0.58, transform shape]
        \node at (2.5,0) (4) { 1};
        \node at (1.5,0) (4) { 2};
        \node at (0.5,0) (4) { 4};
        \node at (-2.5,0) (4) { -1};
        \node at (-1.5,0) (4) { -2};
        \node at (-0.5,0) (4) { -4};
        \node at (0.5,1) (4) { 3};
        \node at (-0.5,1) (4) { $\times$};
        \node at (0.5,-1) (4) { $\times$};
        \node at (-0.5,-1) (4) { -3};
    \end{tikzpicture}
\end{equation*}
In general, if a partition $\lambda=[\lambda_1^{m_1},\dots,\lambda_s^{m_s}]$ has a (necessarily even) coefficient $\lambda_i$ ($i=2,\dots s$) that appears with an odd multiplicity $m_i$
then we replace the first row of boxes 
{\tiny $\numtableaux{{} & \rs{0.7}{\cdots} &  & {} &\rs{0.7}{\cdots}& {}}$} 
of length $\lambda_i$ by two rows 
{\tiny $\numtableaux{{} & \rs{0.7}{\cdots} &{} & \times &\rs{0.7}{\cdots} & \times }$} and 
{\tiny $\numtableaux{ \times & \rs{0.7}{\cdots} & \times & {} & \rs{0.7}{\cdots} & {}}$} 
that are placed skew-symmetrically.
The associated grading is defined as before and the nilpotent element $f$ is given by 
\begin{align}
    f=\sum \delta_{i\rightarrow j}\E_{i,j}+ \sum \delta_{\times\rightarrow i}\E_{-i,i}
\end{align}
where $\delta_{i\rightarrow j}$ is as before and $\delta_{\times\rightarrow i}$ is $1$ if the boxes {\tiny $\numtableaux{\times & i}$} 
appear in the pyramids 
--- then {\tiny $\numtableaux{\text{-}i & \times }$} also appear by skew-symmetry ---
and $0$ otherwise.
We refer to \cite{BG05} for the general detailed construction. 
In the following, the pyramids and the associated good pairs will be presented explicitly for the reader's convenience.

\subsection{\tW-algebras}
Given a simple Lie algebra $\g$, denote by
\begin{equation}
\widehat{\g}=\g[t^{\pm 1}] \oplus \ \C K
\end{equation}
the affine Kac--Moody algebra associated with $\g$. 
The universal affine vertex algebra $\V^\kk(\g)$ associated with $\g$ at level $\kk\in\C$ is the parabolic Verma $\widehat{\g}$-module
\begin{equation}
    \V^\kk(\g)=U(\widehat{\g})\otimes_{U(\g[t]\oplus\C K)}\C_\kk
\end{equation}
where $\C_\kk$ is the one-dimensional representation of $\g[t]\oplus\C K$ on which $\g[t]$ acts trivially and $K$ acts as the multiplication by the scalar $\kk$.
There is a unique vertex algebra structure on $\V^\kk(\g)$, which is generated by the fields
\begin{equation}
    u(z)=Y(u_{-1},z)=\sum_{n\in\Z}u_n z^{-n-1},\qquad u_n=x\otimes t^n,\, u\in\g,
\end{equation}
satisfying the OPEs
\begin{align}
    u(z)v(w)\sim \frac{[u,v](w)}{(z-w)}+\frac{\kk(u,v)}{(z-w)^2}.
\end{align}
where $(\cdot, \cdot)$ is the normalized invariant bilinear form on $\g$.

In general, we will denote by $Y(A,z)$ --- or briefly $A(z)$ ---
the field associated with an element $A$ in a vertex algebra, by $\wun$ the vacuum vector, and use the notation $AB$ instead of $:AB:$ to define the normally ordered product of $A$ and $B$ for simplicity. 

The $\W$-algebras are vertex algebras obtained as the quantum Hamiltonian reductions of the affine vertex algebras $\V^\kk(\g)$ \cite{FF90, KRW03}. 
Fix a good pair $(f, \Gamma)$ consisting of a (non-zero) nilpotent element $f$ and a $\frac{1}{2}\Z$-grading $\Gamma\colon \g=\bigoplus_j\g_j$ good with respect to $f$.
Then the homogeneous space $\g_{1/2}$ has a symplectic structure defined by the bilinear form $\langle u,v \rangle=(f,[u,v])$. 
We associate with it the neutral free fermion vertex algebra $\Phi(\g_{1/2})$ which is strongly and freely generated by the fields $\Phi_\alpha(z)=\sum \Phi_{\alpha,n}z^{-n-1}$ ($\alpha\in \Delta_{1/2}$) satisfying the OPEs
\begin{align}
 \Phi_\alpha(z)\Phi_\beta(w)\sim \frac{\langle e_\alpha, e_\beta \rangle}{(z-w)}.
\end{align}

Let $\bigwedge{}^{\bullet}_{\varphi,\varphi^*}$ denote the $bc$-system strongly and freely generated by the odd fields $\varphi(z),\varphi^*(z)$ satisfying the OPEs
\begin{align}
 \varphi(z)\varphi^*(w)\sim \frac{1}{(z-w)},\quad  \varphi(z)\varphi(w)\sim 0\sim \varphi^*(z)\varphi^*(w).
\end{align}
For each positively $\Gamma$-graded root $\alpha$ in $\Delta_{>0}$, let us denote by $\bigwedge{}^{\bullet}_{\varphi_\alpha,\varphi^*_\alpha}$ the copy of $\bigwedge{}^{\bullet}_{\varphi,\varphi^*}$ generated by the fields $\varphi_\alpha(z),\varphi^*_\alpha(z)$.

The BRST complex associated to the good pair $(f, \Gamma)$ is defined as
\begin{align}\label{BRST cohomology}
    C_f^\bullet(\V^\kk(\g))=\V^\kk(\g)\otimes \Phi(\g_{1/2}) \otimes \chfermion,\quad
    \chfermion:=\bigotimes_{\alpha \in \Delta_{>0}} \bigwedge{}^{\bullet}_{\varphi_\alpha,\varphi^*_\alpha}
\end{align}
and equipped with the differential 
\begin{align}\label{BRST differential}
    d=\int Y(Q,z)\ \dd z,\quad Q=Q_{\mathrm{st}}+Q_\Phi+Q_\chi
\end{align}
where
\begin{equation}
    \begin{aligned}
    &Q_{\mathrm{st}}=\sum_{\alpha\in\Delta_{>0}} e_\alpha \varphi_\alpha^*-\frac{1}{2}\sum_{\alpha,\beta,\gamma\in\Delta_{>0}} c_{\alpha,\beta}^\gamma \varphi_\alpha^*\varphi_\beta^* \varphi_\gamma,\\
    &Q_\chi=\sum_{\alpha\in\Delta_{>0}} (f,e_\alpha)\varphi_\alpha^*,\qquad Q_\Phi=\sum_{\alpha \in \Delta_{1/2}} \Phi_\alpha \varphi_\alpha^*.
\end{aligned}
\end{equation}
The $\W$-algebra $\W^\kk(\g,f)$ associated with $\g$ and the good pair $(\g,f)$ at level $\kk$ is defined as 
\begin{align}
    \W^\kk(\g,f)=\HH^0(C_f^\bullet(\V^\kk(\g)),d).
\end{align}
{Note that $\W^\kk(\g,f)$ is freely strongly generated by fields that correspond to a basis of the centralizer $\g^f$ homogeneous for the grading $\Gamma$ \cite{KW04}.}
However, $\W^\kk(\g,f)$ is independent as a vertex algebra of the choice of the representative $f$ in a given nilpotent orbit $\OO$ and the good grading $\Gamma$ \cite{AKM15} so that we will occasionally use the notation $\W^\kk(\g,\OO)$ in the following.
In addition, replacing $\V^\kk(\g)$ by a $\V^\kk(\g)$-module $M$ in \eqref{BRST cohomology}, we obtain the functor from the category of $\V^\kk(\g)$-modules to the category of $\W^\kk(\g,\OO)$-modules, 
\begin{align}
    \HH^0_\OO\colon \V^\kk(\g)\mod \rightarrow \W^\kk(\g,\OO)\mod.
\end{align}

For later use, let us rewrite the BRST complex \eqref{BRST cohomology} using the semi-infinite cohomology for the Lie algebra $\g_+(\!(t)\!)$. Consider the $\g_+(\!(t)\!)$-module
\begin{align}
    \Phi^\chi(\g):=U(\g_+(\!(t)\!))\otimes_{U(\mathscr{L})}\C_{\chi},
\end{align}
where $\mathscr{L}:=\g_{\frac{1}{2}}[\![t]\!]\ltimes \g_{\geq1}(\!(t)\!)$ is a Lie superalgebra acting on the one-dimensional representation $\C_{\chi}$ by
\begin{equation}
    u_n \mapsto \mathrm{Res}_{t=0} (f,u)t^n \mathrm{d}t=\delta_{n,-1}(f,u).
\end{equation}
As vector space, $\Phi^\chi(\g)\simeq U(\g_{\frac{1}{2}}[t^{-1}]t^{-1})$.
It has a structure of the vertex algebra given by the isomorphism
\begin{equation}
\Phi^\chi(\g)\xrightarrow{\sim} \Phi(\g_{1/2}),\quad e_{\alpha^1,n_1}\dots e_{\alpha^r,n_r}\mapsto \Phi_{\alpha^1,n_1}\dots \Phi_{\alpha^r,n_r},    
\end{equation}
{for $\alpha^i\in\Delta_{1/2}$, $n_i\in\Z$.}
Moreover, it is a $\g_+(\!(t)\!)$-module by construction {and the action} is given by
\begin{align}
   F^\chi\colon e_{\alpha}(z)\mapsto 
   \begin{cases}
       \Phi_\alpha(z) & \text{if }\Gamma(\alpha)=\tfrac{1}{2},\\
       (f,e_\alpha) & \text{if } \Gamma(\alpha)\geq 1.
   \end{cases}
\end{align}

Consider the tensor product representation $\V^\kk(\g)\otimes \Phi^\chi(\g)$ over $\g_+(\!(t)\!)$ and define the semi-infinite complex
\begin{align}\label{semi-infinite complex}
    C^{\frac{\infty}{2}+\bullet}(\g_+(\!(t)\!); {\V^\kk(\g)}\otimes \Phi^\chi(\g)):= ({\V^\kk(\g)}\otimes \Phi^\chi(\g))\otimes \chfermion
\end{align}
equipped with the differential $d_{\mathrm{st}}=\int Y(Q_{\mathrm{st}},z)$, see \eqref{BRST differential}.
Hence, \eqref{semi-infinite complex} agrees with the BRST complex $C_f^\bullet(\V^\kk(\g))$.
Replacing $\V^\kk(\g)$ by a $\V^\kk(\g)$-module $M$, we again obtain an isomorphism
\begin{align}\label{BRST is semi-infinite}
    \HH^\bullet_\OO(M)\simeq \HH^{\frac{\infty}{2}+\bullet}(\g_+(\!(t)\!);M\otimes \Phi^\chi(\g)).
\end{align}

\subsection{Automorphisms of root systems and isomorphisms of \tW-algebras} \label{sec: Automorphisms of root systems}
Recall that the automorphism group $\Aut(\Delta)$ of the root system $\Delta$ decomposes into the semidirect product
 \begin{align}\label{semidirect product of auto}
     \Aut(\Delta)=\Gamma\ltimes W
 \end{align}
of the Weyl group $W$ of $\g$ and the automorphism group $\Gamma$ of the associate Dynkin diagram \cite{Hum72}.
 Given an automorphism $\tau\in\Aut(\Delta)$, $(h_{\tau(\alpha_i)}(\tau(\alpha_j)))_{i,j=1}^\ell$ is the Cartan matrix $C$ of $\g$ by definition. Hence, the Chevalley--Serre presentation of $\g$ implies that the linear map
\begin{align}\label{Lie alg auto from Dynkin auto}
    \g\xrightarrow{\sim} \g,\quad e_\alpha\mapsto e_{\tau(\alpha)},\ h_{\alpha_i}\mapsto h_{\tau(\alpha_i)}
\end{align} 
is an isomorphism of Lie algebras, that we still denote by $\tau$.

\begin{theorem}\label{auto-independence of Walg}
For a nilpotent element $f\in\g$ and an automorphism $\tau\in \Aut(\Delta)$, there exists an isomorphism of vertex algebras
$$\W^\kk(\g,f)\simeq\W^\kk(\g,\tau (f) ).$$
\end{theorem}

\begin{proof}
The case $f=0$ is clear, so we assume $f\neq 0$ and take an $\sll_2$-triple $\{e,h=2x,f\}$. Then $\{\tau (e), \tau (h)=2 \tau (x),\tau (f)\}$ is an $\sll_2$-triple for $\tau (f)$. 
Let $\Gamma_x$ and $\Gamma_{\tau(x)}$ be the Dynkin grading with respect to $x$ and $\tau(x)$ respectively.
We compare the 
BRST complexes corresponding to the good pairs $(f,\Gamma_x)$, $(\tau(f),\Gamma_{\tau(x)})$ that define
$\W^\kk(\g,f), \W^\kk(\g,\tau (f) )$:
\begin{align}
    C_f^\bullet(\V^\kk(\g))=\V^\kk(\g)\otimes \Phi(\g_{1/2}) \otimes \chfermion,\quad C_{\tau (f)}^\bullet(\V^\kk(\g))=\V^\kk(\g)\otimes \Phi(\g_{1/2}^\tau) \otimes \chfermion^\tau. 
\end{align}
Here we denote by $\Phi(\g_{1/2}^\tau)$ the neutral free fermion vertex algebra for the good pair $(\tau (f), \Gamma_{\tau(x)})$ and similarly for $\chfermion^\tau$. 
Note that they are resepctively generated by $\Phi_{\tau(\alpha)}$ $(\alpha \in \Delta_{1/2})$ and $\varphi_{\tau(\alpha)}, \varphi^*_{\tau(\alpha)}$ $(\alpha \in \Delta_{>0})$.
It is straightforward to show that 
  \begin{align}\label{isom on BRST complexes}
  \begin{array}{ccccl}
       \Upsilon\colon& C_f^\bullet(\V^\kk(\g)) &\rightarrow & C_{\tau (f)}^\bullet(\V^\kk(\g)) &  \\
        &u  &\mapsto & \tau(u)& (u\in \g)\\
        &\Phi_\alpha & \mapsto & \Phi_{\tau (\alpha)}\ &(\alpha \in \Delta_{1/2})\\
       &\varphi_\alpha, \varphi^*_\alpha &\mapsto & \varphi_{\tau(\alpha)}, \varphi^*_{\tau(\alpha)} &(\alpha \in \Delta_{>0})
  \end{array}
  \end{align}
is an isomorphism of vertex algebras.
Indeed, the automorphism $\tau\in\Aut(\g)$ (see \eqref{Lie alg auto from Dynkin auto}) preserves the Killing form and thus the normalized invariant bilinear form $(\cdot, \cdot)$. 
For instance, for all $u,v\in\g$ and $\alpha,\beta\in\Delta_{1/2}$, we have
\begin{align*}
    \tau(u)(z) \tau(v)(w)&\sim \frac{[\tau(u),\tau(v)](w)}{(z-w)}+\frac{\kk(\tau(u),\tau(v))}{(z-w)^2}\\
    &= \frac{\tau([u,v])(w)}{(z-w)}+\frac{\kk(u,v)}{(z-w)^2}\sim \tau\left( u(z)v(w)\right),\\
    \Phi_{\tau(\alpha)}(z)\Phi_{\tau(\beta)}(w)&\sim \frac{\langle e_{\tau(\alpha)},e_{\tau(\beta)}\rangle}{(z-w)}\\
    &=\frac{(\tau(f),[e_{\tau(\alpha)},e_{\tau(\beta)}])}{(z-w)}=\frac{(f,[e_{\alpha},e_{\beta}])}{(z-w)}\sim \Phi_{\alpha}(z)\Phi_{\beta}(w).
\end{align*}

Now, it is clear that $\Upsilon$ identifies the differentials $d$ and $d_\tau$ given by the formula \eqref{BRST differential}.
Hence, $\Upsilon$ induces an isomorphism of their cohomologies, in particular their degree $0$ cohomologies that are the $\W$-algebras. This completes the proof.

\end{proof}

 \begin{remark}
     For non-critical levels, the isomorphism in Theorem \ref{auto-independence of Walg} also preserves the Kac--Roan--Wakimoto conformal vectors \cite{KRW03} and thus it is an isomorphism as vertex operator algebras.
 \end{remark}

The automorphisms in $\Aut (\g)$ fall into two cases: inner and outer automorphisms. 
The first case conjugates nilpotent elements inside the same orbits. 
The automorphisms coming from $W$ in \eqref{semidirect product of auto} correspond to this case, and Theorem \ref{auto-independence of Walg} reads a special case of the result in \cite{AKM15} that $\W$-algebras depend on nilpotent orbits rather than their representatives.
On the other hand, automorphisms coming from $\Gamma$ are outer automorphisms and give more isomorphisms when $\Gamma$ is non-trivial. 
Table \ref{table:out_gp} recalls the automorphism group $\Gamma$ for classical and exceptional Lie types.

\begin{table}[h!]
	\centering
	\renewcommand{\arraystretch}{1.3}
	\begin{tabular}{|c||c|M{1.2cm}|c| M{1.2cm}|c|}
		\hline
		Lie type & $A_n$ & $B_n, C_n$ & $D_n$ & $E_6$ & $E_{7,8}$, $F_4$, $G_2$ \\
		\hline
            $\Gamma$ & $\begin{array}{c} 1\ (n=1)\\ \Z_2\ (n>1) \end{array}$ & $1$ & $\begin{array}{c} \mf{S}_3\ (n=4)\\ \Z_2\ (n>4) \end{array}$ & $\Z_2$ & $1$ \\ \hline
	\end{tabular}
 \captionsetup{font=small }
	\caption{Automorphism groups $\Gamma$ of Dynkin diagrams.}\label{table:out_gp}
\end{table}

Recall that nilpotent orbits are parametrized in terms of weighted Dynkin diagrams \cite[\S3]{CM93}.
Concretely, one may pick up a representative $f$ in each orbit $\OO$ and complete it into an $\sll_2$-triple $(e,h=2x,f)$ --- in the zero-orbit case, we take it to be $(0,0,0)$ --- so that $x\in \h$ and $x(\alpha_i)\in \Z_{\geq0}$. 
By construction, the automorphisms in $\Gamma$ can permute the nilpotent orbits 
by changing the weights $x(\alpha_i)$ of their weighted Dynkin diagrams in an obvious way.

When $\g$ is of type $A$ or $E_6$, 
weighted Dynkin diagrams associated to nilpotent orbits are all invariant under the action of
$\Gamma\simeq \Z_2$ 
\cite[\S3.6, 8.4]{CM93}. 
However, when $\g$ is of type $D_n$ ($n\geq3$), the action of $\Gamma$ sometimes exchanges some nilpotent orbits. 
For $D_n$ ($n>3$), $\Gamma(D_n)\simeq \Z_2$ is generated by the translation $\tau$ which swaps the nodes $\alpha_{n-1}$ and $\alpha_n$ in Figure \ref{table:dynkinD}.
For $D_4$, $\Gamma(D_4)\simeq \mf{S}_3$ is generated by $\tau$ and the three-cycle element $\sigma$ which rotates $\alpha_1,\alpha_3,\alpha_4$ as in Figure \ref{table:dynkinD4}.
\begin{figure}[h]
	\begin{minipage}[l]{0.5\linewidth}
		\begin{center} 
			\begin{tikzpicture}
            {\dynkin[root
            radius=.1cm,labels={\alpha_1,\alpha_2,\alpha_{n-2}, \alpha_{n-1},\alpha_n},edge length=0.8cm]D{oo.ooo}};	
            \draw [<->]  (3,0.4) to [in=30, out=-30]  (3,-0.4);
            \end{tikzpicture}
           \captionsetup{font=small }
            \captionof{figure}{Generator $\tau$ for $\Gamma(D_n)$}\label{table:dynkinD}
		\end{center}
	\end{minipage}
	\begin{minipage}[c]{0.49\linewidth}
		\begin{center}
			\begin{tikzpicture}
            {\dynkin[root
            radius=.1cm,labels={\alpha_1,\alpha_2,\alpha_3, \alpha_4},edge length=0.8cm]D{oooo}};
            \draw [->]  (1.5,0.3) to [in=20, out=-20]  (1.5,-0.3);
            \draw [<-]  (0.1,-0.5) to [in=-150, out=-80]  (0.7,-0.8);
            \draw [->]  (0.1,0.5) to [in=170, out=80]  (0.7,0.8);
            \end{tikzpicture}
           \captionsetup{font= small}
            \captionof{figure}{Generator $\sigma$ of $\Gamma(D_4)$}\label{table:dynkinD4}
		\end{center}
	\end{minipage}
\end{figure}

By \cite[\S5.3]{CM93}, the element $\tau$ permutes the nilpotent orbits $\OO^{\mathrm{I}}_\lambda$ and $\OO^{\mathrm{II}}_\lambda$ for very even partitions $\lambda$ and $\sigma$ permutes cyclically the orbits $\{\OO^\mathrm{I}_{[2^4]}, \OO^\mathrm{II}_{[2^4]}, \OO_{[3,1^5]}\}$ and $\{\OO^\mathrm{I}_{[4^2]}, \OO^\mathrm{II}_{[4^2]}, \OO_{[5,1^3]} \}$, respectively.
Therefore, Theorem \ref{auto-independence of Walg} implies the following. 
\begin{corollary}\label{coincidences of W-alg by Dynkin auto}
\phantom{x}
\begin{enumerate}[wide, labelindent=0pt]
\item For type $D$, i.e. $\g=\so_{2N}$, and very even partitions $\lambda\in \mathcal{P}^o_{2N}$, we have 
    \begin{align*}
    \W^\kk(\so_{2N},\OO_{\lambda}^\mathrm{I})\simeq \W^\kk(\so_{2N},\OO_{\lambda}^{\mathrm{II}}).
    \end{align*}
\item In addition, for $\g=\so_{8}$, we have 
    \begin{align*}
    \W^\kk(\so_{8},\OO^\blacklozenge_{[2^4]})\simeq \W^\kk(\so_{8},\OO_{[3,1^5]}),\quad \W^\kk(\so_{8},\OO^{\blacklozenge}_{[4^2]})\simeq \W^\kk(\so_{8},\OO_{[5,1^3]})
    \end{align*}
    for $\blacklozenge= \mathrm{I}, \mathrm{II}$.
    \end{enumerate}
\end{corollary}
\noindent
For this reason, we write $\W^\kk(\so_{2N},\OO_{\lambda})$ instead of  $\W^\kk(\so_{2N},\OO_{\lambda}^{\mathrm{I}})\simeq \W^\kk(\so_{2N},\OO_{\lambda}^{\mathrm{II}})$.

To end this section, let us visualize the nilpotent orbits of $\so_{8}$ for convenience. 
The Hesse diagram of nilpotent orbits is given in Table \ref{table: Hesse so8}. In the table, two nilpotent orbits are connected if and only if the left one appears in the closure of the right one. 
Corollary \ref{coincidences of W-alg by Dynkin auto} \textit{(2)} means that the $\W$-algebras corresponding to nilpotent orbits in the same columns are all isomorphic. 
\begin{center}
\begin{table}[h!]
\begin{tikzpicture}[scale=.8]
  \node (a) at (0,0) {$\OO_{[1^8]}$};
  \node (b) at (2,0) {$\OO_{[2^2,1^4]}$};
  \node (c1) at (4,1) {$\OO_{[3,1^5]}$};
  \node (c2) at (4,0) {$\OO_{[2^4]}^{\mathrm{I}}$};
  \node (c3) at (4,-1) {$\OO_{[2^4]}^{\mathrm{II}}$};
  \node (d) at (6,0) {$\OO_{[3,2^2,1]}$};
  \node (e) at (8,0) {$\OO_{[3^2,1^2]}$};
  \node (f1) at (10,1) {$\OO_{[5,1^3]}$};
  \node (f2) at (10,0) {$\OO_{[4^2]}^{\mathrm{I}}$};
  \node (f3) at (10,-1) {$\OO_{[4^2]}^{\mathrm{II}}$};
  \node (g) at (12,0) {$\OO_{[5,3]}$};
  \node (h) at (14,0) {$\OO_{[7,1]}$};
  \draw (a) -- (b) -- (c1) -- (d) -- (c2) -- (b) -- (c3) -- (d);
  \draw (d) -- (e) -- (f1) -- (g) -- (f2) -- (e) -- (f3) -- (g) -- (h);
\end{tikzpicture}
\caption{Hesse diagram of the nilpotent orbits of $\so_8$.}\label{table: Hesse so8}
\end{table}
\end{center}

\section{Partial reductions}\label{sec: main results}
In this section, we present the main results of the paper on partial reductions of $\W$-algebras corresponding to certain partitions of height two.

\subsection{Fundamental reductions}\label{sec: fundamental reduction}
It has been recently conjectured (see \cite{CFLN} for instance) that
$\W$-algebras in type $A$, i.e. $\W^\kk(\sll_N,\OO)$, 
can be recovered by applying successively BRST reductions associated with hook-type partitions --- those of the form $[n,1^m]$. 
For example, there is an isomorphism of vertex algebras \cite{CFLN, FFFN}
\begin{align}
\W^\kk(\sll_4,\oo{2^2})\simeq \HH_{\oo{2}}^0\circ\HH_{\oo{2,1^2}}^0(\V^\kk(\sll_4)).
\end{align}
Here we apply the second reduction $\HH_{\oo{2}}^0$ to the affine vertex subalgebra $\V^{\kk^\sharp}(\sll_2)$ at level $\kk^\sharp=\kk+1$ sitting inside the $\W$-algebra $\W^\kk(\sll_4,\oo{2,1^2})=\HH_{\oo{2,1^2}}^0(\V^\kk(\sll_4))$.

In general, given a partition $\lambda=(\lambda_1\geq \lambda_2 \geq \cdots \geq \lambda_\ell)$,
there is, for all $i=1,\dots,\ell-1$, an affine vertex subalgebra 
\begin{align}
    \V^{\kk^{\sharp}_{i+1}}(\sll_{N_{i+1}})\subset \HH_{\OO_{\widehat{\lambda}_{i}}}^0\cdots \HH_{\OO_{\widehat{\lambda}_{1}}}^0 (\V^\kk(\sll_N)).
\end{align}
where the partitions $\widehat{\lambda}_i$ are 
\begin{equation}\label{eq:hookA}
    N_i=N-(\lambda_1+\lambda_2+\cdots+\lambda_{i}),\quad \widehat{\lambda}_i=[\lambda_i,1^{N_{i}}],\quad \kk^{\sharp}_{i+1}=\kk^\sharp_i+\lambda_i-1.
\end{equation}
The partitions $\widehat{\lambda}_i$ called hook-type \cite{CL22} because of the shape of their associated Young diagrams.
The following conjecture was proposed and checked in small ranks in \cite{CFLN}.
\begin{conjecture}[\cite{CFLN}]
    There is an isomorphism of vertex algebras
    \begin{align*}
       \W^\kk(\sll_N,\OO_{\lambda})\simeq \HH_{\OO_{\widehat{\lambda}_{\ell}}}^0\cdots \HH_{\OO_{\widehat{\lambda}_{1}}}^0 (\V^\kk(\sll_N)).
    \end{align*}
\end{conjecture}
As explained in detail in \emph{loc. cit.}, this statement is fundamentally important because it reveals a hidden structure of $\W$-algebras. Under certain restrictions on the levels, the isomorphism is expected to be preserved when considering simple quotients which
would lead in particular to a new proof of the rationality of the exceptional $\W$-algebras \cite{AvE19,  McR21}.

A similar conjecture can be formulated for the other classical Lie types using the analogue hook-type partitions \cite{CKL24, CL21}:
\begin{equation}
    [n^\star,1^{m}]
\end{equation}
with
\begin{equation}
    BD:\ \star=\begin{cases}
    1 & n \text{ odd}, \\
    2 & n \text{ even},
\end{cases} \qquad 
    C:\ \star=\begin{cases}
    2 & n \text{ odd}, \\
    1 & n \text{ even}.
\end{cases}
\end{equation}
We refer to the partitions of the form $[n^2,1^m]$ as \emph{thick} hook-types, which appear in order to respect the parity constraints on the partitions classifying nilpotent orbits in types $BCD$.
Collectively, the thick hook-type partitions in type $C$, and the hook-type partitions \cite{CL21}, are believed to form the building blocks of all $\W$-algebras in types $BCD$ \cite{CKL24}. 

Specifically, consider a partition for the nilpotent orbits for $\so_{N}$ or $\spp_{N}$,
\begin{equation}
    \lambda=(\lambda_1^{m_1} \geq \lambda_2^{m_2} \geq \cdots \geq \lambda_\ell^{m_\ell})
\end{equation}
where $m_i =1,2$ if $\lambda_i\equiv 0,1 $ for type $C$ and $\lambda_i\equiv 1,0 $ for type $BD$, respectively.
Generalizing the notations of \eqref{eq:hookA}, define  
\begin{equation}
    N_i=N-(m_1\lambda_1+m_2\lambda_2+\cdots+m_i\lambda_{i}),\quad \widehat{\lambda}_i^{m_i}=[\lambda_i^{m_i},1^{N_{i}}],
\end{equation}
we find inductively an affine vertex subalgebra 
\begin{equation}
  \begin{aligned}
    &\V^{\kk^{\sharp}_{i+1}}(\so_{N_i})\subset \HH_{\OO_{\widehat{\lambda}_{i}^{m_i}}}^0\cdots \HH_{\OO_{\widehat{\lambda}_{1}^{m_1}}}^0 (\V^\kk(\so_N)),\\
    &\V^{\kk^{\sharp}_{i+1}}(\spp_{N_i})\subset \HH_{\OO_{\widehat{\lambda}_{i}^{m_i}}}^0\cdots \HH_{\OO_{\widehat{\lambda}_{1}^{m_1}}}^0 (\V^\kk(\spp_N)).
\end{aligned}  
\end{equation}
We formulate a generalized version of a conjecture first mentioned in \cite{CKL24}.
\begin{conjecture}[\cite{CKL24}]\label{iterated conjecture for type BCD}
 For $\g=\so_{N}, \spp_{N}$, there is an isomorphism of vertex algebras
    \begin{align*}
       \W^\kk(\g,\OO_{\lambda})\simeq \HH_{\OO_{\widehat{\lambda}_{\ell}^{m_\ell}}}^0\cdots \HH_{\OO_{\widehat{\lambda}_{1}^{m_1}}}^0 (\V^\kk(\g)).
    \end{align*}
\end{conjecture}
\noindent
Note that in type $D$, when $N=4n$,
the reduction $\HH_{\OO_{[(2n)^2]}}^0$ is not ambiguous because 
$\W$-algebras corresponding to the nilpotent orbits $\OO_{[(2n)^2]}^{\mathrm{I}}$ and $\OO_{[(2n)^2]}^{\mathrm{II}}$ are isomorphic (Corollary \ref{coincidences of W-alg by Dynkin auto}).

\subsection{Virasoro-type reductions}
More partial reductions have been obtained in \cite{FFFN} by generalizing the previous construction.

Let $\OO$ be a nilpotent orbit of $\g$ and $f\in\g$ a representative of $\OO$ that satisfies
\begin{equation}
    \dim (\g^f\cap \h)=1,\quad \dim (\g^f\cap \n_+)=1.
\end{equation}
In addition, assume the non-zero semisimple element in $\g^f\cap \h$ acts on $\g^f\cap \n_+$ non-trivially.
Then the $\W$-algebra $\W^k(\g,\OO)$ contains a Heisenberg field $H$ corresponding to a non-zero vector of $\g^f\cap \h$ and a generator denoted $G^+$ corresponding to the vector generating $\g^f\cap \n_+$.
The field $G^+$ has conformal weight one if $\W^\kk(\g,\OO)$ has a vertex subalgebra isomorphic to $\V^{\kk^\sharp}(\sll_2)$ --- then we return to the situation described in \S\ref{sec: fundamental reduction}.
However, $G^+$ might also have a conformal weight greater than one. For instance, 
if $\W^\kk(\g,\OO)=\W(1,2^3,3,4^3\dots)$, $G^+$ has conformal weight two.

If the strong generator $G^+$ satisfies the OPE
\begin{align}
    G^+(z)G^+(w)\sim 0,
\end{align}
then it can be identified to the positive root vector in $\sll_2$.
Mimicking the construction of the principal $\W$-algebra of type $\sll_2$, a.k.a. the Virasoro vertex algebra, we introduce the BRST complex 
\begin{align}
    C_{\ydiagram{2}}^\bullet(\W^\kk(\g,\OO))=\W^\kk(\g,\OO)
    \otimes \bigwedge{}^{\bullet}_{\varphi,\varphi^*},
\end{align}
equipped with the differential 
\begin{align}\label{Virasoro reduction diffential}
    \fd{2}=\int Y((G^++1)\varphi^*,z)\,\mathrm{d}z.
\end{align}
We call it the \emph{Virasoro-type} reduction and denote its cohomology by 
\begin{align}
    \fHdeg{2}{\bullet}(\W^\kk(\g,\OO)).
\end{align}

Somehow, Virasoro type reduction is the ``smallest'' possible reduction that can be performed. 
It is expected to relate $\W$-algebras corresponding to adjacent nilpotent orbits.
\begin{conjecture}\label{conj:virasoro-type reductions}
    Assume $\W^\kk(\g,\OO)$ satisfies the previous conditions, then
    \begin{align*}
        \fHdeg{2}{0}(\W^\kk(\g,\OO))\simeq \W^\kk(\g,\widehat{\OO})
    \end{align*}
    where $\widehat{\OO}$ is a nilpotent orbit adjacent to $\OO$ in the Hesse diagram, that is \emph{a} {smallest nilpotent orbit which contains $\OO$ in the boundary}.
\end{conjecture}
Particular cases of the conjecture have been established previously for small ranks $\W$-algebras in type $A$ \cite{CFLN,FFFN,MR97}.
In the following, we prove the conjecture for families of examples in classical types.

\begin{remark}
    The construction can be generalized to apply partial reductions that mimic BRST reductions of affine vertex algebras of higher ranks --- 
    for instance, an example of \emph{Bershadsky--Polyakov-type} reduction is given in \cite{FFFN}.
\end{remark}

\subsubsection{Type $A$}\label{sec: main result in type A}
Consider the $\W$-algebra $\W^\kk(\sll_{2N},\OO_{[N^2]})$. 
The partition $[N^2]$ corresponds to 
the rectangular nilpotent orbit of height two.
Following \S\ref{sec: Nilpotent orbits}, one can associate the pyramid in Figure \ref{fig:pyramidA}
\begin{figure}[h]
		\begin{center}
			\begin{tikzpicture}[every node/.style={draw,regular polygon sides=4,minimum size=0.8cm,line width=0.04em},scale=0.8]
        \node at (2,1) (4) {\tiny 1};
        \node at (2,0) (4) {\tiny 2};
        \node at (1,1) (4) {\tiny 3};
        \node at (1,0) (4) {\tiny 4};
        \node at (0,1) (4) {\footnotesize $\dots$};
        \node at (0,0) (4) {\footnotesize $\dots$};
        \node at (-1,1) (4) {\tiny 2N-3};
        \node at (-1,0) (4) {\tiny 2N-2};
        \node at (-2,1) (4) {\tiny 2N-1};
        \node at (-2,0) (4) {\tiny 2N};
            \end{tikzpicture}
           \captionsetup{font=small}
            \captionof{figure}{Pyramid for $[N^2]$ in type $A$}\label{fig:pyramidA}
		\end{center}
\end{figure}
that gives the nilpotent element
\begin{equation}\label{nilpotentArec}
    f_{N^2}=\sum_{i=1}^{2N-2}\E_{i+2,i}
\end{equation}
as representative of the orbit and the following good {even} grading\footnote{A good $\frac{1}{2}\Z$-grading on $\g$ is said \emph{even} if $\g_j=0$ for $j\in\frac{1}{2}+\Z$.} 
\begin{align}\label{gradingArec}
\begin{tikzpicture}
{\dynkin[root
radius=.1cm,labels={\alpha_1,\alpha_2,\alpha_3, \alpha_4,{}, \alpha_{2N-2},\alpha_{2N-1}},labels*={0,1,0,1,0,1,0},edge length=0.8cm]A{oooo.ooo}};
\node at (-1,0) (4) {$\Gamma_{N^2}\colon$};
\end{tikzpicture}.
\end{align}
Then the $\W$-algebra associated to $\OO_{[N^2]}$ has strong generating type 
\begin{align}
    \W^\kk(\sll_{2N},\OO_{[N^2]})=\W(1^3,2^4,\dots,N^4).
\end{align}
The three weight-1 generators, which we denote by $G^+, J, G^-$, correspond to the vectors
\begin{align}
    e=\sum_{\begin{subarray}c
        1\leq i\leq N\\
        i \text{ odd}
    \end{subarray}}\E_{i,i+1},\quad 
    h=\sum_{\begin{subarray}c
        1\leq i\leq N\\
        i \text{ odd}
    \end{subarray}}\E_{i,i}-\E_{i+1,i+1},\quad 
    f=\sum_{\begin{subarray}c
        1\leq i\leq N\\
        i \text{ odd}
    \end{subarray}}\E_{i+1,i}
\end{align}
in the centralizer $\sll_{2N}^{f_{N^2}}$ and generate an affine subalgebra $\V^{\kk^\sharp}(\sll_2)$ with $\kk^\sharp=N(\kk+2N-2)$. 
The BRST reduction which reduces $\V^{\kk^\sharp}(\sll_2)$ to the Virasoro vertex algebra 
\begin{align}
    \W^{\kk^\sharp}(\sll_2,\OO_{[2]})=\fHdeg{2}{0}(\V^{\kk^\sharp}(\sll_2))
\end{align}
extends to the whole $\W$-algebra
\begin{align}
\fHdeg{2}{0}(\W^\kk(\sll_{2N},\OO_{[N^2]})):=\HH^0\left(\W^\kk(\sll_{2N},\OO_{[N^2]})\otimes \bigwedge{}^{\bullet}_{\varphi,\varphi^*},\fd{2}  \right)
\end{align}
with differential defined in \eqref{Virasoro reduction diffential}.
In \S\ref{sec: proof of main results}, it is proven that this vertex algebra identifies with another $\W$-algebra of type $A$.
\begin{theorem}\label{thm: A}
    For all level $\kk\in\C$, there exists an isomorphism of vertex algebras
    \begin{align*}
\fHdeg{2}{n}(\W^\kk(\sll_{2N},\OO_{[N^2]}))\simeq \indic{n=0}\W^\kk(\sll_{2N},\OO_{[N+1,N-1]}).
\end{align*}
\end{theorem}
\noindent
The case $N=2$ was proven by the first and third authors, with Z. Fehily and E. Fursman, in \cite{FFFN}.

\subsubsection{Type $BCD$}
Analogous statements for the other classical Lie types (i.e. type $BCD$) can be found by considering the $\W$-algebras
\begin{align}\label{list of W-algebras}
    \W^\kk(\so_{2N+1},\OO_{[N^2,1]}),\quad \W^\kk(\spp_{2N},\OO_{[N^2]}),\quad \W^\kk(\so_{2N},\OO_{[N^2]}).
\end{align}

Here the existence of an affine subalgebra $\V^{\kk^\sharp}(\sll_2)$ depends on the parity of $N$.
Indeed, the strong generating types of these $\W$-algebras with respect to the grading given by pyramids on Figures \ref{fig:pyramidBeven}--\ref{fig:pyramidDodd}
are
\begin{equation}\label{eq:gen types}
\begin{split}
    &\W^\kk(\so_{2N+1},\OO_{[N^2,1]})\\
    &\hspace{1cm}=
    \begin{cases}
    \W(1^3,2,3^3,4,\dots,(N-1)^3,N,(\tfrac{N+1}{2})^2) & (N\text{ even})\\
    \W(1,2^3,3,4^3,\dots,(N-2),(N-1)^3,N,(\tfrac{N+1}{2})^2) & (N\text{ odd}),
    \end{cases}\\
    &\W^\kk(\spp_{2N},\OO_{[N^2]})\\
    &\hspace{1cm}=
    \begin{cases}
    \W(1,2^3,3,4^3,\dots,(N-1),N^3) & (N\text{ even})\\
    \W(1^3,2,3^3,4,\dots,(N-2)^3,(N-1),N^3) & (N\text{ odd}),
    \end{cases}\\
    &\W^\kk(\so_{2N},\OO_{[N^2]})\\
    &\hspace{1cm}=
    \begin{cases}
    \W(1^3,2,3^3,4,\dots,(N-1)^3,N) & (N\text{ even})\\
    \W(1,2^3,3,4^3,\dots,(N-2),(N-1)^3,N) & (N\text{ odd}).
    \end{cases}
\end{split}
\end{equation}

Recall that a nilpotent element $f \in \g$ (or orbits $\OO$) is said \emph{distinguished} if $\g^f$ 
does not contain non-zero semisimple elements. 
In type $BCD$, $\OO_{\lambda}$ is distinguished if and only if the partition $\lambda$ has no repeated parts \cite{CM93}. 
Therefore, the orbits in \eqref{list of W-algebras} are not distinguished, and $\g^f\cap \h\neq0$, which implies $\W^\kk(\g,\OO)$ contains a Heisenberg vertex subalgebra.
It will be shown in \S\ref{sec: proof of main results} that the Heisenberg vertex subalgebra has rank one and that $\W^\kk(\g,\OO)$ also have a strong generator $G^+$ satisfying the framework of Conjecture \ref{conj:virasoro-type reductions}.
Hence, we can apply a Virasoro-type reduction.
The main result for type $BCD$ is the identification of $\fHdeg{2}{\bullet}(\W^\kk(\g,\OO))$ with another $\W$-algebra.
\begin{theorem}\label{thm: BCD}
    Let $\W^k(\g,\OO)$ be one of the $\W$-algebras listed in \eqref{list of W-algebras} and $\widehat{\OO}$ be the smallest nilpotent orbit containing $\OO$ in its boundary.
    For $\kk$ a generic level, there exists an isomorphism of vertex algebras
    \begin{align*}
        \fHdeg{2}{0}\left(\W^\kk(\g,\OO)\right)\simeq \W^\kk(\g,\widehat{\OO}).
    \end{align*}
    The explicit nilpotent orbit $\widehat{\OO}$ in the theorem is given in Table \ref{table: nilpotent orbits} below.
\end{theorem}
\begin{table}[h!]
	\centering
	\renewcommand{\arraystretch}{1.3}
	\begin{tabular}{|c||c|c|c|}
		\hline
		$(\g,\OO)$ & $(\so_{2N+1},\OO_{[N^2,1]})$ & $(\spp_{2N},\OO_{[N^2]})$ & $(\so_{2N},\OO_{[N^2]})$ \\
		\hline
             $\widehat{\OO}$ & $\begin{array}{c} \OO_{[N+1,N-1,1]}\\ \OO_{[N+2,N-2,1]} \end{array}$ & $\begin{array}{c} \OO_{[N+2,N-2]}\\ \OO_{[N+1,N-1]} \end{array}$ & $\begin{array}{c} \OO_{[N+1,N-1]}\\ \OO_{[N+2,N-2]} \end{array}$  \\ \hline
	\end{tabular}
 \captionsetup{font=small }
	\caption{Nilpotent orbits $\widehat{\OO}$. The first row corresponds to the cases $N$ even and the second to the cases $N$ odd.}\label{table: nilpotent orbits}
\end{table}

\begin{remark}
    When $\g=\so_{2N},\so_{2N+1}$ with $N$ even or $\g=\spp_{2N}$ with $N$ odd, the isomorphism in Theorem \ref{thm: BCD} holds for all levels and we have the cohomology vanishing $\fHdeg{2}{\neq0}\left(\W^\kk(\g,\OO)\right)=0$ by using the same continuity argument as in the proof of Theorem \ref{thm: A}.
    Although this is also expected in the remaining cases, the argument fails as we cannot directly use the exactness of the functor $\fHdeg{2}{\bullet}:\widehat{\mathcal{O}}^\kk_{\sll_2}\to\W^\kk(\sll_2)\mod$ \cite{Ara05}.
\end{remark}

\begin{remark} In Table \ref{table: nilpotent orbits}, the corresponding $\W$-algebras have the following strong generating type:
    \begin{equation}\label{eq:gen types reduced}
    \begin{split}
    &\W^\kk(\so_{2N+1},\OO_{[N+1,N-1,1]})=\W(2^3,3,4^3,5,\dots,(N-2)^3,N-1,\tfrac{N}{2},\tfrac{N+2}{2}),\\
    &\W^\kk(\so_{2N+1},\OO_{[N+2,N-2,1]})=\W\left( \begin{array}{ll}2^2,3,4^3,5,6^3,\dots,(N-4),(N-3)^3,\\ \hspace{1cm} N-2,(N-1)^2,N,N+1,\tfrac{N-1}{2},\tfrac{N+3}{2}\end{array} \right),\\
    &\W^\kk(\spp_{2N},\OO_{[N+2,N-2]})=\W( 2^2,3,4^3,5,6^3,\dots,N-3,(N-2)^3,N-1,N^2,N+2),\\
    &\W^\kk(\spp_{2N},\OO_{[N+1,N-1]})=\W( 2^3,3,4^3,5,\dots,N^3,N+1,N+2),\\
    &\W^\kk(\so_{2N},\OO_{[N+1,N-1]})=\W( 2^3,3,4^3,5,\dots,(N-2)^3,N-1,N^2),\\
    &\W^\kk(\so_{2N},\OO_{[N+2,N-2]})=\W( 2^2,3,4^3,5,\dots,N-2,(N-1)^2,N,N+1).
    \end{split}
\end{equation}
In particular, except $\OO_{[2+1,2-1,1]}$ and $\OO_{[3+2,3-2,1]}$ in type $B$, all the nilpotent orbits are distinguished.
\end{remark}

\section{Wakimoto realizations of \tW-algebras}\label{sec: Wakimoto}
In this section, we recall the Wakimoto realization of $\W$-algebras at generic levels $\kk$ (e.g.\ $\kk\in \C \backslash \Q$) \cite{Gen20}, which will be used in the proof of the main results.

\subsection{General form}\label{sec: Wakimoto general}
Let $G$ be a simple algebraic group (of adjoint type) whose Lie algebra is $\Lie(G)=\g$ and $N_+\subset G$ be the unipotent subgroup with $\Lie(N_+)=\n_+$.
Using the exponential map $\exp:\n_+\overset{\sim}{\longrightarrow}N_+$, one may identify the coordinate rings $\C[N_+]\simeq \C[z_\alpha\mid \alpha\in \Delta_+]$ where $z_\alpha$ is the coordinate of the root vector $e_\alpha\in\n_+$ ($\alpha \in \Delta_+$).
The left multiplication of $N_+$ on itself induces an anti-homomorphism $\rho:\n_+\to\D(N_+)$ where $\D(N_+)$ is the ring of differential operators on $N_+$. 
Then, one has 
\begin{equation}\label{eq:rho}
    \rho(e_{\alpha_i})=\sum_{\alpha\in \Delta_+}P_i^{\alpha}(z)\partial_{z_\alpha} 
\end{equation}
for some polynomials $P_i^{\alpha}(z)$ in $\C[N_+]$.

The affine vertex algebra $\V^\kk(\g)$ admits a free fields realization 
\begin{align}\label{Wakimoto realization for affine}
   \Psi\colon \V^\kk(\g)\hookrightarrow \bg^{\Delta_+}\otimes \heis.
\end{align}
using the Heisenberg vertex algebra $\heis$ associated with the Cartan subalgebra $\h$ at level $\kk+\hv$.
This embedding is known as the \emph{Wakimoto realization} (see e.g.\ \cite{Fre07}) 
In addition, it gives rise to a family of $\V^\kk(\g)$-modules  $\affWak{\lambda}:=\bg^{\Delta_+}\otimes \Fock{\lambda}$ ($\lambda\in \h^*$) obtained by restriction, called the \emph{Wakimoto representations}. Here $\Fock{\lambda}$ denotes the Fock module over $\heis$ of highest weight $\lambda$.

The affine vertex algebra $\V^\kk(\g)$ is isomorphic to the image of $\Psi$ which is equal to the intersection of the kernels of the screening operators 
\begin{align}\label{affine screenings}
    S_i=\int S_i(z)\ \dd z \colon \affWak{0} \rightarrow \affWak{-\alpha_i},\quad S_i(z)= Y(P_i\fockIO{i},z),\quad(1\leq i\leq l,\, l=\rk\g).
\end{align}
Here $P_i$ is the element in $\bg^{\Delta_+}$ obtained from \eqref{eq:rho} by replacing $z_\alpha$ with $\gamma_\alpha$ and $\partial_{z_\alpha}$ with $\beta_\alpha$ (and taking normally-ordered products when necessary).

Moreover, the embedding \eqref{Wakimoto realization for affine} is completed into an exact sequence
\begin{align}\label{Wakimoto resolution of affine}
    0\rightarrow  \V^\kk(\g) \overset{\Psi}{\longrightarrow} \affWak{0}\overset{\bigoplus S_i}{\longrightarrow} \bigoplus_{i=1,\ldots,l} \affWak{-\alpha_i}\rightarrow \mathrm{C}_2 \rightarrow \cdots \rightarrow \mathrm{C}_{n_\circ} \rightarrow 0
\end{align}
where
\begin{align}
    \mathrm{C}_i=\bigoplus_{\begin{subarray}c w\in W\\ \ell(w)=i\end{subarray}} \affWak{w\circ 0},
\end{align}
see for instance \cite[Proposition 4.2]{ACL19}. 
Here $W$ is the Weyl group of $\g$ and $\ell\colon W\rightarrow \Z_+$ the length function with ${n_\circ}$ the length of the longest element, and $\circ\colon W\times \h^*\rightarrow \h^*$ the dot action given by $w\circ\lambda=w(\lambda+\rho)-\rho$ with $\rho$ the Weyl vector.
The Wakimoto realizations of the $\W$-algebras are obtained by applying $\HH_{\OO}^0$ to \eqref{Wakimoto resolution of affine}.
Since the $\V^\kk(\g)$-modules $\affWak{\lambda}$ are $\HH_{\OO}$-acyclic (see e.g. \cite[Proposition 4.5]{Gen20}), we obtain the complex 
\begin{align}\label{Wakimoto resolution of W algebras}
    0\rightarrow  \W^\kk(\g,\OO) \rightarrow \HH_{\OO}^0(\affWak{0})\overset{\bigoplus S_i^\OO}{\longrightarrow} \bigoplus_{i=1,\ldots,l} \HH_{\OO}^0(\affWak{-\alpha_i})\rightarrow \cdots \rightarrow \HH_{\OO}^0(\mathrm{C}_N) \rightarrow 0,
\end{align}
which is exact by the cohomology vanishing $\HH_{\OO}^{\neq0}(\V^\kk(\g))=0$ \cite{Ara15a}.
The $\W^\kk(\g,\OO)$-modules $\HH_{\OO}^0(\affWak{\lambda})$ and the induced screening operators $S_i^\OO$ are explicitly obtained by using appropriate coordinates systems on $N_+$.

Setting $\g_0^+=\n_+\cap \g_0$, the subalgebra $\n_+$ decomposes as the semidirect product $\n_+= \g_0^+ \ltimes \g_+$. Correspondingly we have the group semidirect product $N_+= G_0^+ \ltimes G_+$.
Then by \cite[Proposition 4.5, Theorem 4.8]{Gen20}, we have
\begin{align}\label{Wakimoto for Walg}
    \HH_{\OO}^0(\affWak{\lambda})\simeq \affWak{\OO,\lambda}:=\bg^\bigstar \otimes \Phi(\g_{1/2}) \otimes \Fock{\lambda},\quad  \bigstar=\Delta^+_0=\{\alpha\in \Delta_+ ; \Gamma(\alpha)=0\}
\end{align}
and the screening operators $S_i^\OO$ are expressed as
\begin{equation}\label{screeinings for Walg}
    S_i^\OO=\int S_i^\OO(z)\ \dd z,\quad S_i^\OO(z):=Y(P_i^\OO\fockIO{i},z)
\end{equation}
where 
\begin{align}
    P_i^\OO=\begin{cases}
         \displaystyle{\sum_{\alpha \in {\Delta^+_{0}}}P_i^{\alpha}\beta_\alpha}& \text{if }\Gamma(\alpha_i)=0, \\
         \displaystyle{\sum_{\alpha \in {\Delta_{1/2}}}P_i^{\alpha}\Phi_\alpha}& \text{if }\Gamma(\alpha_i)=\frac{1}{2},\\
         \displaystyle{\sum_{\alpha \in \Delta_{+}}(f,e_\alpha)P_i^{\alpha}}& \text{if }\Gamma(\alpha_i)=1.
    \end{cases}
\end{align}
Roughly, they are obtained from $S_i$ by evaluating $\beta_\alpha$ as $\beta_\alpha$, $\Phi_\alpha$ and  $(f,e_\alpha)$ depending on the grading $\Gamma(\alpha)=0,1/2,1$.
Then, the Wakimoto realization of the $\W$-algebras is as follows.
\begin{theorem}[\cite{Gen20}]\label{thm: Wakimoto general form}
At generic levels $\kk$, there is an isomorphism of vertex algebras 
\begin{align*}
    \W^\kk(\g,\OO)\simeq \bigcap_{i=1,\dots,l} \ker  S_i^\OO \subset \bg^\bigstar\otimes \Phi(\g_{1/2})\otimes \heis.
\end{align*}
\end{theorem}
Note that the Wakimoto realization still exists for non-generic levels but the isomorphism in Theorem \ref{thm: Wakimoto general form} relaxed into an embedding.

It is obvious from the previous construction that 
different choices of good pairs $(f,\Gamma)$ (and pyramids) provide a variety of explicit realizations.
In the following, we fix convenient realization to prove the main results.

\subsection{Explicit forms}
We give explicit forms of Theorem \ref{thm: Wakimoto general form} for the $\W$-algebras appearing in Theorems \ref{thm: A} and \ref{thm: BCD}. 
Although the realizations are obtained straightforwardly, we sketch the proof in type $A$ for readers that would like to become more familiar with the general computations of screening operators. 

\subsubsection{Type A}
For the $\W$-algebra $\W^\kk(\sll_{2N},\OO_{[N^2]})$, we have the following Wakimoto realization by using the pyramid in Figure \ref{fig:pyramidA}.
\begin{proposition}\label{Wakimoto realization: type A}
For $(\g,\OO)=(\sll_{2N},\OO_{[N^2]})$ and $\kk$ a generic level, there is an isomorphism of vertex algebras
\begin{align*}
    \W^\kk(\g,\OO)\simeq \bigcap_{i=1}^{2N-1} \ker  S_i^\OO \subset \bg^\bigstar \otimes \heis
\end{align*}
with
\begin{align*}
P_{2i+1}^\OO=\beta_{2i+1},\quad 
P_{2i}^\OO= -\gamma_{2i-1}+\gamma_{2i+1},\qquad \bigstar=\{1,3,\dots,2N-1 \}.
\end{align*}    
\end{proposition}

\proof
The pyramid in Figure \ref{fig:pyramidA} gives the good pair $(f_{\OO},\Gamma_{\OO})$ 
defined in \eqref{nilpotentArec} and \eqref{gradingArec}:
\begin{align}\label{nilpotentgradingA}
   f_{\OO}=\sum_{i=1}^{2N-2}\E_{i+2,i},\qquad \Gamma_{\OO}(\alpha_{2i+1})=0,\quad \Gamma_{\OO}(\alpha_{2i})=1.
\end{align}
In particular, $\Gamma_{\OO}(\alpha)=0$ for $\alpha \in \Delta_+$ if and only if $\alpha$ is one of $\alpha_1,\alpha_3,\dots, \alpha_{2N-1}$. Hence $\bigstar=\{1,3,\dots,2N-1 \}$.
In order to derive $S_i^\OO$, we fix a coordinate system on $N_+$ identifying the coordinates $(z_\alpha; \alpha \in \Delta_+)$ on $\n_+$ of the root vectors $e_\alpha$ through the exponential maps $\e{\bullet}$ and the multiplication
\begin{align*}
    \n_+= \g_0^+ \ltimes \g_+ \xrightarrow[(\e{\bullet},\e{\bullet})]{\sim} G_0^+ \ltimes G_+ =N_+.
\end{align*}
By definition of $\rho(e_{\alpha_i})$, the $P_i^\alpha(z)$'s are given by the formula 
\begin{align*}
    \e{\epsilon e_{\alpha_i}}\e{\sum_0 z_\alpha e_\alpha}\e{\sum_+ z_\alpha e_\alpha}=\e{\sum_0 z_\alpha(\epsilon) e_\alpha}\e{\sum_+ z_\alpha(\epsilon) e_\alpha}
\end{align*}
with $z_\alpha(\epsilon)=z_\alpha + \epsilon P_i^\alpha(z)$
by introducing the dual number $\epsilon$ satisfying $\epsilon^2=0$.
Here the sums $\sum_0$ and $\sum_+$ run over $\alpha\in\Delta_0^+$ and $\alpha\in\Delta_{>0}$, respectively.

For $e_{\alpha_i}=\E_{i,i+1}$ with odd $i$, we have $\Gamma_\OO(e_{\alpha_i})=0$ and 
\begin{align*}
     \e{\epsilon e_{\alpha_i}}\e{\sum_0 z_\alpha e_\alpha}\e{\sum_+ z_\alpha e_\alpha}
     =\e{\sum_0 (z_\alpha+\delta_{\alpha,\alpha_i}\epsilon) e_\alpha}\e{\sum_+ z_\alpha e_\alpha}.
\end{align*}
It follows $\rho(e_{\alpha_i})=\partial_{\alpha_i}$ and thus $P_i^\OO=\beta_i$.

For $e_{\alpha_i}=\E_{i,i+1}$ with even $i$, we have $\Gamma_\OO(e_{\alpha_i})=1$ and 
\begin{align*}
     &\e{\epsilon e_{\alpha_i}}\e{\sum_0 z_\alpha e_\alpha}\e{\sum_+ z_\alpha e_\alpha}\\
     &=\e{\sum_0 z_\alpha e_\alpha} (\e{-\sum_0 z_\alpha e_\alpha}\e{\epsilon e_{\alpha_i}}\e{\sum_0 z_\alpha e_\alpha}) \e{\sum_+ z_\alpha e_\alpha}\\
     &=\e{\sum_0 z_\alpha e_\alpha} \e{\epsilon(e_{\alpha_{i}}-z_{\alpha_{i-1}}e_{\alpha_{i-1}+\alpha_i}+z_{\alpha_{i+1}}e_{\alpha_{i}+\alpha_{i+1}})} \e{\sum_+ z_\alpha e_\alpha}.
\end{align*}
where $e_{\alpha_{i-1}+\alpha_i}=\E_{i-1,i+1}$ and $e_{\alpha_{i}+\alpha_{i+1}}=\E_{i,i+2}$.
It follows that
\begin{align*}
\rho(e_{\alpha_i})=\partial_{z_{\alpha_i}}-z_{\alpha_{i-1}}\partial_{z_{\alpha_{i-1}+\alpha_i}}+z_{\alpha_{i+1}}\partial_{z_{\alpha_i+\alpha_{i+1}}}+\sum_{\Gamma(\alpha)>1} P_i^\alpha(z) \partial_{z_{\alpha}}.
\end{align*}
But $(f_\OO,e_\alpha)\neq0$ if and only if $\alpha=\alpha_j+\alpha_{j+1}$ ($1\leq j\leq 2N-2$) and $\Gamma(\alpha_j+\alpha_{j+1})=1$, thus $P_i^\OO=-\gamma_{i-1}+ \gamma_{i+1}$. 
\endproof

\subsubsection{Type B}
\begin{figure}[h]
	\begin{minipage}[l]{0.5\linewidth}
		\begin{center} 
			\begin{tikzpicture}[every node/.style={draw,regular polygon sides=4,minimum size=1cm,line width=0.04em},scale=0.58, transform shape]
       \node at (0,0) (4) { 0};
        \node at (0.5,1) (4) { N-1};
        \node at (1.5,1) (4) {\footnotesize $\dots$};
        \node at (2.5,1) (4) { 5};
        \node at (3.5,1) (4) { 3};
        \node at (4.5,1) (4) { 1};
        \node at (-0.5,1) (4) { -N};
        \node at (-1.5,1) (4) {\footnotesize $\dots$};
        \node at (-2.5,1) (4) { -6};
        \node at (-3.5,1) (4) { -4};
        \node at (-4.5,1) (4) { -2};
        \node at (0.5,-1) (4) { N};
        \node at (1.5,-1) (4) {\footnotesize $\dots$};
        \node at (2.5,-1) (4) { 6};
        \node at (3.5,-1) (4) { 4};
        \node at (4.5,-1) (4) { 2};
        \node at (-0.5,-1) (4) { 1-N};
        \node at (-1.5,-1) (4) {\footnotesize $\dots$ };
        \node at (-2.5,-1) (4) { -5};
        \node at (-3.5,-1) (4) { -3};
        \node at (-4.5,-1) (4) { -1};
    \end{tikzpicture}
            \captionsetup{font=small}
            \captionof{figure}{Pyramid for $[N^2,1]$, $N=2n$, in type $B$}\label{fig:pyramidBeven}
		\end{center}
	\end{minipage}
	\begin{minipage}[c]{0.49\linewidth}
		\begin{center}
			\begin{tikzpicture}[every node/.style={draw,regular polygon sides=4,minimum size=1cm,line width=0.04em},scale=0.58, transform shape]
        \node at (0,0) (4) { 0};
        \node at (0.5,1) (4) { N-1};
        \node at (1.5,1) (4) {\footnotesize $\dots$};
        \node at (2.5,1) (4) { 4};
        \node at (3.5,1) (4) { 2};
        \node at (4.5,1) (4) { 1};
        \node at (-0.5,1) (4) { -N};
        \node at (-1.5,1) (4) {\footnotesize $\dots$};
        \node at (-2.5,1) (4) { -5};
        \node at (-3.5,1) (4) { -3};
        \node at (-4.5,1) (4) {$\times $};
        \node at (0.5,-1) (4) { N};
        \node at (1.5,-1) (4) {\footnotesize $\dots$};
        \node at (2.5,-1) (4) { 5};
        \node at (3.5,-1) (4) { 3};
        \node at (4.5,-1) (4) { $\times$};
        \node at (-0.5,-1) (4) { 1-N};
        \node at (-1.5,-1) (4) {\footnotesize $\dots$ };
        \node at (-2.5,-1) (4) { -4};
        \node at (-3.5,-1) (4) { -2};
        \node at (-4.5,-1) (4) { -1};
    \end{tikzpicture}
            \captionsetup{font=small}
            \captionof{figure}{Pyramid for $[N^2,1]$, $N=2n+1$, in type $B$}\label{fig:pyramidBodd}
		\end{center}
	\end{minipage}
\end{figure}
For the $\W$-algebra $\W^\kk(\so_{2N+1},\OO_{[N^2,1]})$, we use the pyramids in Figures \ref{fig:pyramidBeven}-\ref{fig:pyramidBodd} depending on the parity of $N$.
They give respectively the nilpotent elements and good gradings
\begin{equation}\label{nilpotentgradingB}
\begin{aligned}
    &f_\OO=\E^o_{-N+1,N}+\sum_{i=1}^{N-2}\E^o_{i+2,i},\qquad &&f_\OO=\E^o_{2,1}+\E^o_{-N+1,N}+\sum_{i=2}^{N-2}\E^o_{i+2,i}.\\
    &\begin{tikzpicture}{\dynkin[root
radius=.1cm,labels={\alpha_1,\alpha_2,\alpha_{N-3}, \alpha_{N-2},\alpha_{N-1},\alpha_N},labels*={0,1,0,1,0,\tfrac{1}{2}},edge length=0.8cm]B{oo.oooo}};		
\end{tikzpicture}\qquad
&&\begin{tikzpicture}
{\dynkin[root
radius=.1cm,labels={\alpha_1,\alpha_2,\alpha_{N-2}, \alpha_{N-1},\alpha_N},labels*={1,0,1,0,\tfrac{1}{2}},edge length=0.8cm]B{oo.ooo}};		
\end{tikzpicture}
\end{aligned}
\end{equation}
Let $\Phi_{1/2}$ be a copy of the $\bg$-system generated by $\Phi_1, \Phi_2$ corresponding to $\beta, \gamma$, respectively. 
\begin{proposition}\label{Wakimoto realization: type B}
For $(\g,\OO)=(\so_{2N+1},\OO_{[N^2,1]})$ and $\kk$ a generic level, there is an isomorphism of vertex algebras
\begin{align*}
    \W^\kk(\g,\OO)\simeq \bigcap_{i=1}^{N} \ker  S_i^\OO \subset \bg^\bigstar \otimes \Phi_{1/2}  \otimes \heis
\end{align*}
with \\
\begin{enumerate}[label=(\roman*)]
    \item if $N=2n$,
\begin{align*}
P_{2i+1}^\OO=\beta_{2i+1},\quad 
P_{2i}^\OO=\begin{cases} -\gamma_{2i-1}+\gamma_{2i+1} &(i<n) \\ \Phi_{2}-\gamma_{N-1} \Phi_{1} &(i=n) \end{cases},\qquad \bigstar=\{1,3,\dots,N-1 \},
\end{align*}
    \item if $N=2n+1$,
\begin{align*}
P_{2i+1}^\OO=\begin{cases} \wun  & (i=0)\\ -\gamma_{2i}+\gamma_{2i+2} &(0<i<n) \\ \Phi_{2}-\gamma_{2n} \Phi_{1} &(i=n) \end{cases},\quad 
P_{2i}^\OO=\beta_{2i},\qquad \bigstar=\{2,4,\dots,N-1 \}.
\end{align*}
\end{enumerate}
\end{proposition}

\proof
Regardless the parity of $N$, we have
\begin{equation}
    \g_{1/2}=\C \E^o_{0,-N+1} \oplus \C \E^o_{0,-N}.
\end{equation}
Hence, $\Phi(\g_{1/2})$ is generated by the corresponding fields $\Phi_1,\Phi_2$.
Since 
\begin{align}
    (f_\OO,[\E^o_{0,-N+1}, \E^o_{0,-N}])=1,
\end{align}
we have $\Phi_1(z)\Phi_2(w)\sim \wun/(z-w)$ and thus $\Phi(\g_{1/2})\simeq \Phi_{1/2}$.
Hence, we have
\begin{align*}
    \W^\kk(\g,\OO)\simeq \bigcap_{i=1}^{N} \ker  S_i^\OO \subset \bg^\bigstar \otimes \Phi_{1/2}  \otimes \heis.
\end{align*}
The computations of $\bigstar$ and the $P_i$'s are parallel to the proof of Proposition \ref{Wakimoto realization: type A}, which we omit.
\endproof

\subsubsection{Type C}
\begin{figure}[h]
	\begin{minipage}[l]{0.5\linewidth}
		\begin{center} 
			\begin{tikzpicture}[every node/.style={draw,regular polygon sides=4,minimum size=1cm,line width=0.04em},scale=0.58, transform shape]
        \node at (0,0) (4) { -N};
        \node at (4,0) (4) { 3};
        \node at (3,0) (4) { 5};
        \node at (2,0) (4) { $\dots$};
        \node at (1,0) (4) { N-1};
        \node at (-5,0) (4) { -1};
        \node at (-4,0) (4) { -2};
        \node at (-3,0) (4) { -4};
        \node at (-2,0) (4) { $\dots$};
        \node at (-1,0) (4) { 2-N};
        \node at (0,1) (4) { N};
        \node at (5,1) (4) { 1};
        \node at (4,1) (4) { 2};
        \node at (3,1) (4) { 4};
        \node at (2,1) (4) { $\dots$};
        \node at (1,1) (4) { N-2};
        \node at (-4,1) (4) { -3};
        \node at (-3,1) (4) { -5};
        \node at (-2,1) (4) { $\dots$};
        \node at (-1,1) (4) {  1-N};
    \end{tikzpicture}
            \captionsetup{font=small}
            \captionof{figure}{Pyramid for $[N^2]$, $N=2n$, in type $C$}\label{fig:pyramidCeven}
		\end{center}
	\end{minipage}
	\begin{minipage}[c]{0.49\linewidth}
		\begin{center}
			\begin{tikzpicture}[every node/.style={draw,regular polygon sides=4,minimum size=1cm,line width=0.04em},scale=0.58, transform shape]
        \node at (4,1) (4) { 1};
        \node at (4,0) (4) { 2};
        \node at (3,1) (4) { 3};
        \node at (3,0) (4) { 4};
        \node at (2,1) (4) { $\dots$};
        \node at (2,0) (4) { $\dots$};
        \node at (1,1) (4) { N-2};
        \node at (1,0) (4) { N-1};
        \node at (0,1) (4) { N};
        \node at (0,0) (4) { -N};
        \node at (-1,1) (4) { 1-N};
        \node at (-1,0) (4) { 2-N};
        \node at (-2,1) (4) { $\dots$};
        \node at (-2,0) (4) { $\dots$};
        \node at (-3,1) (4) { -4};
        \node at (-3,0) (4) { -3};
        \node at (-4,1) (4) { -2};
        \node at (-4,0) (4) { -1};
    \end{tikzpicture}
            \captionsetup{font=small}
            \captionof{figure}{Pyramid for $[N^2]$, $N=2n+1$, in type $C$}\label{fig:pyramidCodd}
		\end{center}
	\end{minipage}
\end{figure}
For the $\W$-algebra $\W^\kk(\spp_{2N},\OO_{[N^2]})$, we use the pyramids in Figures \ref{fig:pyramidCeven}-\ref{fig:pyramidCodd} again depending on the parity of $N$.
They associate respectively the nilpotent elements and good gradings
\begin{equation}\label{nilpotentgradingC}
    \begin{aligned}
        &f_\OO=\E^s_{2,1}+\E^s_{-N,N+1}+\sum_{i=2}^{N-2}\E^s_{i+2,i},\qquad &&f_\OO=\E^s_{-N+1,N}+\sum_{i=1}^{N-2}\E^s_{i+2,i}.\\
        &\begin{tikzpicture}
{\dynkin[root
radius=.1cm,labels={\alpha_1,\alpha_2,\alpha_3, \alpha_4,\alpha_{N-1},\alpha_N},labels*={1,0,1,0,1,0},edge length=0.8cm]C{oooo.oo}};		
\end{tikzpicture}\qquad 
&&\begin{tikzpicture}
{\dynkin[root
radius=.1cm,labels={\alpha_1,\alpha_2,\alpha_3, \alpha_4,{}, \alpha_{N-1},\alpha_N},labels*={0,1,0,1,0,1,0},edge length=0.8cm]C{oooo.ooo}};		
\end{tikzpicture}
    \end{aligned}
\end{equation}

\begin{proposition}\label{Wakimoto realization: type C}
For $(\g,\OO)=(\spp_{2N},\OO_{[N^2]})$ and $\kk$ a generic level, there is an isomorphism of vertex algberas
\begin{align*}
    \W^\kk(\g,\OO)\simeq \bigcap_{i=1}^{N} \ker  S_i^\OO \subset \bg^\bigstar  \otimes \heis
\end{align*}
with
\begin{enumerate}[label=(\roman*)]
\item if $N=2n$
\begin{align*}
P_{2i+1}^\OO=\begin{cases} \wun  & (i=0)\\ -\gamma_{2i}+\gamma_{2i+2} &(i >0) \end{cases},\quad P_{2i}^\OO=\beta_{2i},\qquad \bigstar=\{2,4,\dots,N \},
\end{align*}
\item if $N=2n+1$
\begin{align*}
P_{2i+1}^\OO=\beta_{2i+1},\quad P_{2i}^\OO=-\gamma_{2i-1}+\gamma_{2i+1},\qquad \bigstar=\{1,3,\dots,N \}.
\end{align*}
\end{enumerate}
\end{proposition}

\subsubsection{Type D}
\begin{figure}[h]
	\begin{minipage}[l]{0.5\linewidth}
		\begin{center} 
			\begin{tikzpicture}[every node/.style={draw,regular polygon sides=4,minimum size=1cm,line width=0.04em},scale=0.58, transform shape]
        \node at (3.5,0) (4) { 2};
        \node at (2.5,0) (4) { 4};
        \node at (1.5,0) (4) { $\dots$};
        \node at (0.5,0) (4) { N};
        \node at (-3.5,0) (4) { -1};
        \node at (-2.5,0) (4) { -3};
        \node at (-1.5,0) (4) { $\dots$};
        \node at (-0.5,0) (4) { 1-N};
        \node at (3.5,1) (4) { 1};
        \node at (2.5,1) (4) { 3};
        \node at (1.5,1) (4) { $\dots$};
        \node at (0.5,1) (4) { N-1};
        \node at (-3.5,1) (4) { -2};
        \node at (-2.5,1) (4) { -4};
        \node at (-1.5,1) (4) { $\dots$};
        \node at (-0.5,1) (4) { -N};
    \end{tikzpicture}
            \captionsetup{font=small}
            \captionof{figure}{Pyramid for $[N^2]$, $N=2n$, in type $D$}\label{fig:pyramidDeven}
		\end{center}
	\end{minipage}
	\begin{minipage}[c]{0.49\linewidth}
		\begin{center}
			\begin{tikzpicture}[every node/.style={draw,regular polygon sides=4,minimum size=1cm,line width=0.04em},scale=0.58, transform shape]
        \node at (3.5,0) (4) { 3};
        \node at (2.5,0) (4) { 5};
        \node at (1.5,0) (4) { $\dots$};
        \node at (0.5,0) (4) { N};
        \node at (-4.5,0) (4) { -1};
        \node at (-3.5,0) (4) { -2};
        \node at (-2.5,0) (4) { -4};
        \node at (-1.5,0) (4) { $\dots$};
        \node at (-0.5,0) (4) { 1-N};
        \node at (4.5,1) (4) { 1};
        \node at (3.5,1) (4) { 2};
        \node at (2.5,1) (4) { 4};
        \node at (1.5,1) (4) { $\dots$};
        \node at (0.5,1) (4) { N-1};
        \node at (-3.5,1) (4) { -3};
        \node at (-2.5,1) (4) { -5};
        \node at (-1.5,1) (4) { $\dots$};
        \node at (-0.5,1) (4) { -N};
    \end{tikzpicture}
            \captionsetup{font=small}
            \captionof{figure}{Pyramid for $[N^2]$, $N=2n+1$, in type $D$}\label{fig:pyramidDodd}
		\end{center}
	\end{minipage}
\end{figure}
For the $\W$-algebra $\W^\kk(\so_{2N},\OO_{[N^2]})$, we use the pyramids in Figures \ref{fig:pyramidDeven}-\ref{fig:pyramidDodd} depending on the parity of $N$.
They associate the nilpotent elements and good gradings
\begin{equation}\label{nilpotentgradingD}
\begin{aligned}
    &f_\OO=\sum_{i=1}^{N-2}\E^o_{i+2,i}+\E^o_{-N+1,N},\qquad &&f_\OO=\E^o_{2,1}+\E^o_{-N,N-1}+\sum_{i=1}^{N-2}\E^o_{i+2,i}.\\
    &\begin{tikzpicture}
{\dynkin[root
radius=.1cm,labels={\alpha_1,\alpha_2,\alpha_3, \alpha_4,{},{},\alpha_{N-1},\alpha_N},labels*={0,1,0,1,0,1,0,1},edge length=0.8cm]D{oooo.oooo}};		
\end{tikzpicture}\qquad 
&&\begin{tikzpicture}
{\dynkin[root
radius=.1cm,labels={\alpha_1,\alpha_2,\alpha_3,{},{},\alpha_{N-1},\alpha_N},labels*={1,0,1,0,1,0,1},edge length=0.8cm]D{ooo.oooo}};		
\end{tikzpicture}
\end{aligned}
\end{equation}

\begin{proposition}\label{Wakimoto realization: type D}
Assume first that $\kk$ is a generic level.
For $(\g,\OO)=(\so_{2N},\OO_{[N^2]})$ and $\kk$ a generic level, there is an isomorphism 
\begin{align*}
    \W^\kk(\g,\OO)\simeq \bigcap_{i=1}^{N} \ker  S_i^\OO \subset \bg^\bigstar  \otimes \heis
\end{align*}
with \\
\textup{(i)} $N=2n$
\begin{align*}
P_{2i+1}^\OO=\beta_{2i+1},\quad P_{2i}^\OO=\begin{cases} -\gamma_{2i-1}+\gamma_{2i+1}  & (i<n)\\ \wun &(i=n) \end{cases} ,\qquad \bigstar=\{1,3,\dots,N-1 \},
\end{align*}
\textup{(ii)} $N=2n+1$
\begin{align*}
P_{2i+1}^\OO=\begin{cases} \wun & (i=0,n-1)\\ -\gamma_{2i}+\gamma_{2i+2} &(i\neq 0,n-1) \end{cases},\quad P_{2i}^\OO=\beta_{2i} ,\qquad \bigstar=\{2,4,\dots,N-1 \}.
\end{align*}
\end{proposition}

\section{Proof of Theorem \ref{thm: A} and \ref{thm: BCD}}\label{sec: proof of main results}
In this section, we prove Theorems \ref{thm: A} and \ref{thm: BCD} using the Wakimoto realizations given in \S\ref{sec: Wakimoto}.
We detail the proof of Theorem \ref{thm: A} and only sketch the proof of Theorem \ref{thm: BCD} as it is very similar.
{In all cases, we start assuming $\kk$ to be a generic level.}

\subsection{Type A} 
Setting $(\g,\OO)=(\sll_{2N},\OO_{[N^2]})$, we take the pyramid in Figure \ref{fig:pyramidA}
which associates the nilpotent element $f_\OO$ and the grading $\Gamma_\OO$ defined in \eqref{nilpotentgradingA}.
By Proposition \ref{Wakimoto realization: type A}, we have the Wakimoto realization 
\begin{equation}
        \W^\kk(\g,\OO)\simeq\bigcap_{i=1}^{2N-1}\ker  S_i^\OO\subset \bg^\bigstar\otimes \heis,\quad S_i^\OO=\scr{i}{P_i^\OO}
    \end{equation}
where 
\begin{align}\label{screeningsA}
P_{2i+1}^\OO=\beta_{2i+1},\quad 
P_{2i}^\OO= -\gamma_{2i-1}+\gamma_{2i+1},\qquad \bigstar=\{1,3,\dots,2N-1 \}.
\end{align}

Then it is straightforward to show that
the vector 
\begin{equation}
    \E_{1,2}+\E_{3,4}+\cdots +\E_{2N-2,2N}
\end{equation}
in the centralizer $\g^{f}$ 
gives rise to a strong generator $G^+$
realized as 
\begin{align}\label{eq:GptypeA}
    G^+=\beta_1+\beta_3+\dots+\beta_{2N-1}.
\end{align}
Note that we check the OPE $G^+(z)G^+(w)\sim0$ once again from this formula.

In order to realize $\fHdeg{2}{0}\left(\W^\kk(\g,\OO)\right)$, we apply the functor $\fHdeg{2}{0}$ to the exact sequence \eqref{Wakimoto resolution of W algebras} and consider the cohomologies of the Wakimoto representations
\begin{align}
    \fHdeg{2}{\bullet}(\HH_{\OO}^0(\affWak{\lambda}))\simeq \fHdeg{2}{\bullet}(\bg^\bigstar\otimes \Fock{\lambda}).
\end{align}
To perform the reduction on $\bg^\bigstar\otimes \Fock{\lambda}$, we gauge the free field realization \eqref{eq:GptypeA}. 
In order to do so, it is convenient, to first use
the isomorphism of vertex algebras
    \begin{equation}\label{isom_BG_A}
    \begin{gathered}
        \ff{2,4,\dots,2N}\xrightarrow{\sim}\ff{1,3,\dots,2N-1}\\
        \beta_{2i}\mapsto \begin{cases}
        -\gamma_{2i-1}+\gamma_{2i+1}\ &(i<N)\\
        -\gamma_{2N-1},\ &(i=N)
    \end{cases},\qquad
    \gamma_{2i}\mapsto\beta_1+\beta_3+\dots+\beta_{2i-1}.
    \end{gathered}
    \end{equation}
It transforms the differential $\fd{2}$ from \eqref{Virasoro reduction diffential} into 
\begin{align*}
    \fd{2}=\int Y(\gamma_{2N}+\wun)\varphi^*,z)\,\mathrm{d} z.
\end{align*}
and the coefficients \eqref{screeningsA} of the screening operators $S_i^\OO$  into 
\begin{align}\label{new screening in type A}
P_{2i+1}^\OO=\begin{cases}
    \gamma_{2} & (i=0)\\
    -\gamma_{2i}+\gamma_{2i+2} &(i>0)
 \end{cases},\qquad 
P_{2i}^\OO=\beta_{2i}.
\end{align}

Computing the cohomology explicitly, one finds that 
\begin{align*}
  \fHdeg{2}{p}(\bg^\bigstar\otimes \Fock{\lambda})
  &\simeq (\beta\gamma^{\bigstar_{c}}\otimes \Fock{\lambda})\otimes \fHdeg{2}{p}(\ff{2N})\\
  &\simeq \delta_{p,0}(\beta\gamma^{\bigstar_{c}}\otimes \Fock{\lambda})
\end{align*}
where $\bigstar_{c}=\{2,4,\dots,2N-2\}$.
In particular, we have the cohomology vanishing $ \fHdeg{2}{\neq 0}(\HH_{\OO}^0(\affWak{\lambda}))=0$ and thus the exact sequence \eqref{Wakimoto resolution of W algebras} induces a complex 
\begin{equation}
    \begin{aligned}
    0\rightarrow  \fHdeg{2}{0}(\W^\kk(\g,\OO))
    &\rightarrow \fHdeg{2}{0}(\HH_{\OO}^0(\affWak{0}))\\ &\overset{\bigoplus [S_i^\OO]}{\longrightarrow} \bigoplus_{i=1,\ldots,l} \fHdeg{2}{0}(\HH_{\OO}^0(\affWak{-\alpha_i}))
    \rightarrow \cdots \rightarrow \fHdeg{2}{0}(\HH_{\OO}^0(\mathrm{C}_N)) \rightarrow 0,
\end{aligned}
\end{equation}
whose cohomology computes $\fHdeg{2}{p}(\W^\kk(\g,\OO))$. Hence, $\fHdeg{2}{0}\left(\W^\kk(\g,\OO)\right)$ is realized by the reduction of the Wakimoto realization
\begin{align}\label{reduced Wakimoto for type A}
    \fHdeg{2}{0}(\W^\kk(\g,\OO))\simeq \bigcap_{i=1}^{2N-1}\ker [S_i^\OO]\subset \fHdeg{2}{0}(\bg^\bigstar\otimes \heis),\qquad
    [S_i^\OO]=\scr{i}{[P_i^\OO]}
\end{align}
where $[P_i^\OO]$ is the cohomology class of $P_i^\OO$ in \eqref{new screening in type A}, that is
\begin{align}\label{new screenings type A}
    [P_{2i+1}^\OO]=\begin{cases}
        \gamma_2 & (i=0)\\
        \gamma_{2i+2}-\gamma_{2i} & (0<i<N-1)\\
        -(\wun+\gamma_{2N-2}) & (i=N-1)
    \end{cases},
    \qquad 
    [P_{2i}^\OO]=\beta_{2i}.
\end{align}

The set of screening operators defined by \eqref{new screenings type A} are \emph{formally} realized as screening operators in Theorem \ref{thm: Wakimoto general form} by setting the nilpotent element $f_c$ and the grading $\Gamma_c$ to be
\begin{align}
    \begin{tikzpicture}
{\dynkin[root
radius=.1cm,labels={\alpha_1,\alpha_2,\alpha_3, \alpha_4,{}, \alpha_{2N-2},\alpha_{2N-1}},labels*={1,0,1,0,1,0,1},edge length=0.8cm]A{oooo.ooo}};
\node at (-4,0) (4) {$f_c=\E_{2N,2N-1}+\displaystyle{\sum_{i=1}^{2N-2}}\E_{i+2,i},$};
\end{tikzpicture}.
\end{align}
Note that we do \emph{not} claim the grading $\Gamma_c$ is good for $f_c$ at this moment.
\begin{figure}[h]
	\begin{minipage}[l]{0.5\linewidth}
		\begin{center} 
	\begin{tikzpicture}[every node/.style={draw,regular polygon sides=4,minimum size=0.8cm,line width=0.04em},scale=0.8]
        \node at (0,1) {$\dots$};
        \node at (1,1) {\tiny 4};
        \node at (2,1) {\tiny 2};
        \node at (-1,1) {\tiny 2N-4};
        \node at (-2,1) {\tiny 2N-2};
        \node at (0,0) {$\dots$};
        \node at (1,0) {\tiny 5};
        \node at (2,0) {\tiny 3};
        \node at (3,0) {\tiny 1};
        \node at (-1,0) {\tiny 2N-3};
        \node at (-2,0) {\tiny 2N-1};
        \node at (-3,0.5) {\tiny 2N};
    \end{tikzpicture}
            \captionsetup{font=small}
            \captionof{figure}{Generalized pyramid for $[N+1,N-1]$ in type $A$}\label{fig:generalizedpyramidA}
		\end{center}
	\end{minipage}
	\begin{minipage}[c]{0.49\linewidth}
		\begin{center}
			    \begin{tikzpicture}[every node/.style={draw,regular polygon sides=4,minimum size=0.8cm,line width=0.04em},scale=0.8]
        \node at (0,1) {$\dots$};
        \node at (1,1) {\tiny 4};
        \node at (2,1) {\tiny 2};
        \node at (-1,1) {\tiny 2N-4};
        \node at (-2,1) {\tiny 2N-2};
        \node at (0,0) {$\dots$};
        \node at (1,0) {\tiny 5};
        \node at (2,0) {\tiny 3};
        \node at (3,0) {\tiny 1};
        \node at (-1,0) {\tiny 2N-3};
        \node at (-2,0) {\tiny 2N-1};
        \node at (-3,0) {\tiny 2N};
    \end{tikzpicture}
            \captionsetup{font=small}
            \captionof{figure}{Pyramid for $[N+1,N-1]$ in type $A$}\label{new pyramid in type A}
		\end{center}
	\end{minipage}
\end{figure}
However, $f_c\in \OO_{[N+1,N-1]}$ as desired and the pair $(f_c,\Gamma_c)$ is realized by taking the (generalized) pyramid in Figure \ref{fig:generalizedpyramidA}
from which we read the grading by the $x$-axis as usual and the nilpotent element $f_c$ as 
$f_c=\sum \delta_{i\rightarrow j}\E_{i,j}$
where $\delta_{i\rightarrow j}$ is $1$ if the boxes {\tiny $\numtableaux{i}$} and {\tiny $\numtableaux{j}$} are adjacent and $0$ otherwise. 
The pyramid in Figure \ref{new pyramid in type A} is the 
closest honest pyramid in type $A$. It assigns the nilpotent element
\begin{align*}
f_{\widehat{\OO}}=\E_{2N,2N-1}+\sum_{i=1}^{2N-3}\E_{i+2,i} \in \OO_{[N+1,N-1]}
\end{align*}
and the same grading $\Gamma_c$. 
The nilpotent elements $f_c$ and $f_{\widehat{\OO}}$ are conjugate with each other $\mathrm{Ad}_g(f_c)=f_{\widehat{\OO}}$ by the grading-preserving adjoint action $\mathrm{Ad}_g\in\mathrm{SL}_{2N}$ with
    \begin{equation}
        g=
        \begin{pmatrix}
           1&  & & & \\
           & A & & & \\
           &  & \ddots& & \\
           &  & & A& \\
           &  & & & 1
        \end{pmatrix},\quad 
        A= \begin{pmatrix} 1 &0\\ 1 &1 \end{pmatrix}.
    \end{equation}
Therefore, the realization \eqref{reduced Wakimoto for type A} of $\fHdeg{2}{0}\left(\W^\kk(\g,\OO)\right)$ is equivalent to the Wakimoto realization of $\W^\kk(\g,\OO_{[N+1,N-1]})$ associated with the pyramid in Figure \eqref{new pyramid in type A}. Hence, for generic $\kk$,
\begin{align}\label{iso type A}
    \fHdeg{2}{0}\left(\W^\kk(\g,\OO)\right)\simeq\W^\kk(\g,\OO_{[N+1,N-1]}). 
\end{align}

We now generalize the result to all levels $\kk$. The free field realization of the $\W$-algebra $\W^\kk(\g,\OO)$ can be completed into a short exact sequence
\begin{equation}
    0\to\W^\kk(\g,\OO)\to \beta\gamma^{\bigstar}\otimes \Fock{\lambda}\to A_\kk\to0
\end{equation}
where $A_\kk$ identifies with the quotient $\beta\gamma^{\bigstar}\otimes \Fock{\lambda}/\W^\kk(\g,\OO)$.
Both $\W^\kk(\g,\OO)$ and $\beta\gamma^{\bigstar}\otimes \Fock{\lambda}$ are $\V^{\kk^\natural}(\sll_2)$-modules in the category $\widehat{\mathcal{O}}^\kk_{\sll_2}$ and so is $A_\kk$.
The functor $\fHdeg{2}{\bullet}:\widehat{\mathcal{O}}^\kk_{\sll_2}\to\W^\kk(\sll_2)\mod$ is exact \cite{Ara05} and thus we obtain the cohomology vanishing $\fHdeg{2}{\neq0}\left(\W^\kk(\g,\OO)\right)=0$ for all levels as well as the short exact sequence
\begin{equation}
    0\to\fHdeg{2}{0}\left(\W^\kk(\g,\OO)\right)\to \beta\gamma^{\bigstar_c}\otimes \Fock{\lambda} \to \fHdeg{2}{0}\left(A_\kk\right)\to0.
\end{equation}

The embeddings $\{\fHdeg{2}{0}\left(\W^\kk(\g,\OO)\right)\subset \beta\gamma^{\bigstar_c}\otimes \Fock{\lambda}\}_{\kk\in\C}$ form a continuous family characterized by the kernel of screening operators \eqref{reduced Wakimoto for type A} at generic levels.
On the other hand, the Wakimoto realization gives the continuous family
$\{\W^\kk(\g,\OO_{[N+1,N-1]})\subset\beta\gamma^{\bigstar_{c}}\otimes \Fock{\lambda}\}_{\kk\in\C}$ characterized by the same screenings at generic levels.
Hence, the isomorphism \eqref{iso type A} holds for all levels by continuity and we have the cohomology vanishing
\begin{equation}
    \fHdeg{2}{n}\left(\W^\kk(\g,\OO)\right)\simeq\indic{n=0}\W^\kk(\g,\OO_{[N+1,N-1]}). 
\end{equation}
This completes the proof.

\subsection{Type B} We prove Theorem \ref{thm: BCD} for type $B$, depending on the parity of $N$.
\subsubsection{Even case}
By setting $N=2n$ and considering the pyramid in Figure \ref{fig:pyramidBeven}
we have the Wakimoto realization of $\W^\kk(\g,\OO)$ for $(\g,\OO)=(\so_{2N+1},\OO_{[N^2,1]})$ given in Proposition \ref{Wakimoto realization: type B}.

Consider the nilpotent element to be as in \eqref{nilpotentgradingB}, 
$f_\OO=\E^o_{-N+1,N}+\sum_{i=1}^{N-2}\E^o_{i+2,i}$.
The intersection of its centralizer $\g^{f_\OO}$ with $\n_+$ is generated by the vector
\begin{equation}
    \E^o_{1,2}+\E^o_{3,4}+\dots + \E^o_{N-1,N}
\end{equation}
and the corresponding strong generator $G^+$ realized as 
\begin{align}
    G^+=\beta_1+\beta_3+\dots+\beta_{N-1}+\frac{1}{2}\Phi_1^2.
\end{align}

To proceed with the Virasoro-type reduction $\fHdeg{2}{0}$ on the Wakimoto realization of $\W^\kk(\g,\OO)$, we apply the isomorphism of vertex algebras
\begin{equation}
\begin{gathered}
     \ff{0,2,4,\dots,N}\xrightarrow{\sim}\bg^\bigstar  \otimes \Phi_{1/2},\\
     \beta_{2i}\mapsto \begin{cases} -\gamma_{1} & (i=0) \\ -\gamma_{2i-1}+\gamma_{2i+1} & (0<i<n)\\ \gamma_{2n-1}\Phi_1-\Phi_2 & (i=n) \end{cases},\quad
     \gamma_{2i}\mapsto \begin{cases}
        \beta_{2i+1}+\beta_{2i+3}+\dots+\beta_{2n-1}+\frac{1}{2}\Phi_1^2 & (i<n) \\ 
        \Phi_1 & (i=n).
    \end{cases}
\end{gathered}
\end{equation}
It transforms the differential $\fd{2}$ from \eqref{Virasoro reduction diffential} into 
\begin{align}
    \fd{2}=\int Y(\gamma_{0}+\wun)\varphi^*,z)\,\dz
\end{align}
and the coefficients of the screening operators $S_i^\OO$ appearing in Proposition \ref{Wakimoto realization: type B} into 
\begin{align}
P_{2i}^\OO=\beta_{2i},
    \qquad 
    P_{2i+1}^\OO=\begin{cases}
        (\gamma_{2i}-\gamma_{2i+2}) & (i<n-1)\\
        \gamma_{2n-2}-\tfrac{1}{2}\gamma_{2n}^2 & (i=n-1).
    \end{cases}
\end{align}
It follows that 
\begin{align*}
  \fHdeg{2}{p}(\bg^\bigstar\otimes \Phi_{1/2} \otimes \heis)
  &\simeq (\beta\gamma^{\bigstar_{c}}\otimes \heis)\otimes \fHdeg{2}{p}(\bg^{\{0\}})\\
  &\simeq \delta_{p,0}(\beta\gamma^{\bigstar_{c}}\otimes \heis)
\end{align*}
where $\bigstar_{c}=\{2,4,\dots,N\}$ and as in type $A$,  
we obtain 
\begin{align}
    \fHdeg{2}{0}(\W^\kk(\g,\OO))\simeq \bigcap_{i=1}^{N}\ker [S^\OO_i]\subset \bg^{\bigstar_c}\otimes \heis,\quad [S^\OO_i]=\scr{i}{[P_i]}.
\end{align}
where
\begin{align}
    [P_{2i}^\OO]=\beta_{2i},
    \qquad 
    [P_{2i+1}^\OO]=\begin{cases}
        -(\wun+\gamma_{2}) & (i=0)\\
        (\gamma_{2i}-\gamma_{2i+2}) & (0<i<n-1)\\
        \gamma_{2n-2}-\tfrac{1}{2}\gamma_{2n}^2 & (i=n-1).
    \end{cases}
\end{align}

Up to a rescaling automorphism of the $\beta\gamma$-systems, one finds that this set of screening operators appear in Theorem \ref{thm: Wakimoto general form} by formally taking the nilpotent element and the grading to be
\begin{align}
\begin{tikzpicture}
{\dynkin[root
radius=.1cm,labels={\alpha_1,\alpha_2,\alpha_3, {}, \alpha_{N-1},\alpha_N},labels*={1,0,1,0,1,0},edge length=0.8cm]B{oooo.oo}};
\node at (-4,0) (4) {$f_{c}=\E^o_{2,1}+2\E^o_{-N,N-1}+\displaystyle{\sum_{i=1}^{N-2}}\E^o_{i+2,i}$};
\end{tikzpicture}.
\end{align}
We check that $f_c\in \OO_{[N+1,N-1,1]}$ and that they form a good pair as desired since they are conjugate to the standard one associated with the honest pyramid in Figure \ref{new pyramid in type B even}.
\begin{figure}[h]
	\begin{minipage}[l]{0.5\linewidth}
		\begin{center} 
	\begin{tikzpicture}[every node/.style={draw,regular polygon sides=4,minimum size=1cm,line width=0.04em},scale=0.58, transform shape]
        \node at (0,0) (4) { 0};
        \node at (5,0) (4) { 1};
        \node at (4,0) (4) { 2};
        \node at (3,0) (4) { 4};
        \node at (2,0) (4) { $\dots$};
        \node at (1,0) (4) { N-2};
        \node at (-5,0) (4) { -1};
        \node at (-4,0) (4) { -2};
        \node at (-3,0) (4) { -4};
        \node at (-2,0) (4) { $\dots$};
        \node at (-1,0) (4) { 2-N};
        \node at (0,1) (4) { N};
        \node at (4,1) (4) { 3};
        \node at (3,1) (4) { 5};
        \node at (2,1) (4) { $\dots$};
        \node at (1,1) (4) { N-1};
        \node at (-4,1) (4) { $\times$};
        \node at (-3,1) (4) { $\times$};
        \node at (-2,1) (4) { $\dots$};
        \node at (-1,1) (4) { $\times$};
        \node at (0,-1) (4) { -N};
        \node at (4,-1) (4) { $\times$};
        \node at (3,-1) (4) { $\times$};
        \node at (2,-1) (4) { $\dots$};
        \node at (1,-1) (4) { $\times$};
        \node at (-4,-1) (4) { -3};
        \node at (-3,-1) (4) { -5};
        \node at (-2,-1) (4) { $\dots$};
        \node at (-1,-1) (4) { 1-N};
    \end{tikzpicture}
\captionsetup{font=small }
         \captionof{figure}{Pyramid for $[N+1,N-1,1]$, $N=2n$, in type $B$}\label{new pyramid in type B even}
		\end{center}
	\end{minipage}
	\begin{minipage}[c]{0.49\linewidth}
		\begin{center}
	\begin{tikzpicture}[every node/.style={draw,regular polygon sides=4,minimum size=1cm,line width=0.04em},scale=0.58, transform shape]
        \node at (0,0) (4) { 0};
        \node at (6,0) (4) { 1};
        \node at (5,0) (4) { 2};
        \node at (4,0) (4) { 3};
        \node at (3,0) (4) { 5};
        \node at (2,0) (4) { $\dots$};
        \node at (1,0) (4) { N-2};
        \node at (-6,0) (4) { -1};
        \node at (-5,0) (4) { -2};
        \node at (-4,0) (4) {-3};
        \node at (-3,0) (4) { -5};
        \node at (-2,0) (4) { $\dots$};
        \node at (-1,0) (4) { 2-N};
        \node at (0,1) (4) { N};
        \node at (4,1) (4) { 4};
        \node at (3,1) (4) { 6};
        \node at (2,1) (4) { $\dots$};
        \node at (1,1) (4) { N-1};
        \node at (-4,1) (4) { $\times$};
        \node at (-3,1) (4) { $\times$};
        \node at (-2,1) (4) { $\dots$};
        \node at (-1,1) (4) { $\times$};
        \node at (0,-1) (4) { -N};
        \node at (4,-1) (4) { $\times$};
        \node at (3,-1) (4) { $\times$};
        \node at (2,-1) (4) { $\dots$};
        \node at (1,-1) (4) { $\times$};
        \node at (-4,-1) (4) { -4};
        \node at (-3,-1) (4) { -6};
        \node at (-2,-1) (4) { $\dots$};
        \node at (-1,-1) (4) { 1-N};
    \end{tikzpicture}
\captionsetup{font=small}
         \captionof{figure}{Pyramid for $[N+2,N-2,1]$, $N=2n+1$, in type $B$}\label{new pyramid in type B odd}
        \end{center}
	\end{minipage}
\end{figure}

Indeed the pyramid in Figure \ref{new pyramid in type B even} gives the nilpotent element
\begin{align}
    f_{\widehat{\OO}}=\E^o_{2,1}+\E^o_{3,2}+\E^o_{N,N-1}+\E^o_{-N+1,N}+\sum_{i=3}^{N-3}\E^s_{i+2,i}
\end{align}
and the same grading.
Moreover $f_c$ and $f_{\widehat{\OO}}$ are conjugate through the grading-preserving adjoint action $\mathrm{Ad}_g\in \mathrm{SO}_{2N+1}$ with 
    \begin{equation}
        g=\left(\begin{array}{c|ccccc|cccccc}
        0&&&&&i&&&&&&\tfrac{-i}{2} \\ \hline
        &2i&&&&&&&&&& \\
        &&A&&&&&&&&& \\
        &&&\ddots&&&&&&&& \\
        &&&&A&&&&&&& \\
        \tfrac{-1}{\sqrt{2}}&&&&&\tfrac{-i}{\sqrt{2}}&&&&&&\tfrac{-i}{2\sqrt{2}} \\ \hline
        &&&&&&\tfrac{-i}{2}&&&&& \\
        &&&&&&&B&&&& \\
        &&&&&&&&&\ddots&& \\
        &&&&&&&&&&B& \\
        \tfrac{-1}{\sqrt{2}}&&&&&\tfrac{i}{\sqrt{2}}&&&&&&\tfrac{i}{2\sqrt{2}}
        \end{array}\right)
    \end{equation}
    where
    \begin{align}\label{small matrices}
         A=\left(\begin{array}{cc} i&i \\ \tfrac{-i}{\sqrt{2}} & \tfrac{i}{\sqrt{2}}\end{array}\right),\quad
          B=\left(\begin{array}{cc} \tfrac{-i}{2}& \tfrac{-i}{2} \\ \tfrac{i}{\sqrt{2}} & \tfrac{-i}{\sqrt{2}}\end{array}\right).
    \end{align}
This completes the proof.

\subsubsection{Odd case}
By setting $N=2n+1$ and considering the pyramid in Figure \ref{fig:pyramidBodd}
we have the Wakimoto realization of $\W^\kk(\g,\OO)$ for $(\g,\OO)=(\so_{2N+1},\OO_{[N^2,1]})$ again given in Proposition \ref{Wakimoto realization: type B}.

We pick up the nilpotent element to be 
$f_\OO=\E^o_{2,1}+\E^o_{-N+1,N}+\sum_{i=2}^{N-2}\E^o_{i+2,i}$,
as in \eqref{nilpotentgradingB}.
Then a non-zero vector in $\g^{f_\OO}\cap\n_+$ is
\begin{equation}
   \E^o_{2,3}+\E^o_{4,5}+\dots + \E^o_{N-1,N}, 
\end{equation}
and the corresponding strong generator $G^+$
realized as 
\begin{align}
    G^+=\beta_{2}+\beta_4+\dots+\beta_{N-1}+\frac{1}{2}\Phi_1^2.
\end{align}

We first apply the isomorphism of vertex algebras
\begin{equation}
\begin{gathered}
    \ff{1,3,\dots,N}\xrightarrow{\sim}\bg^\bigstar  \otimes \Phi_{1/2}\\
    \beta_{2i+1}\mapsto \begin{cases} -\gamma_{2} & (i=0) \\ -\gamma_{2i}+\gamma_{2i+2} & (0<i<n)\\ -\gamma_{2n}\Phi_1+\Phi_2 & (i=n) \end{cases},\quad
    \gamma_{2i+1}=\begin{cases}
        \beta_{2i+2}+\beta_{2i+4}+\dots+\beta_{2n}+\frac{1}{2}\Phi_1^2 & (i<n) \\ 
        \Phi_1 & (i=n), 
    \end{cases}
\end{gathered}    
\end{equation}
which transforms the differential $\fd{2}$ from \eqref{Virasoro reduction diffential} into 
\begin{align*}
   \fd{2}=\int Y(\gamma_{1}+\wun)\varphi^*,z)\,\dz
\end{align*}
and the coefficients of the screening operators $S_i^\OO$ in Proposition \ref{Wakimoto realization: type B} into 
\begin{align}
P_{2i+1}^\OO=\begin{cases}
        \wun & (i=0)\\
        \beta_{2i+1} & (i>0)
    \end{cases},
    \qquad 
    P_{2i}^\OO=\begin{cases}
        \gamma_{2i-1}-\gamma_{2i+1} & (i<n)\\
        \gamma_{2n-1}-\frac{1}{2}\gamma_{2n+1}^2 & (i=n).
    \end{cases}
\end{align}
Hence the Virasoro-type reduction $\fHdeg{2}{0}$ applied to the Wakimoto realization of $\W^\kk(\g,\OO)$ satisfies 
\begin{align*}
  \fHdeg{2}{p}(\bg^\bigstar\otimes \Phi_{1/2} \otimes \heis)
  &\simeq (\beta\gamma^{\bigstar_{c}}\otimes \heis)\otimes \fHdeg{2}{p}(\bg^{\{1\}})\\
  &\simeq \delta_{p,0}(\beta\gamma^{\bigstar_{c}}\otimes \heis)
\end{align*}
where $\bigstar_{c}=\{3,5,\dots,N-1\}$ and we obtain 
\begin{align}
    \fHdeg{2}{0}(\W^\kk(\g,\OO))\simeq \bigcap_{i=1}^{N}\ker [S^\OO_i]\subset \bg^{\bigstar_c}\otimes \heis,\quad [S^\OO_i]=\scr{i}{[P_i^\OO]}.
\end{align}
where
\begin{align}
    [P_{2i+1}^\OO]=\begin{cases}
        \wun & (i=0)\\
        \beta_{2i+1} & (i>0)
    \end{cases},
    \qquad 
    [P_{2i}^\OO]=\begin{cases}
        -(\wun+\gamma_{3}) & (i=1)\\
        \gamma_{2i-1}-\gamma_{2i+1} & (0<i<n)\\
        \gamma_{2n-1}-\frac{1}{2}\gamma_{2n+1}^2 & (i=n).
    \end{cases}
\end{align}

Up to a rescaling automorphism of the $\beta\gamma$-systems, this set of screening operators appears in Theorem \ref{thm: Wakimoto general form} by formally setting the nilpotent element and the grading to be
\begin{align}
\begin{tikzpicture}
{\dynkin[root
radius=.1cm,labels={\alpha_1,\alpha_2,\alpha_3, \alpha_4,{}, \alpha_{N-1},\alpha_N},labels*={1,1,0,1,0,1,0},edge length=0.8cm]B{ooooo.oo}};		
\node at (-4,0) (4) {$f_{c}=\E^o_{2,1}+\E^o_{3,2}+2\E^o_{-N,N-1}+\displaystyle{\sum_{i=1}^{N-2}}\E^o_{i+2,i}$};
\end{tikzpicture}.
\end{align}
Then we check that $f_c\in \OO_{[N+2,N-2,1]}$ and they form a good pair as desired since they are conjugate to the standard one associated with the pyramid in Figure \ref{new pyramid in type B odd}
which assigns the nilpotent element
\begin{align}
    f_{\widehat{\OO}}=\E^o_{2,1}+\E^o_{3,2}+\E^o_{N^2-1}+\E^o_{-N+1,N}+\sum_{i=3}^{N-3}\E^o_{i+2,i}
\end{align}
and the same grading, 
through the grading-preserving adjoint action $\mathrm{Ad}_g\in\mathrm{SO}_{2N+1}$ with 
\begin{equation}
     g=\left(\begin{array}{c|cccccc|ccccccc}
        0&&&&&&i&&&&&&&\tfrac{-i}{2} \\ \hline
        &2i&&&&&&&&&&&& \\
        &&2i&&&&&&&&&&& \\
        &&&A&&&&&&&&&& \\
        &&&&\ddots&&&&&&&&& \\
        &&&&&A&&&&&&&& \\
        \tfrac{-1}{\sqrt{2}}&&&&&&\tfrac{-i}{\sqrt{2}}&&&&&&&\tfrac{-i}{2\sqrt{2}} \\ \hline
        &&&&&&&\tfrac{-i}{2}&&&&&& \\
        &&&&&&&&\tfrac{-i}{2}&&&&& \\
        &&&&&&&&&B&&&& \\
        &&&&&&&&&&\ddots&&& \\
        &&&&&&&&&&&&B& \\
        \tfrac{-1}{\sqrt{2}}&&&&&&\tfrac{i}{\sqrt{2}}&&&&&&&\tfrac{i}{2\sqrt{2}}
        \end{array}\right)
\end{equation}
where $A,B$ are the matrices given in \eqref{small matrices}. 
This completes the proof.

\subsection{Type C} 
We prove Theorem \ref{thm: BCD} for type $C$, depending on the parity of $N$.
\subsubsection{Even case}
By setting $N=2n$ and considering the pyramid in Figure \ref{fig:pyramidCeven},
we have the Wakimoto realization of $\W^\kk(\g,\OO)$ for $(\g,\OO)=(\spp_{2N},\OO_{[N^2]})$ given in Proposition \ref{Wakimoto realization: type C}.

We pick the nilpotent element to be 
$f_\OO=\E^s_{2,1}+\E^s_{-N,N+1}+\sum_{i=2}^{N-2}\E^s_{i+2,i}$,
as in \eqref{nilpotentgradingC}.
Then the intersection $\g^{f_\OO}\cap\n_+$ is generated by the vector
\begin{equation}
    \E^s_{2,3}+\E^s_{4,5}+\dots + \E^s_{N-2,N-2},
\end{equation}
whose corresponding strong generator $G^+\in\W^\kk(\g,\OO)$ realized as 
\begin{align*}
    G^+=\beta_{2}+\beta_4+\dots+\beta_N
\end{align*}
in the Wakimoto realization.

We apply the isomorphism of vertex algebras
\begin{equation}
\begin{gathered}
    \beta\gamma^{\{3,5,\dots,N+1\}}\xrightarrow{\simeq}\beta\gamma^{\{2,4,\dots,N\}}\\
    \beta_{2i+1}\mapsto 
\begin{cases}
-\gamma_{2i}+\gamma_{2i+2} & (i<n)\\
-\gamma_{N} & (i=n),
\end{cases}
\quad \gamma_{2i+1}\mapsto \beta_2+\dots +\beta_{2i}.
\end{gathered}
\end{equation}
which transforms the differential $\fd{2}$ from \eqref{Virasoro reduction diffential} into 
\begin{align*}
    \fd{2}=\int Y(\gamma_{N+1}+\wun)\varphi^*,z)\dz
\end{align*}
and the coefficients of the screening operators $S_i^\OO$ appearing in Propositionn \ref{Wakimoto realization: type C} into 
\begin{align}
P_{2i+1}^\OO=\begin{cases} \wun  & (i=0)\\ \beta_{2i+1} &(i >0) \end{cases},\quad P_{2i}^\OO=\begin{cases} \gamma_3 & (i=1)\\  -\gamma_{2i-1}+\gamma_{2i+1} & (i>1) \end{cases}.
\end{align}

It follows that 
\begin{align}
  \fHdeg{2}{p}(\bg^\bigstar\otimes \heis)
  &\simeq (\beta\gamma^{\bigstar_{c}}\otimes \heis)\otimes \fHdeg{2}{p}(\ff{N+1})\\
  &\simeq \delta_{p,0}(\bg^{\bigstar_{c}}\otimes \heis)
\end{align}
where $\bigstar_{c}=\{3,5,\dots,N-1\}$ and therefore
\begin{align}
    \fHdeg{2}{0}(\W^\kk(\g,\OO))\simeq \bigcap_{i=1}^{N}\ker [S^\OO_i]\subset \bg^{\bigstar_c}\otimes \heis,\quad [S^\OO_i]=\scr{i}{[P_i^\OO]}.
\end{align}
where
\begin{align}\label{new coefficients screenings C even}
    [P_{2i+1}^\OO]=\begin{cases}
        \wun & (i=0)\\
        \beta_{2i+1} & (i>0),
    \end{cases}
    \qquad 
    [P_{2i}^\OO]=\begin{cases}
        (\gamma_{2i+3}-\gamma_{2i+1}) & (i<n)\\
        -(\wun+\gamma_{N-1}) & (i=n).
    \end{cases}
\end{align}

The set of screening operators with coefficicents are given by \eqref{new coefficients screenings C even} appears in Theorem \ref{thm: Wakimoto general form} by formally setting the nilpotent element and the grading to be
\begin{align}
\begin{tikzpicture}
{\dynkin[root
radius=.1cm,labels={\alpha_1,\alpha_2,\alpha_3, \alpha_4,\alpha_{N-1},\alpha_N},labels*={1,1,0,1,0,1},edge length=0.8cm]C{oooo.oo}};		
\node at (-4,0) (4) {$f_{c}=\E^s_{2,1}+\E^s_{-N,N}+\E^s_{-N,N-1}+\displaystyle{\sum_{i=2}^{N-2}}\E^s_{i+2,i}$};
\end{tikzpicture}.   
\end{align}
\begin{figure}[h]
	\begin{minipage}[l]{0.5\linewidth}
		\begin{center} 
	\begin{tikzpicture}[every node/.style={draw,regular polygon sides=4,minimum size=1cm,line width=0.04em},scale=0.58, transform shape]
        \node at (5.5,0) (4) { 1};
        \node at (4.5,0) (4) { 2};
        \node at (3.5,0) (4) { 4};
        \node at (2.5,0) (4) { 6};
        \node at (1.5,0) (4) { $\dots$};
        \node at (0.5,0) (4) { N};
        \node at (-5.5,0) (4) { -1};
        \node at (-4.5,0) (4) { -2};
        \node at (-3.5,0) (4) { -4};
        \node at (-2.5,0) (4) { -6};
        \node at (-1.5,0) (4) { $\dots$};
        \node at (-0.5,0) (4) { -N};
        \node at (3.5,1) (4) { 3};
        \node at (2.5,1) (4) { 5};
        \node at (1.5,1) (4) { $\dots$};
        \node at (0.5,1) (4) { N-1};
        \node at (-3.5,1) (4) { $\times$};
        \node at (-2.5,1) (4) { $\times$};
        \node at (-1.5,1) (4) { $\dots$};
        \node at (-0.5,1) (4) { $\times$};
        \node at (3.5,-1) (4) { $\times$};
        \node at (2.5,-1) (4) {$\times$};
        \node at (1.5,-1) (4) { $\dots$};
        \node at (0.5,-1) (4) { $\times$};
        \node at (-3.5,-1) (4) { -3};
        \node at (-2.5,-1) (4) { -5};
        \node at (-1.5,-1) (4) { $\dots$};
        \node at (-0.5,-1) (4) { 1-N};
    \end{tikzpicture}
\captionsetup{font=small}
         \captionof{figure}{Pyramid for $[N+2,N-2]$, $N=2n$, in type $C$}\label{new pyramid in type C even}
		\end{center}
	\end{minipage}
	\begin{minipage}[c]{0.49\linewidth}
		\begin{center}
  \begin{tikzpicture}[every node/.style={draw,regular polygon sides=4,minimum size=1cm,line width=0.04em},scale=0.58, transform shape]
        \node at (4.5,0) (4) { 1};
        \node at (3.5,0) (4) { 3};
        \node at (2.5,0) (4) { 5};
        \node at (1.5,0) (4) { $\dots$};
        \node at (0.5,0) (4) { N};
        \node at (-4.5,0) (4) { -1};
        \node at (-3.5,0) (4) { -3};
        \node at (-2.5,0) (4) { -5};
        \node at (-1.5,0) (4) { $\dots$};
        \node at (-0.5,0) (4) { -N};
        \node at (3.5,1) (4) { 2};
        \node at (2.5,1) (4) { 4};
        \node at (1.5,1) (4) { $\dots$};
        \node at (0.5,1) (4) { N-1};
        \node at (-3.5,1) (4) { $\times$};
        \node at (-2.5,1) (4) { $\times$};
        \node at (-1.5,1) (4) { $\dots$};
        \node at (-0.5,1) (4) { $\times$};
        \node at (3.5,-1) (4) { $\times$};
        \node at (2.5,-1) (4) { $\times$};
        \node at (1.5,-1) (4) { $\dots$};
        \node at (0.5,-1) (4) { $\times$};
        \node at (-3.5,-1) (4) { -2};
        \node at (-2.5,-1) (4) { -4};
        \node at (-1.5,-1) (4) { $\dots$};
        \node at (-0.5,-1) (4) { 1-N};
    \end{tikzpicture}
\captionsetup{font=small }
         \captionof{figure}{Pyramid for $[N+1,N-1]$, $N=2n+1$, in type $C$}\label{new pyramid in type C odd}
        \end{center}
	\end{minipage}
\end{figure}
Then we check that $f_c\in \OO_{[N+2,N-2]}$ and they form a good pair since they are conjugate to a standard good pair associated with the pyramid in Figure \ref{new pyramid in type C even} which assigns the nilpotent element
\begin{align}
f_{\widehat{\OO}}=\E^s_{2,1}+\E^s_{-N,N}+\E^s_{-N+1,N-1}+\sum_{i=2}^{N-2}\E^s_{i+2,i}
\end{align} and the same grading
through the grading-preserving adjoint action $\mathrm{Ad}_g\in \mathrm{Sp}_{2N}$ with 
    \begin{equation}
        g=\left(\begin{array}{c|c}
        \begin{matrix}
           1&  & & & & &\\
             &1& & & & &\\
             &  &-i& & & &\\
             &  &1 &1& & &\\
             &  & &  &\ddots& &\\
             &  & & & & -i&\\
             &  & & & &  1&1
        \end{matrix}&\\
        \hline
        &\begin{matrix}
            1&  & & & & &\\
             &1& & & & &\\
             &  &i&-i & & &\\
             &  & &1& & &\\
             &  & &  &\ddots& &\\
             &  & & & &i&-i\\
             &  & & & &  &1
        \end{matrix}\end{array}\right).
    \end{equation}
This completes the proof.

\subsubsection{Odd case}
By setting $N=2n+1$ and considering the pyramid in Figure \ref{fig:pyramidCodd}
we have the Wakimoto realization of $\W^\kk(\g,\OO)$ for $(\g,\OO)=(\spp_{2N},\OO_{[N^2]})$ given in Proposition \ref{Wakimoto realization: type C}
Since we pick the nilpotent element to be 
$f_\OO=\E^s_{-N+1,N}+\sum_{i=1}^{N-2}\E^s_{i+2,i}$, as in \eqref{nilpotentgradingC}
the strong generator $G^+$, which corresponds to the centralizing element 
\begin{equation}
    \E^s_{1,2}+\E^s_{3,4}+\dots + \E^s_{N-2,N-1},
\end{equation}
is realized as 
\begin{align}
    G^+=\beta_{1}+\beta_3+\dots+\beta_N.
\end{align}

To proceed with the Virasoro-type reduction $\fHdeg{2}{0}$ to this realization, we apply the isomorphism of vertex algebras
\begin{equation}
    \begin{gathered}
        \ff{2,4,\dots,N+1}\xrightarrow{\sim}\ff{1,3,\dots,N}\\
        \beta_{2i}\mapsto \begin{cases} -\gamma_{2i-1}+\gamma_{2i+1} & (i<n)\\ -\gamma_{N} & (i=n) \end{cases},\qquad
    \gamma_{2i}=\beta_1+\beta_3+\dots,\beta_{2i-1}.
    \end{gathered}
\end{equation}
It transforms the differential $\fd{2}$ from \eqref{Virasoro reduction diffential} into 
\begin{align}
 \fd{2}=\int Y(\gamma_{N+1}+\wun)\varphi^*,z)\,\dz
\end{align}
and the coefficients of the screening operators $S_i^\OO$ in Proposition \ref{Wakimoto realization: type C} into 
\begin{align}
P_{2i+1}^\OO=\begin{cases}
    \gamma_{2} & (i=0)\\
    -\gamma_{2i}+\gamma_{2i+2} &(i>0)
 \end{cases},\qquad 
P_{2i}^\OO=\beta_{2i}.
\end{align}

It follows that 
\begin{align*}
  \fHdeg{2}{p}(\bg^\bigstar\otimes \heis)
  &\simeq (\beta\gamma^{\bigstar_{c}}\otimes \heis)\otimes \fHdeg{2}{p}(\ff{N+1})\\
  &\simeq \delta_{p,0}(\bg^{\bigstar_{c}}\otimes \heis)
\end{align*}
where $\bigstar_{c}=\{3,5,\dots,N-1\}$ and  we obtain 
\begin{align}\label{new screenings C odd}
    \fHdeg{2}{0}(\W^\kk(\g,\OO))\simeq \bigcap_{i=1}^{N}\ker [S^\OO_i]\subset \bg^{\bigstar_c}\otimes \heis,\quad [S^\OO_i]=\scr{i}{[P_i^\OO]}.
\end{align}
where
\begin{align}
    [P_{2i+1}^\OO]=\begin{cases}
        \gamma_2 & (i=0)\\
        -\gamma_{2i}+\gamma_{2i+2} & (0<i<n)\\
        -(\wun+\gamma_{2n}) & (i=n)
    \end{cases}
    \qquad 
    [P_{2i}^\OO]=\beta_{2i}.
\end{align}

By Theorem \ref{thm: Wakimoto general form}, the realization \eqref{new coefficients screenings C even} coincides with the Wakimoto realization of the $\W$-algebra $\W^\kk(\g,\OO_{[N+1,N-1]})$ setting the nilpotent element and the grading to be
\begin{align}
\begin{tikzpicture}
{\dynkin[root
radius=.1cm,labels={\alpha_1,\alpha_2,\alpha_3, \alpha_4,{}, \alpha_{N-1},\alpha_N},labels*={1,0,1,0,1,0,1},edge length=0.8cm]C{oooo.ooo}};
\node at (-4,0) (4) {$f_{c}=\E^s_{-N,N}+\E^s_{-N,N-1}+\displaystyle{\sum_{i=2}^{N-2}}\E^s_{i+2,i}$};
\end{tikzpicture}.
\end{align}
They form a good pair since they are associated to the pyramid in Figure \ref{new pyramid in type C odd}.
This completes the proof.

\subsection{Type D} 
Finally, we prove Theorem \ref{thm: BCD} for type $D$, depending on the parity of $N$.
\subsubsection{Even case}
By setting $N=2n$ and considering the pyramid in Figure \ref{fig:pyramidDeven}
we obtain the Wakimoto realization of $\W^\kk(\g,\OO)$ for $(\g,\OO)=(\so_{2N},\OO_{[N^2]})$
given in Proposition \ref{Wakimoto realization: type D}.

Picking the nilpotent element to be 
$f_\OO=\E^o_{-N+1,N}+\sum_{i=1}^{N-2}\E^o_{i+2,i}$, as in \eqref{nilpotentgradingD},
the strong generator $G^+$, which corresponds to the vector
\begin{equation}
   \E^o_{1,2}+\E^o_{3,4}+\dots + \E^o_{N-1,N}, 
\end{equation}
in $\g^{f_\OO}\cap\n_+$,
is realized as 
\begin{align}
   G^+=\beta_{1}+\beta_3+\dots+\beta_{N-1}.
\end{align}

We apply the isomorphism of vertex algebras
\begin{equation}
    \begin{gathered}
        \ff{2,4,\dots,N} \xrightarrow{\sim} \ff{1,3,\dots,N-1}\\
        \beta_{2i}\mapsto 
\begin{cases}
-\gamma_{2i-1}+\gamma_{2i+1} & (i<n)\\
-\gamma_{N-1} & (i=n),
\end{cases}
\qquad
\gamma_{2i}\mapsto \beta_1+\beta_3+\dots+\beta_{2i-1},
    \end{gathered}
\end{equation}
before proceeding to the Virasoro-type reduction $\fHdeg{2}{0}$.
It transforms the differential $\fd{2}$ from \eqref{Virasoro reduction diffential} into 
\begin{align*}
 \fd{2}=\int Y(\gamma_{N}+\wun)\varphi^*,z)\,\dz
\end{align*}
and the coefficients of the screening operators $S_i^\OO$ in Proposition \ref{Wakimoto realization: type D} into 
\begin{align}
    P_{2i-1}^\OO=\begin{cases}
        \gamma_2 & (i=1)\\
        \gamma_{2i+2}-\gamma_{2i} & (i>1)
    \end{cases},
    \qquad 
    P_{2i}^\OO=\begin{cases}
        \beta_{2i} & (i<n)\\
        \wun & (i=n).
    \end{cases}
\end{align}

It follows that 
\begin{align*}
  \fHdeg{2}{p}(\bg^\bigstar\otimes \heis)
  &\simeq (\beta\gamma^{\bigstar_{c}}\otimes \heis)\otimes \fHdeg{2}{p}(\bg^{N+1})\\
  &\simeq \delta_{p,0}(\beta\gamma^{\bigstar_{c}}\otimes \heis)
\end{align*}
where $\bigstar_{c}=\{2,4,\dots,N-2\}$ and  
\begin{align}
    \fHdeg{2}{0}(\W^\kk(\g,\OO))\simeq \bigcap_{i=1}^{N}\ker [S^\OO_i]\subset \bg^{\bigstar_c}\otimes \heis,\quad [S^\OO_i]=\scr{i}{[P_i^\OO]}.
\end{align}
where
\begin{align}\label{new coefficients screenings D even}
    [P_{2i-1}^\OO]=\begin{cases}
        \gamma_2 & (i=1)\\
        \gamma_{2i+2}-\gamma_{2i} & (1<i<n)\\
        -(\wun+\gamma_{N-2}) & (i=n)\\
    \end{cases},
    \qquad 
    [P_{2i}^\OO]=\begin{cases}
        \beta_{2i} & (i<n)\\
        \wun & (i=n).
    \end{cases}
\end{align}
\begin{figure}[h]
	\begin{minipage}[l]{0.49\linewidth}
		\begin{center} 
	\begin{tikzpicture}[every node/.style={draw,regular polygon sides=4,minimum size=1cm,line width=0.04em},scale=0.58, transform shape]
        \node at (4,1) (4) { 1};
        \node at (4,0) (4) {$\times$};
        \node at (3,1) (4) { 3};
        \node at (3,0) (4) { 2};
        \node at (2,1) (4) { $\dots$};
        \node at (2,0) (4) { $\dots$};
        \node at (1,1) (4) { N-1};
        \node at (1,0) (4) { N-2};
        \node at (0,1) (4) { -N};
        \node at (0,0) (4) { N};
        \node at (-1,1) (4) { 2-N};
        \node at (-1,0) (4) { 1-N};
        \node at (-2,1) (4) { $\dots$};
        \node at (-2,0) (4) { $\dots$};
        \node at (-3,1) (4) { -2};
        \node at (-3,0) (4) { -3};
        \node at (-4,1) (4) {$\times$};
        \node at (-4,0) (4) { -1};
    \end{tikzpicture}
\captionsetup{font=small}
         \captionof{figure}{Pyramid for $[N+1,N-1]$, $N=2n$, in type $D$}\label{new pyramid in type D even}
		\end{center}
	\end{minipage}
	\begin{minipage}[c]{0.50\linewidth}
		\begin{center}
  \begin{tikzpicture}[every node/.style={draw,regular polygon sides=4,minimum size=1cm,line width=0.04em},scale=0.58, transform shape]
        \node at (6,1) (4) { 1};
        \node at (6,0) (4) {$\times$};
        \node at (5,1) (4) { 2};
        \node at (5,0) (4) {$\times$};
        \node at (4,1) (4) { 4};
        \node at (4,0) (4) { 3};
        \node at (3,1) (4) { 6};
        \node at (3,0) (4) { 5};
        \node at (2,1) (4) { $\dots$};
        \node at (2,0) (4) { $\dots$};
        \node at (1,1) (4) { N-1};
        \node at (1,0) (4) { N-2};
        \node at (0,1) (4) { -N};
        \node at (0,0) (4) { N};
        \node at (-1,1) (4) { 2-N};
        \node at (-1,0) (4) { 1-N};
        \node at (-2,1) (4) { $\dots$};
        \node at (-2,0) (4) { $\dots$};
        \node at (-3,1) (4) { -5};
        \node at (-3,0) (4) { -6};
        \node at (-4,1) (4) { -3};
        \node at (-4,0) (4) { -4};
        \node at (-5,1) (4) {$\times$ };
        \node at (-5,0) (4) { -2};
        \node at (-6,1) (4) {$\times$};
        \node at (-6,0) (4) { -1};
    \end{tikzpicture}
\captionsetup{font=small}
         \captionof{figure}{Pyramid for $[N+2,N-2]$, $N=2n+1$, in type $D$}\label{new pyramid in type D odd}
        \end{center}
	\end{minipage}
\end{figure}
By Theorem \ref{thm: Wakimoto general form}, the set of screening operators with coefficients \eqref{new coefficients screenings D even} coincides with the Wakimoto realization of $\W^\kk(\g,\OO_{[N+1,N-1]})$ by choosing the good pair 
\begin{align}
\begin{tikzpicture}
{\dynkin[root
radius=.1cm,labels={\alpha_1,\alpha_2,\alpha_3, \alpha_4,{},{},\alpha_{N-1},\alpha_N},labels*={1,0,1,0,1,0,1,1},edge length=0.8cm]D{oooo.oooo}};	
\node at (-4,0) (4) {$f_{c}=\E^o_{N,N-1}+\E^o_{-N+1,N}+\displaystyle{\sum_{i=2}^{N-2}}\E^o_{i+2,i}$};
\end{tikzpicture}
\end{align}
that are associated with the pyramid in Figure \ref{new pyramid in type D even}.
This completes the proof.

\subsubsection{Odd case}
By setting $N=2n+1$ and considering the pyramid in Figure \ref{fig:pyramidDodd}
we have the Wakimoto realization of $\W^\kk(\g,\OO)$ for $(\g,\OO)=(\so_{2N},\OO_{[N^2]})$ given in Proposition \ref{Wakimoto realization: type D}.

Picking the nilpotent element $f_\OO=\E^o_{2,1}+\E^o_{-N,N-1}+\sum_{i=1}^{N-2}\E^o_{i+2,i}$, as in \eqref{nilpotentgradingD},
the intersection of the centralizer $\g^{f_\OO}$ with $\n_+$ is generated by the vector
\begin{equation}
    \E^o_{2,4}+\E^o_{4,6}+\dots + \E^o_{N-1,N}.
\end{equation}
The corresponding strong generator $G^+$ realized as 
\begin{align*}
  G^+=\beta_{2}+\beta_4+\dots+\beta_{N-1}.
\end{align*}

To proceed with the Virasoro-type reduction $\fHdeg{2}{0}$ to this realization, we apply the isomorphism of vertex algebras
\begin{equation}
    \begin{gathered}
        \ff{3,5,\dots,N} \xrightarrow{\sim} \ff{2,4,\dots,N-1}\\
        \beta_{2i+1}\mapsto 
\begin{cases}
-\gamma_{2i}+\gamma_{2i+2} & (i<n)\\
-\gamma_{N-1} & (i=n),
\end{cases}
\qquad
\gamma_{2i+1}\mapsto \beta_2+\beta_4+\dots+\beta_{2i}.
    \end{gathered}
\end{equation}
It transforms the differential $\fd{2}$ from \eqref{Virasoro reduction diffential} into 
\begin{align*}
 \fd{2}=\int Y(\gamma_{N}+\wun)\varphi^*,z)\,\dz
\end{align*}
and the coefficients of the screening operators $S_i^\OO$ in Proposition \ref{Wakimoto realization: type D} into 
\begin{align}
    P_{2i+1}^\OO=\begin{cases}
        \wun & (i=1,n)\\
        \beta_{2i+1}& (1<i<n)
    \end{cases},
    \qquad 
    P_{2i}^\OO=\begin{cases}
        \gamma_{3} & (i=1)\\
        -\gamma_{2i-1}+\gamma_{2i+1} & (i>n).
    \end{cases}
\end{align}

It follows that 
\begin{align*}
  \fHdeg{2}{p}(\bg^\bigstar\otimes \heis)
  &\simeq (\beta\gamma^{\bigstar_{c}}\otimes \heis)\otimes \fHdeg{2}{p}(\ff{N})\\
  &\simeq \delta_{p,0}(\beta\gamma^{\bigstar_{c}}\otimes \heis)
\end{align*}
where $\bigstar_{c}=\{3,5,\dots,N-2\}$ and  
\begin{align}
    [P_{2i+1}^\OO]=\begin{cases}
        \wun & (i=1,n)\\
        \beta_{2i+1}& (1<i<n)
    \end{cases},
    \qquad 
    [P_{2i}^\OO]=\begin{cases}
        \gamma_{3} & (i=1)\\
        -\gamma_{2i-1}+\gamma_{2i+1} & (1<i<n)\\
        -(\wun+\gamma_{2n-1}) &(i=n).
    \end{cases}
\end{align}

As before, by Theorem \ref{thm: Wakimoto general form}, this set of screening operators coincides with the Wakimoto realitation of the $\W$-algebra $\W^\kk(\g,\OO_{[N+2,N-2]})$ corresponding to the good pair
\begin{align}
\begin{tikzpicture}
{\dynkin[root
radius=.1cm,labels={\alpha_1,\alpha_2,\alpha_3, \alpha_4,{},{},{},{},\alpha_N},labels*={1,1,0,1,0,1,0,1,1},edge length=0.8cm]D{ooooo.oooo}};		
\node at (-5,0) (4) {$f_{c}=\E^o_{2,1}+\E^o_{N,N-1}+\E^o_{-N+1,N}+\displaystyle{\sum_{i=2}^{N-2}}\E^o_{i+2,i}$,};
\end{tikzpicture}
\end{align}
associated with the pyramid in Figure \ref{new pyramid in type D odd}.
This completes the proof.

\section{Inverse Hamiltonian reductions} \label{sec: iHR}
Inverse Hamiltonian reductions are the embeddings between $\W$-algebras up to certain free field algebras. They have recently attracted a lot of interest for their applications to the representation theory (see among many others \cite{Ad19,ACG23,AKR21,FFFN,Fehily23,Fehily24}).
The guiding example is the embedding between $\V^\kk(\sll_2)$ and $\W^\kk(\sll_2,\oo{2})$ \cite{Sem94}:
\begin{align}\label{iHR for sl2}
    \V^\kk(\sll_2) \hookrightarrow \W^\kk(\sll_2,\oo{2}) \otimes \Pi.
\end{align}
Here $\Pi$ is the half-lattice vertex algebra realized as the subalgebra
\begin{align}
    \Pi:=\bigoplus_{n\in\Z} \pi^{x,y}_{n(x+y)}\subset V_{\Z x\oplus \Z y}
\end{align}
of the lattice vertex superalgebra $V_{\Z x\oplus \Z y}$ associated with the Laurentian lattice $\Z x\oplus \Z y$ equipped with the bilinear form $(\cdot,\cdot)$ satisfying
\begin{align}
    (x,x)=1,\quad  (y,y)=-1,\quad  (x,y)=0.  
\end{align}
The vertex algebra $\Pi$ localizes the $\beta\gamma$-system \cite{FMS86}:
\begin{align}\label{FMS bosonization}
    \beta\gamma \hookrightarrow \Pi,\quad \beta\mapsto \ee^{(x+y)},\ \gamma\mapsto -x\ee^{-(x+y)}
\end{align}
and the $\beta\gamma$-system realizes as the kernel of the screening operator $S_{\mathrm{FMS}}\colon \Pi\rightarrow V_{\Z x\oplus \Z y}$
\begin{align}\label{FMS screening}
\beta\gamma \simeq \ker S_{\mathrm{FMS}}, \qquad S_{\mathrm{FMS}}:=\int Y(e^x,z)\,\dz.
\end{align}

The embedding \eqref{iHR for sl2} can be obtained by comparing the screening operators of the Wakimoto realization on both sides:
\begin{align}
    \V^\kk(\sll_2)\simeq \ker S^{\oo{1^2}}\subset \beta\gamma \otimes \heis,\qquad  \W^\kk(\sll_2,\oo{2})\simeq \ker S^{\oo{2}}\subset \heis
\end{align}
where 
\begin{align}
    S^{\oo{1^2}}=\scr{1}{\beta},\quad  S^{\oo{2}}=\scr{1}{}.
\end{align}
Localizing $\bg$ into $\Pi$, we find a suitable automorphism on $\Pi\otimes \heis$ which identifies these two screening operators and thus we obtain the commutative diagram (see for instance \cite{Fehily24})
\begin{equation}
    \begin{tikzcd}[column sep = huge]
			\V^\kk(\sll_2)
			\arrow[d, hook] \arrow[r, hook]&
			\Pi \otimes \W^\kk(\sll_2,\oo{2})\arrow[d, hook]
			\\
			\Pi\otimes \heis \arrow[r, "\sim"]&
			\Pi\otimes \heis.
		\end{tikzcd}
\end{equation}
 
This trick is generalized to the Virasoro-type reductions in \S\ref{sec: main results}. 
\begin{theorem}\label{thm:IHR}
For all level $\kk\in\C$,
\begin{enumerate}[wide, labelindent=0pt]
    \item there exists an embedding of vertex algebras 
\begin{align*}
    \W^\kk(\sll_{2N},\oo{N^2})\hookrightarrow \Pi\otimes \W^\kk(\sll_{2N},\oo{N+1,N-1}).
\end{align*}
    \item for type $BCD$, there exists an embedding of vertex algebras 
\begin{align*}
    \W^\kk(\g,\OO)\hookrightarrow \Pi\otimes \W^\kk(\g,\widehat{\OO})
\end{align*}
where $(\g,\OO)$ and $\widehat{\OO}$ are like in Table \ref{table: nilpotent orbits}.
\end{enumerate}
\end{theorem}
\proof
For \textit{(1)}, assume first that $\kk$ is generic and start with the realization $\W^\kk(\g,\OO)=\W^\kk(\sll_{2N},\oo{N^2})$ obtained in the proof of Theorem \ref{thm: A} which uses \eqref{new screening in type A}:
\begin{align}
    \W^\kk(\g,\OO)\simeq\bigcap_{i=1}^{2N-1}\ker  S_i^\OO\subset \bg^\bigstar\otimes \heis,\quad S_i^\OO=\scr{i}{P_i^\OO}
\end{align}
where 
\begin{align}
P_{2i+1}^\OO=\begin{cases}
    \gamma_{2} & (i=0)\\
    -\gamma_{2i}+\gamma_{2i+2} &(i>0)
 \end{cases},\qquad 
P_{2i}^\OO=\beta_{2i},\quad \bigstar=\{2,4,\dots,2N\}.
\end{align}
We localize $\ff{2N}$ into $\Pi$ using the embedding
\begin{align}
    \beta_{2N}\mapsto -x \ee^{-(x+y)},\quad \gamma_{2N}\mapsto -\ee^{(x+y)}.
\end{align}
Applying the automorphism of $\bg^{\bigstar_c} \otimes \Pi \otimes \heis$ ($\bigstar_c=\bigstar\backslash \{2N\}$) given by
\begin{equation}
\begin{aligned}
    &\beta_{2i}\mapsto \beta_{2i}\,\ee^{(x+y)},
    \quad \gamma_{2i}\mapsto \gamma_{2i}\,\ee^{-(x+y)},\\
    &x\mapsto x{+\sum_{j=1}^{N-1}\beta_{2j}\,\gamma_{2j}}-\Lambda-\frac{1}{4N}(x+y),
    \quad y\mapsto y{-\sum_{j=1}^{N-1}\beta_{2j}\,\gamma_{2j}}+\Lambda+\frac{1}{4N}(x+y),\\
    &h_{a}\mapsto h_{a}+{(-1)^a}(\kk+\hv)(x+y), 
\end{aligned}
\end{equation}
where $\Lambda=\sum (-1)^a\varpi_a$ is the alternating sum of all the fundamental weights $\varpi_a$,
it follows that
\begin{align}
    \W^\kk(\g,\OO)\subset \bigcap_{i=1}^{2N-1}\ker  [S_i^\OO]\subset \bg^{\bigstar_c}\otimes \Pi \otimes \heis
\end{align}
where $[S_i^\OO]$ are the screening operators in \eqref{reduced Wakimoto for type A} acting on the component $\bg^{\bigstar_c}\otimes \heis$.
As $[S_i^\OO]$ characterize $\W^\kk(\sll_{2N},\oo{N+1,N-1})\subset \bg^{\bigstar_c}\otimes \heis$, we obtain 
\begin{align}\label{eq:embedding IHR type A}
    \W^\kk(\g,\OO)\subset \Pi \otimes \W^\kk(\sll_{2N},\oo{N+1,N-1}).
\end{align}

We use the same continuity argument as in \cite{FFFN} to extend the result to all level. The Wakimoto realization of $\W$-algebras admits an integral form \cite{Gen20} replacing the level $\kk$ to the polynomial ring $\C[\mathbf{k}]$ and the specializations $\mathbf{k}\mapsto \kk \in \C$ recover the usual realization at each level.
In particular, the embeddings $\{\W^\kk(\sll_{2N},\oo{N+1,N-1})\subset \bg^{\bigstar_c}\otimes \heis\}_{\kk\in\C}$ and $\{\W^\kk(\sll_{2N},\oo{N^2})\subset \bg^{\bigstar_c}\otimes \Pi \otimes \heis\}_{\kk\in\C}$
form continuous families characterized by the kernel of screening operators $[S_i^\OO]$ at generic levels.
Since the embedding \eqref{eq:embedding IHR type A} is obtained by the characterization by the same screening operators for generic levels, it extends to all levels by continuity.
This completes the proof.
The proof of \textit{(2)} is similar, we omit it.
\endproof

\section{More Virasoro-type reductions on type \tA} \label{sec: More on type A}
The $\W$-algebra $\W^\kk(\sll_{2N},\OO_{[N+1,N-1]})$ obtained as the Virasoro-type reduction of $\W^\kk(\sll_{2N},\OO_{[N^2]})$ again satisfies the framework of Conjecture \ref{conj:virasoro-type reductions}. Hence, one can apply $\fHdeg{2}{0}$ again and adapt the proof of Theorem \ref{thm: A} to show that $\fHdeg{2}{0}(\W^\kk(\sll_{2N},\OO_{[N+1,N-1]}))\simeq \W^\kk(\sll_{2N},\OO_{[N+2,N-2]})$.
In fact, the proof of Theorem \ref{thm: A} can be generalized broadly to the following family.
\begin{theorem}\label{thm:A_general}
\phantom{x}
\begin{enumerate}[wide,labelindent=0pt]
\item For generic levels $\kk$, there is an isomorphism of vertex algebras 
\begin{align*}
    \fHdeg{2}{0}(\W^\kk(\sll_{N+M},\oo{N,M})) \simeq \W^\kk(\sll_{N+M},\oo{N+1,M-1}).
\end{align*}
\item On the other hand, for all $\kk\in\C$, we have an embedding of vertex algebras 
\begin{align*}
    \W^\kk(\sll_{N+M},\oo{N,M}) \hookrightarrow \Pi \otimes \W^\kk(\sll_{N+M},\oo{N+1,M-1}).
\end{align*}
\end{enumerate}
\end{theorem}    

The proof is similar to the one of Theorem \ref{thm: A} and relies on the following Wakimoto realization of the $\W$-algebra $\W^\kk(\sll_{N+M},\OO{[N,M]})$.
\begin{proposition}\label{Wakimoto_typeAmore}
For $(\g,\OO)=(\sll_{N+M},\OO_{[N,M]})$ with $N\geq M$ and $\kk$ a generic level, there is an isomorphism
\begin{align*}
    \W^\kk(\g,\OO)\simeq \bigcap_{i=1}^{N+M-1} \ker  S_i^\OO \subset \bg^\bigstar \otimes \heis
\end{align*}
with $\bigstar=\{N-M+1,\ N-M+3,\dots,\ N+M-1\}$ and
\begin{align*}
&P_{i}^\OO=1,&&i=1,\dots,N-M,\\
&P_{N-M+2i-1}^\OO=\beta_{N-M+2i-1},&&i=1,\dots,M,\\
&P_{N-M+2i}^\OO= -\gamma_{N-M+2i-1}+\gamma_{N-M+2i+1},&&i=1,\dots,M-1.
\end{align*}    
\end{proposition}

\begin{figure}
    \begin{center}
	\begin{tikzpicture}[every node/.style={draw,regular polygon sides=4,minimum size=1.1cm,line width=0.04em},scale=1.1]
        \node at (0,1) {$\dots$};
        \node at (1,1) {\tiny N-M+2};
        \node at (2,1) {\tiny N-M+2};
        \node at (-1,1) {\tiny N+M-2};
        \node at (-2,1) {\tiny N+M};
        \node at (0,0) {$\dots$};
        \node at (1,0) {\tiny N-M+3};
        \node at (2,0) {\tiny N-M+1};
        \node at (3,0) {\tiny N-M};
        \node at (4,0) {$\dots$};
        \node at (5,0) {\tiny 2};
        \node at (6,0) {\tiny 1};
        \node at (-1,0) {\tiny N+M-3};
        \node at (-2,0) {\tiny N+M-1};
    \end{tikzpicture}
\captionsetup{font=small }
    \captionof{figure}{Pyramid for $[N,M]$ in type $A$}\label{more pyramid in type A}
    \end{center}
\end{figure}
\begin{proof}
The pyramid in Figure \ref{more pyramid in type A} gives the good pair $(f_{\OO},\Gamma_{\OO})$ such that 
\begin{align}\label{nilpotentgradingAmore}
    \begin{tikzpicture}
{\dynkin[root
radius=.1cm,labels={\alpha_1,\alpha_2,\alpha_{N-M}, {}, {},{},\alpha_{N+M-1}},labels*={1,1,1,0,1,1,0},edge length=0.8cm]A{oo.ooo.oo}};
\node at (-4,0) (4) {$f_{\OO}=\displaystyle{\sum_{i=1}^{N-M}\E_{i+1,i}+\sum_{i=N-M+1}^{N+M-2}\E_{i+2,i}},$};
\end{tikzpicture}
\end{align}
that is
\begin{equation}
    \Gamma_{\OO}(\alpha_{i})=\left\{\begin{aligned}
       &0, \quad \text{if } i=N-M+1,\ N-M+3,\dots,\ N+M-1,\\
       &1, \quad \text{otherwise}.\end{aligned}\right.
\end{equation}
Thus $\bigstar=\{N-M+1,\ N-M+3,\dots,\ N+M-1\}$ follows. 

We define a coordinate system on $N_+=G_0^+ \ltimes G_+$ as in the proof of Proposition \ref{Wakimoto realization: type A}.
For $e_{\alpha_i}=\E_{i,i+1}$ with $i=N-M-1+2j$, $j=1,\dots M$, we have $\Gamma_\OO(e_{\alpha_i})=0$ and the computation of $P_i^\alpha(z)$ is exactly the same as in the proof of Proposition \ref{Wakimoto realization: type A}. Thus $P_i^\OO=\beta_i$.
For $e_{\alpha_i}=\E_{i,i+1}$ with $i=1,2,\dots, N-M$, $\Gamma_\OO(e_{\alpha_i})=1$ and
\begin{align}
     \e{\epsilon e_{\alpha_i}}\e{\sum_0 z_\alpha e_\alpha}\e{\sum_+ z_\alpha e_\alpha}
     =\e{\sum_0 z_\alpha e_\alpha} \e{\epsilon(e_{\alpha_{i}}+z_{\alpha_{N-M+1}}\indic{i=N-M}e_{\alpha_{N-M}+\alpha_{N-M+1}})} \e{\sum_+ z_\alpha e_\alpha}
\end{align}
Since $(f_\OO, e_{\alpha})\neq0$ if and only if $\alpha=\alpha_i$ with $i=1,\dots, N-M$ or $\alpha=\alpha_i+\alpha_{i+1}$ with $i=N-M+1,\dots,N+M-2$, $P_i^\OO=\wun$.
Finally, for $e_{\alpha_i}=\E_{i,i+1}$ with $i=N-M+2j$, $j=1,\dots M-1$, we have $\Gamma_\OO(e_{\alpha_i})=1$. 
The result is the same as for the even case in the proof of Proposition \ref{Wakimoto realization: type A}.
The only subtlety is in the intermediate computation where an additional term appears in the middle exponential
\begin{align*}
     &\e{\epsilon e_{\alpha_i}}\e{\sum_0 z_\alpha e_\alpha}\e{\sum_+ z_\alpha e_\alpha}\\
     &=\e{\sum_0 z_\alpha e_\alpha} \e{\epsilon(e_{\alpha_{i}}-z_{\alpha_{i-1}}e_{\alpha_{i-1}+\alpha_i}+z_{\alpha_{i+1}}e_{\alpha_{i}+\alpha_{i+1}}-z_{\alpha_{i-1}}z_{\alpha_{i+1}}e_{\alpha_{i-1}+\alpha_{i}+\alpha_{i+1}})} \e{\sum_+ z_\alpha e_\alpha}.
\end{align*}
However this extra term does not play any role as $(f_{\OO},\cdot):\Delta_+\to\R$ annihilates (in particular) all positive roots of length higher or equal to $3$.
Thus $P_i^\OO=-\gamma_{i-1}+ \gamma_{i+1}$ as previously.     
\end{proof}

\begin{proof}[Proof of Theorem \ref{thm:A_general}]
We are in a position to prove Theorem \ref{thm:A_general}. 
For \textit{(1)}, the idea of the proof is similar to the one of Theorem \ref{thm: A}. 
Assuming $\kk$ generic, set $(\g,\OO)=(\sll_{N+M},\OO_{[N,M]})$ and consider the nilpotent element $f_\OO$ and the grading defined in \eqref{nilpotentgradingAmore}.
The intersection $\g^{f_\OO}\cap\n_+$ is generated by the vector
\begin{align}
    \E_{N-M+1,N-M+2}+\E_{N-M+3,N-M+4}+\dots+\E_{N+M-1,N+M}
\end{align}
and the corresponding strong generator $G^+$ realizes as
\begin{equation}
    G^+=\beta_{N-M+1}+\beta_{N-M+3}+\dots+\beta_{N+M-1}
\end{equation}
in the Wakimoto realization given in Proposition \ref{Wakimoto_typeAmore}.

In order to proceed the Virasoro-type reduction $\fHdeg{2}{0}$ explicitly on $\bg^\bigstar\otimes \Fock{\lambda}$, we first apply the isomorphism of vertex algebras inspired from \eqref{isom_BG_A}
\begin{equation}
    \begin{gathered}
    \ff{N-M+2,\ N-M+2,\dots,\ N+M}\xrightarrow{\sim}\ff{N-M+1,\ N-M+3,\dots,\ N+M-1}\\
    \beta_{N-M+2i}\mapsto \begin{cases}
        -\gamma_{N-M+2i-1}+\gamma_{N-M+2i+1}\, &(i<M)\\
        -\gamma_{N+M-1},\ &(i=M)
    \end{cases},\\
    \gamma_{N-M+2i}\mapsto\beta_{N-M+1}+\beta_{N-M+3}+\dots,\beta_{N-M+2i-1}.
    \end{gathered}
\end{equation}
It transforms the differential $\fd{2}$ from \eqref{Virasoro reduction diffential} into
\begin{equation}
    \fd{2}=\int Y((\gamma_{N+M}+1)\varphi^*,z)\,\mathrm{d}z
\end{equation}
and the coefficient of the screening operators $S_i^\OO$ in Proposition \ref{Wakimoto_typeAmore} into 
\begin{equation}\label{new_screening_Amore}
    \begin{aligned}
&P_{N-M+2i-1}^\OO=\begin{cases}
        \gamma_{N-M+2}\,&(i=1)\\
        -\gamma_{N-M+2i-2}+\gamma_{N-M+2i}\, &(i>1)
        \end{cases},
&P_{N-M+2i}^\OO= \beta_{N-M+2i}
    \end{aligned}
\end{equation}
and the first $N-M$ screenings remain unchanged.

Hence the Virasoro-type reduction $\fHdeg{2}{0}$ applied to the Wakimoto realization of $\W^\kk(\g,\OO)$ satisfies 
\begin{align*}
  \fHdeg{2}{p}(\bg^\bigstar\otimes \Phi_{1/2} \otimes \heis)
  &\simeq (\beta\gamma^{\bigstar_{c}}\otimes \heis)\otimes \fHdeg{2}{p}(\bg^{\{1\}})\\
  &\simeq \delta_{p,0}(\beta\gamma^{\bigstar_{c}}\otimes \heis)
\end{align*}
where $\bigstar_{c}=\{N-M+2,\ N-M+4,\dots,\ N+M-2\}$ and we obtain 
\begin{align}\label{reduced screenings Amore}
    \fHdeg{2}{0}(\W^\kk(\g,\OO))\simeq \bigcap_{i=1}^{N}\ker [S^\OO_i]\subset \bg^{\bigstar_c}\otimes \heis,\quad [S^\OO_i]=\scr{i}{[P_i^\OO]}.
\end{align}
where
\begin{equation}
\begin{aligned}
    &[P_i^\OO]=1\,(i=1,\dots,N-M),\\
    &[P_{N-M+2i-1}^\OO]=\begin{cases}
        \gamma_{N-M+2}\,&(i=1)\\
        -\gamma_{N-M+2i-2}+\gamma_{N-M+2i}\, &(1<i<M)\\
        -(\wun+\gamma_{N+M-2})\,&(i=M)
    \end{cases},\\
    \qquad 
    &[P_{N-M+2i}^\OO]=\beta_{N-M+2i}.
\end{aligned} 
\end{equation}

We show with the same argument as for Theorem \ref{thm: A} that this set of screenings can be obtained by Theorem \ref{Wakimoto for Walg} considering a pyramid similar to the one on Figure \ref{fig:generalizedpyramidA}. The associated nilpotent element and grading are conjugate to a good pair that can be obtained by reading a pyramid similar to the one on Figure \ref{new pyramid in type A}.
We omit the details.

For \textit{(2)}, we proceed as in the proof of Theorem \ref{thm:IHR}, localizing $\beta\gamma^{\{N+M\}}$ in the Wakimoto realization whose coefficients of the screening operators are given in \eqref{new_screening_Amore} and use the automorphism of 
$\bg^{\bigstar_c} \otimes \Pi \otimes \heis$ given by
\begin{equation}
\begin{aligned}
    &\beta_{N-M+2i}\mapsto \beta_{N-M+2i}\,\ee^{(x+y)},
    \quad \gamma_{N-M+2i}\mapsto \gamma_{N-M+2i}\,\ee^{-(x+y)},
    \\
    &x\mapsto x+\sum_{j=1}^{N-1}\beta_{N-M+2j}\,\gamma_{N-M+2j}-\Lambda+{\tfrac{N-M-1}{2(N+M)}(x+y)},
    \quad\\
    &y\mapsto y-\sum_{j=1}^{N-1}\beta_{N-M+2j}\,\gamma_{N-M+2j}+\Lambda-{\tfrac{N-M-1}{2(N+M)}(x+y)},
    \\
    &h_{a}\mapsto \begin{cases}
        h_{a},&(a=1,\dots,N-M)\\
        h_{a}+{(-1)^a}(\kk+\hv)(x+y), &(a=N-M+1,\dots,N+M-1)
    \end{cases}, 
\end{aligned}
\end{equation}
where $\Lambda=\sum(-1)^a\varpi_a$ is the alternating sum of the fundamental weights $\varpi_a$ for $N-M<a<N+M$,
in order to identify with the screening operators $[S^\OO_i]$ from \eqref{reduced screenings Amore}. 
The continuity argument generalizes the embedding to all levels and concludes the proof.
\end{proof}

\section{Virasoro-type reduction for modules} \label{sec: Virasoro reduction for modules} 
Let $\KL^\kk(\g)$ denote the Kazhdan--Lusztig category, i.e., the category of $\V^\kk(\g)$-modules which are bounded from below by the conformal grading and with finite-dimensional graded spaces.
In this section, we generalize the isomorphisms in Theorem \ref{thm: A} and \ref{thm: BCD} for modules in the Kazhdan--Lusztig category. 

We assume $\kk$ is generic through this section. Then $\KL^\kk(\g)$ is semisimple and the simple modules are the Weyl modules 
\begin{equation}
    \weyl_\lambda^\kk=U(\widehat{\g})\otimes_{U(\g[t]\oplus\C K)}L_\lambda,
\end{equation}
which are induced from the simple highest weight $\g$-modules $L_\lambda$ with dominant integral highest weights $\lambda \in P_+$.
By \cite{ACL19}, the Weyl modules $\weyl_\lambda^\kk$ admit resolutions by Wakimoto modules, generalizing the case $\lambda=0$ in \eqref{Wakimoto resolution of affine}, of the form
\begin{align}\label{Wakimoto resolution of Weyl modules}
    0\rightarrow  \weyl_{\lambda}^\kk \rightarrow \affWak{\lambda}\overset{\bigoplus S_{i,\lambda}}{\longrightarrow} \bigoplus_{i=1,\ldots,l} \affWak{s_i\circ \lambda}\rightarrow \mathrm{C}_2^\lambda \rightarrow \cdots \rightarrow \mathrm{C}_{n_\circ}^\lambda \rightarrow 0
\end{align}
with 
\begin{align}
    \mathrm{C}_i^\lambda=\bigoplus_{\begin{subarray}c w\in W\\ \ell(w)=i\end{subarray}} \affWak{w\circ \lambda}.
\end{align}
where $s_i\in W$ is the $i$-th simple Weyl reflection. 
They realize as an intersection of screening operators $S_{i,\lambda}$ generalizing \eqref{affine screenings}. 
They are defined as $S_{i,\lambda}=S_{i}[\lambda(h_i)+1]$ where we set
\begin{align}\label{product of screening operators}
    S_{i}[n]=\int_{\Upsilon} S_i(z_1)\dots S_{i}(z_{n})\ \dd z_1\dots \dd z_{n}\colon \affWak{\lambda}\rightarrow \affWak{\lambda-n \alpha_i}
\end{align}
for some local system $\Upsilon$ on the configuration space $Y_{n}=\{(z_1,\dots,z_{n})\mid z_p\neq z_q\}$.

The Wakimoto modules over the $\W$-algebra $\W^\kk(\g, \OO)$ are obtained by applying the BRST reduction to the Wakimoto modules $\affWak{\lambda}$ over $\V^\kk(\g)$
\begin{align}
    \affWak{\OO,\lambda}:=\bg^{\bigstar}\otimes \Phi(\g_{1/2}) \otimes \heis\simeq \HH^0_{\OO}(\affWak{\lambda})
\end{align}
 and define a linear map
\begin{align}
S^\OO_{i}[n]=\int_{\Upsilon} S_i^\OO(z_1)\dots S_{i}^\OO(z_{n})\ \dd z_1\dots \dd z_{n}\colon \affWak{\OO, \lambda}\rightarrow \affWak{\OO, \lambda-n \alpha_i}.
\end{align}
Applying the reduction $\HH^0_{\OO}$ to \eqref{product of screening operators} give the induced homomorphisms $[S_{i}[n]]$ that satisfy the following property.
\begin{proposition}\label{prop: screening for modules}
The following diagram commutes
\begin{center}
\begin{tikzcd}
\HH^0_{\OO}(\affWak{\lambda}) \arrow{rr}{[S_{i}[n]]} \arrow{d}{\simeq} && \HH^0_{\OO}(\affWak{\lambda-n \alpha_i}) \arrow{d}{\simeq}\\
\affWak{\OO,\lambda} \arrow{rr}{S^\OO_{i}[n]} && \affWak{\OO,\lambda-n \alpha_i}.
\end{tikzcd}
\end{center}
\end{proposition}

\begin{proof}
Recall from \eqref{BRST is semi-infinite} that the quantum Hamiltonian reduction $\HH^\bullet_{\OO}(M)$ agrees with the semi-infinite cohomology
\begin{align}
    \HH^\bullet_{\OO}(M)\simeq \HH^{\frac{\infty}{2}+\bullet}(\g_+(\!(z)\!);M\otimes \Phi^\chi(\g)).
\end{align}
Taking $M=\affWak{\lambda}$ and the good coordinates using the decomposition $N_+\simeq G_{0,+} \ltimes G_+$ as in \S\ref{sec: Wakimoto general}, we have 
\begin{align}
    \HH^\bullet_{\OO}(\affWak{\lambda})\simeq \HH^{\frac{\infty}{2}+\bullet}(\g_+(\!(z)\!); \bg^{\Delta_{>0}} \otimes \Phi^\chi(\g))\otimes (\bg^\bigstar\otimes \Fock{\lambda})
\end{align}
where $\bigstar=\Delta_0^+$.
The free field algebra $\bg^{\Delta_{>0}}$ is a $\g_+(\!(z)\!)$-module called a \emph{semi-regular bimodule} \cite{Ara14, Voro} whose bimodule structure is given by the homomorphism $\Psi\colon \V^0(\g_+) \rightarrow \bg^{\Delta_{>0}}$ together with the anti-homomorphism 
$\rho\colon \V^0(\g_+) \rightarrow \bg^{\Delta_{>0}}$
induced from the left and right multiplications of $G_+$ on itself, see \eqref{eq:rho}-\eqref{Wakimoto realization for affine}. 
By \cite{Ara14, ACL19}, there exists a vector space isomorphism 
\begin{align}
   \xi\colon \bg^{\Delta_{>0}} \otimes \Phi^\chi(\g) \xrightarrow{\sim} \bg^{\Delta_{>0}}\otimes \Phi^\chi(\g)
\end{align}
which intertwines the actions
\begin{align}\label{switching property}
\begin{split}
 &(\xi \circ \sigma)(u)= (u\otimes 1) \circ \xi,\\
 &\xi \circ \left(\rho(u)\otimes 1\right)=\left(\rho(u)\otimes 1+1\otimes F^\chi(u)\right)\circ \xi
\end{split}
\end{align}
for $u\in \g_+(\!(t)\!)$ where $\sigma$ is the coproduct $\sigma(u)=u \otimes 1+1\otimes u$.
Then $\xi$ induces an isomorphism  
\begin{equation}\label{isom: semi-infinite cohomology}
    \begin{aligned}
   \HH^{\frac{\infty}{2}}(\g_+(\!(z)\!);\bg^{\Delta_{>0}} \otimes \Phi^\chi(\g) )\xrightarrow{\sim} 
   \HH^{\frac{\infty}{2}}(\g_+(\!(z)\!);\bg^{\Delta_{>0}})\otimes \Phi^\chi(\g) \simeq \Phi^\chi(\g)\simeq \Phi(\g_{1/2}).
\end{aligned}
\end{equation}
The isomorphism \eqref{isom: semi-infinite cohomology} is indeed an isomorphism of vertex algebras 
since by \cite[Proposition 4.3]{Gen20}
\begin{align}\label{the second isom in the computation}
    \Phi(\g_{1/2})\xrightarrow{\sim}\HH^{\frac{\infty}{2}}(\g_+(\!(z)\!);\bg^{\Delta_{>0}} \otimes \Phi^\chi(\g)) ,\quad \Phi_\alpha \mapsto \Phi_\alpha+\sum_{\alpha'\in \Delta_{1/2}} \langle e_\alpha,e_{\alpha'} \rangle \gamma_{\alpha'}.
\end{align}

Moreover, we have an isomorphism of graded vector spaces
\begin{align}
    \C \xrightarrow{\sim}\HH^{\frac{\infty}{2}}(\g_+(\!(z)\!);\bg^{\Delta_{>0}})
\end{align}
where we take a trivial bigrading on the left-hand side, whereas on the right-hand side we set
\begin{equation}
    \begin{aligned}
    &\mathrm{bideg}(\beta_{\alpha})=(1,\alpha),\qquad \mathrm{bideg}(\gamma_{\alpha})=(0,-\alpha),\\
    &\mathrm{bideg}(\varphi_{\alpha})=(1,\alpha),\qquad \mathrm{bideg}(\varphi^*_{\alpha})=(0,-\alpha).
\end{aligned}
\end{equation}
Then, under the isomorphism 
\begin{equation}
    \HH_{\OO}(\affWak{\lambda})\simeq \HH^{\frac{\infty}{2}}(\g_+(\!(z)\!);\bg^{\Delta_{>0}}\otimes \Phi^\chi(\g)) \otimes (\bg^\bigstar\otimes \Fock{\lambda})
\end{equation}
we identify
\begin{equation}
    \begin{aligned}
    S_{i}[n]&=\int_{\Upsilon} \prod_{j=1}^n Y(F^\chi(P_i)\fockIO{i},z_j)\ \dd z_1\dots \dd z_{n}\\
    &=\int_{\Upsilon} \prod_{j=1}^n Y(P_i^\OO\fockIO{i},z_j)\ \dd z_1\dots \dd z_{n}\\
    &=\int_{\Upsilon} S^\OO_i(z_1)\dots S^\OO_{i}(z_{n})\ \dd z_1\dots\dd z_{n}
    =S^\OO_{i}[n]
    \end{aligned}
\end{equation}
through the isomorphism \eqref{the second isom in the computation}. This completes the proof. 
\end{proof}

Denote $S^\OO_{i,\lambda}=S^\OO_{i}[h_i(\lambda)+1]$ in the following.
Applying $\fHdeg{2}{0}$ to the exact sequence \eqref{Wakimoto resolution of Weyl modules}, we obtain 
\begin{align}\label{Wakimoto resolution of Wakimoto Weyl modules}
\begin{split}
     0\rightarrow  \fHdeg{2}{0}(\weyl_{\lambda}^\kk) \rightarrow \affWak{\OO,\lambda}\overset{\bigoplus S^\OO_{i,\lambda}}{\longrightarrow} \bigoplus_{i=1,\ldots,l} \affWak{\OO,s_i\circ \lambda}
    &\rightarrow \fHdeg{2}{0}(\mathrm{C}_2^\lambda)\\
    &\rightarrow \cdots \rightarrow \fHdeg{2}{0}(\mathrm{C}_{n_\circ}^\lambda) \rightarrow 0,
\end{split}
\end{align}
with 
\begin{align}
    \mathrm{C}_n^{\OO,\lambda}=\bigoplus_{\begin{subarray}c w\in W\\ \ell(w)=i\end{subarray}} \affWak{\OO,w\circ \lambda}.
\end{align}
This complex is again exact as $\fHdeg{2}{\neq0}(\weyl_{\lambda}^\kk)=0$ by \cite{Ara15a} and the following is clear.
\begin{corollary}\label{Free field realization of W-modules}
    For $\lambda\in P_+$ and $\kk$ a generic level, there is an isomorphism of $\W^\kk(\g,\OO)$-modules
    \begin{align*}
        \HH^n_{\OO}(\weyl_\lambda^\kk)\simeq \delta_{n,0}\bigcap_{i=1,\dots,l} \ker S^\OO_{i,\lambda} \subset \affWak{\OO,\lambda}.
    \end{align*}
\end{corollary}
\noindent
This is a generalization of Theorem \ref{thm: Wakimoto general form} proven in \cite{Gen20} and also some results in \cite{AF19, FFFN} for $\W$-algebras associated with regular nilpotent orbits or $\g=\sll_4$.

We have an analogue to Theorems \ref{thm: A}, \ref{thm: BCD} and \ref{thm:A_general} for modules in the Kazhdan--Lusztig category.
\begin{theorem}\label{Isom for the modules 1}
Let $\W^\kk(\g, \OO)$ be the $\W$-algebras in Theorems \ref{thm: A}, \ref{thm: BCD} and \ref{thm:A_general}.
For $\lambda\in P_+$ and $\kk$ a generic level, there is an isomorphism of $\W^\kk(\g, \widehat{\OO})$-modules
\begin{align*}
\fHdeg{2}{0}\left(\HH^0_\OO(\weyl_\lambda^\kk) \right)\simeq \HH^0_{\widehat{\OO}}(\weyl_\lambda^\kk).
\end{align*}
\end{theorem}
\begin{proof}
Thanks to Corollary \ref{Free field realization of W-modules}, the isomorphism follows from the identification of the screening operators appearing in the proofs of Theorem \ref{thm: A}, \ref{thm: BCD} and \ref{thm:A_general}.
Indeed, we apply $\fHdeg{2}{0}$ to \eqref{Wakimoto resolution of Wakimoto Weyl modules}, we obtain 
\begin{equation}
\fHdeg{2}{0}(\HH^0_\OO(\weyl_\lambda^\kk)) \simeq \bigcap_{i=1,\dots,l} \ker [S^\OO_{i,\lambda}] \subset \fHdeg{2}{0}(\affWak{\OO,\lambda}).
\end{equation}
as $\fHdeg{2}{\neq0}(\affWak{\OO,\mu})=0$ holds for all $\mu \in \h^*$.
As in the proof of Proposition \ref{prop: screening for modules}, we identify the screening operators after applying $\fHdeg{2}{0}$
so that 
\begin{equation}
    \begin{split}
        \bigcap_{i=1,\dots,l} \ker \longstick{[S^\OO_{i,\lambda}]}{\fHdeg{2}{0}(\affWak{\OO,\lambda})}\simeq \bigcap_{i=1,\dots,l} \ker \longstick{S^{\widehat{\OO}}_{i,\lambda}}{\affWak{\widehat{\OO},\lambda}}\simeq \HH^0_{\widehat{\OO}}(\weyl_\lambda^\kk).
    \end{split}
\end{equation}
This completes the proof.
\end{proof}

\section{Virasoro-type reduction for the universal 2-parameter $\W_\infty$-algebra $\Wsp(c,\kk)$} 
\label{sec: Wsp}

Part of the $\W$-algebras $\W^\kk(\g,\OO)$ appearing in the Theorem \ref{thm: BCD}
can be obtained as quotients of the universal $\W_\infty$-algebra $\Wsp(c,\kk)$ constructed in \cite{CKL24}. This vertex algebra is of free strong generating type
\begin{equation}\W(1^3, 2, 3^3, 4, 5^3, 6,\dots)\end{equation}\label{sp2starting} 
defined and free over a localization $R$ of the polynomial ring $\C[c,\kk]$ (see \cite[\S 6]{CKL24}).
In the following, we denote the strong generators of the prescribed conformal weights $1^3, 2, 3^3, 4, 5^3,\dots$ by 
\begin{align}
    W_2, W_4, W_6, W_8,\dots,\qquad e_1, h_1, f_1, E_3, H_3, F_3, E_5, H_5, F_5,\dots
\end{align}
where weights are given by the sub-index. They satisfy the following properties:
\begin{itemize}
    \item[(P1)] $e_1, h_1, f_1$ generate the affine vertex subalgebra $\V^\kk(\sll_2)$;
    \item[(P2)] $W_2$ is a Virasoro element of central charge $c$  commuting with $e_1, h_1, f_1$ such that the sum 
    \begin{align}
        \mathbf{L}=W_2+L,\quad L_{\sll_2}:=\frac{1}{2(\kk+2)}(e_1f_1+f_1e_1+\tfrac{1}{2}h_1h_1)
    \end{align}
    is the conformal vector of $\Wsp(c,\kk)$;
    \item[(P3)] $W_i$ commutes with $e_1, h_1, f_1$ for all $i$;
    \item[(P4)] $E_3, H_3, F_3$ are the basis of the lowest weight subspace of the Weyl module over $\V^\kk(\sll_2)$ associated with the adjoint representation $L(\alpha)$ of $\sll_2$ and are primary vectors for the Virasoro element $W_2$;
    \item[(P5)] the elements $E_i, H_i, F_i$ ($i\geq 3$) satisfy the recursive relations
    \[E_{i+2}=W_{4}{}_{(1)}E_{i},\quad H_{i+2}=W_{4}{}_{(1)}H_{i},\quad F_{i+2}=W_{4}{}_{(1)}F_{i}.\]
\end{itemize}

By using the affine subalgebra $\V^\kk(\sll_2)\subset \Wsp(c,\kk)$, one may perform the Virsoro-type reduction. Namely, one considers the following complex:
\begin{equation}
    C_{\ydiagram{2}}^\bullet(\Wsp(c,\kk))=\Wsp(c,\kk) \otimes \bigwedge{}^{\bullet}_{\varphi,\varphi^*}, \qquad
    \fd{2}=\int Y((e_1+\wun)\varphi^*,z)\,\mathrm{d}z.
\end{equation}
In this section, we study this reduction and prove the following result:
\begin{theorem}\label{thm: reduction_Wsp}
    The Virasoro-type reduction $\fHdeg{2}{0}(\Wsp(c,\kk))$ of the universal $\W_\infty$- algebra $\Wsp(c,\kk)$ is a simple freely generated vertex algebra, free over $R$, of generating type $$\W(2^3,3,4^3,5,6^3,\dots).$$ 
\end{theorem}

\begin{proof}
As $\Wsp(c,\kk)$ is free over $R$ lying in the Kazhdan-Lusztig category as a $\V^\kk(\sll_2)$-module, one has the cohomology vanishing $\fHdeg{2}{\neq0}(\Wsp(c,\kk))=0$ and $\fHdeg{2}{0}(\Wsp(c,\kk))$ is free over $R$. 
Indeed, one can show these properties and determine a set of strong generators following \cite{FB04, KW04}.

We first take the base change to the fraction fields $Q(R)$ so that we can apply the arguments in \cite{FB04, KW04} directly.
After the base change $\otimes_RQ(R)$, $\Wsp(c,\kk)$ becomes completely reducible as a $\V^\kk(\sll_2)$-module and thus admits a decomposition 
\begin{align}\label{decomposition for Wsp}
   \Wsp(c,\kk)_{Q(R)}:= \Wsp(c,\kk)\otimes_RQ(R)\simeq \bigoplus_{n\geq0}C^\kk(n\alpha)\underset{Q(R)}{\otimes} \mathbb{V}^\kk_{n\alpha}
\end{align}
where $\mathbb{V}^\kk_{n\alpha}$ is the Weyl module over $\V^\kk(\sll_2)$ (and thus over $Q(R)$) of highest weight $n\alpha$ and $C^\kk(n)$ is the corresponding multiplicity space consisting of highest weight vectors. 
Thanks to the base change, the results in \cite[\S 6]{KW04} (for the base field $\C$) hold; we have $\fHdeg{2}{\neq0}(\Wsp(c,\kk)_{Q(R)})=0$ and
\begin{align}\label{decomposition at generic level}
    \fHdeg{2}{0}(\Wsp(c,\kk)_{Q(R)})\simeq \bigoplus_{n\geq0}C^\kk(n\alpha)\underset{Q(R)}{\otimes} \fHdeg{2}{0}(\mathbb{V}^\kk_{n\alpha}).
\end{align}
In the component
\begin{align}
    C^\kk(0)\underset{Q(R)}{\otimes} \fHdeg{2}{0}(\mathbb{V}^\kk_{0})=C^\kk(0)\underset{Q(R)}{\otimes}\W^\kk(\sll_2,\oo{2})
\end{align}
we have the obvious strong generators $W_2, W_4, W_6, W_8,\dots$ (thanks to (P3)) together with a new Virasoro element in $\W^\kk(\sll_2,\oo{2})$, that is,
\begin{equation}
    \tilde{L}=L+\frac{1}{2}\partial h_1+(\partial\varphi)\varphi^*
\end{equation}
of central charge $-\frac{(2 \kk+1) (3 \kk+4)}{(\kk+2)}$. We set the conformal vector of $\fHdeg{2}{0}(\Wsp(c,\kk)_{Q(R)})$ to be 
\[\tilde{\mathbf{L}}=W_2+\tilde{L}.\]

Next, $E_3, H_3, F_3$ lie in $C^\kk(\alpha)\otimes L(\alpha)$ by (P4) and so are all $E_i, H_i, F_i$ thanks to (P3) and (P5).
We introduce the following corrections of $E_i, H_i, F_i$
\begin{equation}\label{eq:deformed fields}
\begin{split}
\tilde E_{i-1} = E_{i},\quad 
\tilde H_{i} =H_{i}+\mathbf{h}_1E_{i},\quad 
\tilde F_{i+1}=F_{i}-\frac{1}{2}\mathbf{h}_1H_{i}-\frac{1}{4} (\mathbf{h}_1^2-2 \partial \mathbf{h}_1) E_{i}
\end{split}
\end{equation}
where we set $\mathbf{h}_1=h_1+2 \varphi \varphi^*$.
Then by \cite[\S 4, 6]{KW04}, they define non-zero cohomology classes in $\fHdeg{2}{0}(\Wsp(c,\kk))$.
These generators have conformal weights $i-1$, $i$, $i+1$ with respect to $\tilde{\mathbf{L}}$, as indicated by the sub-index.

Now, notice that the above-constructed generators
\begin{align}\label{list of strong generators}
    \tilde{L}, W_2, \tilde{E}_2,\quad \tilde{H}_3,\quad \tilde{F}_4, W_4, \tilde{E}_4,\quad \tilde{H}_5,\quad \tilde{F}_6, W_6, \tilde{E}_6,\dots.   
\end{align}
are well-defined on $C_{\ydiagram{2}}^\bullet(\Wsp(c,\kk))$, i.e. over $R$. We decompose the complex into
\begin{align}\label{decomposition of the compelx}
   C_{\ydiagram{2}}^\bullet(\Wsp(c,\kk))= C_0 \underset{R}{\otimes} C_{\sll_2}^\bullet
\end{align}
where $C_0$ is the free $R$-submodule spanned by (non-commutative) differential monomials with variables
\begin{align}
    W_{2}, \tilde E_{2}, \tilde H_{3}, \tilde F_{4}, W_4, \tilde E_{4}, \tilde H_{5}, \tilde F_{6}, \dots
\end{align}
and $C_{\sll_2}^\bullet$ is the vertex subalgebra 
\begin{align}
    C_{\sll_2}^\bullet=V^\kk(\sll_2)\otimes \bigwedge{}^{\bullet}_{\varphi,\varphi^*}.
\end{align}
By construction, the differential $\fd{2}$ acts by 0 on $C_0$ and $(C_{\sll_2}^\bullet, \fd{2})$ is the BRST complex which computes the $\W$-algebra $\W^\kk(\sll_2,\oo{2})$ over $R$. In particular, \eqref{decomposition of the compelx} is a decomposition as a complex. 
One has the cohomology vanishing on the second factor:
\begin{align}
    \HH^{n}(C_{\sll_2}^\bullet)=\fHdeg{2}{n}\left(\V^\kk(\sll_2)\right)\simeq \delta_{n,0}\  \W^\kk(\sll_2,\oo{2})
\end{align}
and $\W^\kk(\sll_2,\oo{2})$ is free lover $R$, see \cite{ACL19}.
Hence, the Künneth formula shows
\begin{align}\label{cohomology}
    \fHdeg{2}{n}(\Wsp(c,\kk))\simeq \delta_{n,0} C_0\underset{R}{\otimes} \W^\kk(\sll_2,\oo{2}).
\end{align}
In particular, we have the cohomology vanishing $\fHdeg{2}{\neq0}(\Wsp(c,\kk))$ and $\fHdeg{2}{0}(\Wsp(c,\kk))$ is free over $R$.
As $\W^\kk(\sll_2,\oo{2})$ is freely strongly generated by $\tilde{L}$, $\fHdeg{2}{0}(\Wsp(c,\kk))$ is freely strongly generated by elements in \eqref{list of strong generators}. Note that \eqref{cohomology} implies that the BRST reduction commutes with the base change:
\begin{align}\label{BRST vs localization}
\fHdeg{2}{0}(\Wsp(c,\kk)_{Q(R)}) \simeq  \fHdeg{2}{0}(\Wsp(c,\kk))\otimes_R Q(R). 
\end{align}

Finally, we show the simplicity of $\fHdeg{2}{0}(\Wsp(c,\kk))$, that is, any proper ideal $\mathscr{I}=\bigoplus_{\Delta\in \Z_{\geq0}}\mathscr{I}_\Delta \subsetneq \fHdeg{2}{0}(\Wsp(c,\kk))$ graded by conformal weights satisfies
$\mathscr{I}_0=0$.
By definition, the conformal grading is compatible with the base change $\otimes_R Q(R)$,
\begin{align}
    \mathscr{I}_{Q(R)}:=\mathscr{I}\otimes_R Q(R)=\bigoplus_{\Delta\in \Z_{\geq0}}(\mathscr{I}_{\Delta}\otimes_R Q(R)) \subset \fHdeg{2}{0}(\Wsp(c,\kk))_{Q(R)}
\end{align}
is a proper ideal of $\fHdeg{2}{0}(\Wsp(c,\kk))_{Q(R)}$ and so of 
 $\fHdeg{2}{0}(\Wsp(c,\kk)_{Q(R)})$ thanks to \eqref{BRST vs localization}. 
Now, the decomposition \eqref{decomposition at generic level} implies that $\mathscr{I}_{Q(R)}$ induces a non-trivial graded ideal, say, $\widehat{\mathscr{I}}_{Q(R)}=\bigoplus_\Delta \widehat{\mathscr{I}}_{Q(R)}{}_{\Delta} $, on $\Wsp(c,\kk)_{Q(R)}$ since the BRST reduction $\fHdeg{2}{0}$ is a (fully faithful) braided tensor functor from the Kazhdan-Luztig category of $\V^\kk(\sll_2)$-modules to the category of $C_1$-cofinite $\W^\kk(\sll_2,\oo{2})$-modules:
\begin{align}\label{BTC equiv}
  \fHdeg{2}{0}\colon \mathrm{KL}_\kk(\sll_2) \overset{\simeq}{\rightarrow}  \mathrm{KL}_\kk(\sll_2,\oo{2})\subset \W^\kk(\sll_2,\oo{2})\mod  
\end{align}
by \cite{CC+21, FZ92}. 
By the simplicity of $\Wsp(c,\kk)_{Q(R)}$ \cite{CKL24}, $\widehat{\mathscr{I}}_{Q(R)}{}_0=0$ and thus $\mathscr{I}_0=0$. This completes the proof. 
\end{proof}

Theorem \ref{thm: reduction_Wsp} can be regarded as a universal version of the  Theorem \ref{thm: BCD}.
Moreover, the $\W$-algebras 
$$\W^\kk(\spp_{4n+2},\oo{2n+1,2n+1}),\quad \W^\kk(\so_{4n},\oo{2n, 2n})$$
can be obtained as quotients of 
the new $\W_\infty$-type vertex algebra $\fHdeg{2}{0}(\Wsp(c,\kk))$ over the field of rational functions $\C(\kk)$. 
\begin{corollary}\label{new universality}
    The $\W$-algebras $\W^\kk(\spp_{4n+2},\oo{2n,2n+2})$ and $\W^\kk(\so_{4n},\oo{2n-1, 2n+1})$ arise as 1-parameter quotients of $\fHdeg{2}{0}(\Wsp(c,\kk))$ over $\C(\kk)$.
\end{corollary}
\begin{proof}
Set
\begin{align}
    \W^\kk_{D,\C(\kk)}=\W^\kk(\so_{4n},\oo{2n, 2n})_{\C(\kk)},\quad \W^\kk_{C,\C(\kk)}=\W^\kk(\spp_{4n+2},\oo{2n+1, 2n+1})_{\C(\kk)}
\end{align}
for the $\W$-algebras over $\C(\kk)$.
Let us consider the localization $\W^{\mathfrak{sp}}_{\infty}(c,\kk)_{\C(\kk)}$, i.e., the base change $R\mapsto R\otimes_{\C[\kk]}\C(\kk)$.
By \cite{CKL24}, they are obtained as simple, 1-parameter quotients of $\W^{\mathfrak{sp}}_{\infty}(c,\kk)_{\C(\kk)}$, by ideals $I^D_n$ and $I^C_n$ associated with the specializations
\begin{align}
    &c_D=-\frac{\kk (2 \kk+1) (2 \kk n-\kk+2 n) (2 \kk n+\kk+4 n)}{(\kk+2) \kk (\kk+2 n)},\\
    &c_C=-\frac{\kk (2 \kk+1) (2 \kk n+2 n) (2 \kk n+2 \kk+4 n+3)}{(\kk+2) (\kk+1) (\kk+2 n+2)}.
\end{align}
Namely, we have the following short exact sequences 
\begin{align}\label{eq:sequence}
    \begin{split}
        0\rightarrow I^{X}_{n} \rightarrow \W^{\mathfrak{sp}}_{\infty}(c,\kk)_{\C(\kk)} \rightarrow \W^\kk_{X,\C(\kk)} \rightarrow 0
    \end{split}
\end{align}
for $X=C,D$. We note that the ideals $I^X_n$ are graded by conformal weight with the lowest weight conformal weight $2n+2$ and $2n+3$, respectively. Moreover, $I^X_n$ are weakly generated by a singular vector in the lowest conformal weight and contain all strong generators of larger conformal weights.

Now, we apply the BRST reduction functor $\fHdeg{2}{0}$ to \eqref{eq:sequence}.
Since $\fHdeg{2}{0}$ is exact on the (semisimple) Kazhdan-Lustzig category of $\V^\kk(\sll_2)$-modules over the field $\C(\kk)$ (see, e.g., \cite{KW04}), we obtain short exact sequences 
\begin{equation}
\begin{split}
0\rightarrow \fHdeg{2}{0}(I^{X}_{n}) \rightarrow \fHdeg{2}{0}(\W^{\mathfrak{sp}}_{\infty}(c,\kk))_{\C(\kk)} \rightarrow \fHdeg{2}{0}(\W^\kk_{X,\C(\kk)}) \rightarrow 0.
\end{split}
\end{equation}
Since Theorem \ref{thm: BCD} also holds over $\C(\kk)$ (by the same argument as in the proof), one has 
\begin{align}
\begin{split}
    &\fHdeg{2}{0}(\W^\kk_{C,\C(\kk)})\simeq \W^\kk(\spp_{4n+2},\oo{2n,2n+2})_{\C(\kk)},\\
    &\fHdeg{2}{0}(\W^\kk_{D,\C(\kk)})\simeq \W^\kk(\so_{4n},\oo{2n-1, 2n+1})_{\C(\kk)}.
\end{split}
\end{align}
Therefore, we obtain the desired conclusion. 
\end{proof}

Corollary \ref{new universality} can be summarized by the commutative diagram below  (left) where 
\begin{equation}\label{local notation}
    \begin{aligned}
\W^\kk(\g,\OO)&=\W^\kk(\spp_{4n+2},\oo{2n+1,2n+1}) \qquad &&(\text{respect. }\W^\kk(\so_{4n},\oo{2n, 2n})),\\
\W^\kk(\g,\widehat{\OO})&=\W^\kk(\spp_{4n+2},\oo{2n,2n+2}) \qquad &&(\text{respect. }\W^\kk(\so_{4n},\oo{2n-1, 2n+1})).
    \end{aligned}
\end{equation}
Together with Theorem \ref{thm:IHR}, this suggests the existence of the universal version of the inverse Hamiltonian reduction over appropriate base rings (or fields) such that the (right) diagram below commutes: 
\begin{center}
\begin{tikzcd}
 \Wsp(c,\kk) \arrow[d, twoheadrightarrow] \arrow[r,dashrightarrow, "\fHdeg{2}{0}"] & \fHdeg{2}{0}(\Wsp(c,\kk)) \arrow[d, twoheadrightarrow] &  \Wsp(c,\kk) \arrow[d, twoheadrightarrow] \arrow[r,hookrightarrow,] & \fHdeg{2}{0}(\Wsp(c,\kk))\otimes \Pi \arrow[d, twoheadrightarrow] \\
\W^\kk(\g,\OO) \arrow[r, dashrightarrow,  "\fHdeg{2}{0}"] & \W^\kk(\g,\widehat{\OO})  &\W^\kk(\g,\OO) \arrow[r,hookrightarrow] & \W^\kk(\g,\widehat{\OO}) \otimes \Pi.
\end{tikzcd}
\end{center}
Such a universal embedding exists over the field $\C(\kk)$ at least. 
\begin{theorem}\label{thm: iHR_Wsp}
There is an embedding of vertex algebras over $\C(\kk)$
\begin{align}
    \Wsp(c,\kk)_{\C(\kk)} \hookrightarrow \fHdeg{2}{0}(\Wsp(c,\kk))_{\C(\kk)}\otimes \Pi.
\end{align}
\end{theorem}
\proof
Note that the quasi-inverse of the equivalence in \eqref{BTC equiv} (over $\C(\kk)$) is given by the inverse Hamiltonian reduction
\begin{align*}
    \mathrm{iH} \colon \mathrm{KL}_\kk(\sll_2,\oo{2}) \rightarrow \mathrm{KL}_\kk(\sll_2),\quad M \mapsto (M\otimes \Pi)^{\sll_2}
\end{align*}
where $(M\otimes \Pi)^{\sll_2}\subset M\otimes \Pi$ denotes the maximal integrable $\sll_2$-submodule whose action is induced by the inverse Hamiltonian reduction for $\V^\kk(\sll_2)$ in \eqref{iHR for sl2}. Indeed, this quasi-inverse is stated in \cite{Ad19} for objects, which induces a functor as both categories are semisimple. Then, it is straightforwardly upgraded to a braided tensor functor as intertwining operators in the category $\mathrm{KL}_\kk(\sll_2,\oo{2})$ induce intertwining operators among $\W^\kk(\sll_2,\oo{2})\otimes \Pi$-modules after tensoring with the structure map of $\Pi$, which restrict to the maximal integrable $\sll_2$-submodules, i.e. $(- \otimes \Pi)^{\sll_2}$ as desired. 
One observes that the functorial isomorphisms on the spaces of intertwining operators given by $\mathrm{iH}$ and $\fHdeg{2}{0}$ differ by the structure map of $\Pi$, which does not change the braided tensor structure on the categories. 

We apply $\mathrm{iH}$ to $\fHdeg{2}{0}(\Wsp(c,\kk)_{Q(R)})$ viewed as a vertex algebra object in $\mathrm{KL}_\kk(\sll_2,\oo{2})$, which gives 
an isomorphism 
\begin{align}\label{eq:iota_iso}
\begin{split}
    \Wsp(c,\kk)_{\C(\kk)} \simeq \big(\fHdeg{2}{0}(&\Wsp(c,\kk)_{Q(R)}) \otimes \Pi \big)^{\sll_2}\\
    & \subset \fHdeg{2}{0}(\Wsp(c,\kk)_{Q(R)}) \otimes \Pi 
\end{split}
\end{align}
as $\V^\kk(\sll_2)$-modules.
By construction, the structure map $Y(\cdot,z)$ on $\fHdeg{2}{0}(\Wsp(c,\kk)_{Q(R)})$ is induced by the structure map of the BRST complex $C_{\ydiagram{2}}^\bullet(\Wsp(c,\kk))$. 
By regarding the structure map as a family of intertwining operators among $C^\kk(0)\otimes \W^\kk(\sll_2,\oo{2})$-modules through the decomposition \eqref{decomposition at generic level}, the braided tensor equivalence given by $\mathrm{iH}$ implies that the isomorphism \eqref{eq:iota_iso} preserves the structure map, and thus it is a homomorphism of vertex algebras. This completes the proof. 
\endproof

According to Conjecture \ref{iterated conjecture for type BCD} for $\W^\kk(\g,\widehat{\OO})=\W^\kk(\spp_{4n+2},\oo{2n,2n+2})$ as in \eqref{local notation}, one has 
\begin{align}
    \W^\kk(\g,\widehat{\OO})_{\C(\kk)} \simeq \HH_{\oo{2n}}^0\HH_{\oo{2n+2,1^{2n}}}^0 (\V^\kk(\g))_{\C(\kk)}.
\end{align}
This implies that there is a (conformal) embedding 
\begin{align}
   C^\kk_1(\g,\widehat{\OO})_{\C(\kk)}\underset{_{\C(\kk)}}{\otimes} C^\kk_2(\g,\widehat{\OO})_{\C(\kk)}\hookrightarrow \W^\kk(\g,\widehat{\OO})_{\C(\kk)}
\end{align}
where 
\begin{align}\label{eq:cosetC}
\begin{split}
    C^\kk_1(\g,\widehat{\OO})_{\C(\kk)}&=\W^{\kk_\sharp}(\spp_{2n},\oo{2n})_{\C(\kk)},\\
    C^\kk_2(\g,\widehat{\OO})_{\C(\kk)}&=\mathrm{Com}\left(\V^{\kk_\sharp}(\spp_{2n})_{\C(\kk)},\W^\kk(\g,\oo{2n+2,1^{2n}})_{\C(\kk)}  \right)
\end{split}
\end{align}
with $\kk_\sharp=\kk+2n+\frac{1}{2}$.
Similar considerations apply to $\W^\kk(\g,\widehat{\OO})=\W^\kk(\so_{4n},\oo{2n-1, 2n+1})$.
Since the factors \eqref{eq:cosetC} both arise as 1-parameter quotients of the even spin $\W_\infty$-algebra $\W^{\mathrm{ev}}_{\infty}(c,\lambda)$ \cite{CL21, KL19}, it is natural to expect the existence of the commutative diagram
\begin{center}
\begin{tikzcd}
\W^{\mathrm{ev}}_{\infty}(c_1,\lambda_1)\otimes \W^{\mathrm{ev}}_{\infty}(c_2,\lambda_2) \arrow[d, twoheadrightarrow] \arrow[r, hook]& \fHdeg{2}{0}(\Wsp(c,\kk)) \arrow[d, twoheadrightarrow] \\
C^\kk_1(\g,\widehat{\OO})\otimes C^\kk_2(\g,\widehat{\OO}) \arrow[r, hook]& \W^\kk(\g,\widehat{\OO}) 
\end{tikzcd}
\end{center}
over appropriate base rings (or fields).

In the proof of Theorem \ref{thm: reduction_Wsp}, we obtained the conformal vector of $\fHdeg{2}{0}(\Wsp(c,\kk))$ as the sum of two commuting Virasoro elements, namely $\tilde {\mathbf{L}}=\tilde{L}+W_2.$
The explicit description of the OPEs of $\fHdeg{2}{0}(\Wsp(c,\kk))$ for lower conformal weight generators\footnote{A Mathematica notebook presenting the OPEs can be found on the webpages of the authors.} shows that there is yet another decomposition  into two commuting Virasoro elements 
\begin{align}
 \tilde{\mathbf{L}}=\Lev_1+\Lev_2
\end{align}
where 
\begin{align}\label{new commuting virasoro}
    \Lev_1=\frac{1}{2}(\tilde{L}+W_2)+\frac{1}{2(3\kk+4)\gamma(c,\kk)}\mathscr{L}, \quad \Lev_2=\frac{1}{2}(\tilde{L}+W_2)-\frac{1}{2(3\kk+4)\gamma(c,\kk)}\mathscr{L}
\end{align}
after the extension of the base ring $\widehat{R}=R(\gamma(c,\kk))$ with 
\begin{equation}
    \gamma(c,\kk)=\sqrt{\big(c (\kk+2)-\kk(2\kk+1)\big)\big(c (\kk+2)-\kk(2\kk+1)(2\kk+3)^2\big)}.
\end{equation}
Here, $\mathscr{L}$ is given by the formula
\begin{align}
\begin{split}
    \mathscr{L}=-2(\kk+2)(3\kk+4) \tilde{E}_2
    &+\big(3\kk+4)((2\kk+1)(4\kk^2+11\kk+8)+c(\kk+2)\big)W_2\\
    &+\big(\kk(2\kk+1)(2\kk+3)(3\kk+4)-c(\kk+2)(2\kk^2+5\kk+4)\big)\tilde{L}.
\end{split}
\end{align}
Their central charges $c_1$ and $c_2$ are given by
\begin{equation}\label{eq:central_charges}
    \begin{aligned}
c_1=\frac{1}{2 (2 + \kk) \gamma(c,\kk)}
    &\left(c^2 (2 + \kk)^2+ c (2 + \kk) (12 + 40 \kk + 38 \kk^2 + 12 \kk^3 + \gamma(c,\kk))\right.\\
    &\left.\quad - (4 + 11 \kk + 6 \kk^2) (3 \kk + 8 \kk^2 + 4 \kk^3 +\gamma(c,\kk))\right)\\
c_2=\frac{1}{2 (2 + \kk) \gamma(c,\kk)}
    &\left(c^2 (2 + \kk)^2+ c (2 + \kk) (12 + 40 \kk + 38 \kk^2 + 12 \kk^3 - \gamma(c,\kk))\right.\\
    &\left.\quad- (4 + 11 \kk + 6 \kk^2) (3 \kk + 8 \kk^2 + 4 \kk^3 - \gamma(c,\kk))\right).
    \end{aligned}
\end{equation}
In addition, there exist two generators which are primary vectors of conformal weight $4$, $\gWev{1}$ and $\gWev{2}$:
\begin{equation}
    \Lev_i(z)\gWev{j}(w)\sim\indic{i,j}\left(\frac{4\gWev{i}(w)}{(z-w)^{2}}+\frac{\partial\gWev{i}(w)}{(z-w)}\right).
\end{equation}
Recalling that the universal even spin $\W_\infty$-algebra $\W^{\mathrm{ev}}_{\infty}(c,\lambda)$ introduced in \cite{KL19} is weakly generated by a conformal vector and a primary weight $4$ generator, 
the previous considerations provide strong evidence for the following conjecture suggested first in \cite{CKL24}.
\begin{conjecture}
The Virasoro-type reduction $\fHdeg{2}{0}(\Wsp(c,\kk))$ of the universal 2-parameter vertex algebra $\Wsp(c,\kk)$ is an extension of the tensor product of two even spin $\W_\infty$-algebras 
$$\W^{\mathrm{ev}}_{\infty}(c_1,\lambda_1)\underset{\widehat{R}}{\otimes}\W^{\mathrm{ev}}_{\infty}(c_2,\lambda_2) \hookrightarrow \widehat{R}\underset{R}{\otimes}\fHdeg{2}{0}(\Wsp(c,\kk)).$$
Here the parameters $c_1,c_2$ are defined in \eqref{eq:central_charges}.
\end{conjecture}

\printbibliography

\end{document}